\RequirePackage[l2tabu, orthodox]{nag}

\documentclass[letterpaper,pagesize=auto,twoside,openright,headsepline,bibliography=totoc,listof=totoc]{scrbook}
\usepackage{scrhack}

\usepackage[all,warning]{onlyamsmath}

\usepackage[pass]{geometry}

\usepackage[tracking=true]{microtype}

\usepackage{amsmath,amsthm,amssymb,amsfonts}

\usepackage{fixltx2e}

\usepackage[T1]{fontenc}
\usepackage[utf8]{inputenc}

\usepackage{varioref}

\usepackage{booktabs, longtable}

\usepackage{listings}

\usepackage[full]{textcomp}

\usepackage[osf]{mathpazo}
\KOMAoptions{DIV=calc}

\usepackage{mathtools}

\usepackage[style=alphabetic,backend=bibtex]{biblatex}
\usepackage[french,ngerman,english]{babel}
\usepackage[babel,german=swiss]{csquotes}

\usepackage{ellipsis}

\usepackage[final]{graphicx}

\pdfpageattr {/Group << /S /Transparency /I true /CS /DeviceRGB>>}
\usepackage[cmyk]{xcolor}

\usepackage{color}

\usepackage{faktor}

\usepackage{xfrac}

\usepackage{tabularx}

\theoremstyle{plain}% default
\newtheorem{theorem}{Theorem}[section]
\newtheorem*{theorem*}{Theorem}
\newtheorem{lemma}[theorem]{Lemma}
\newtheorem{proposition}[theorem]{Proposition}
\newtheorem{corollary}[theorem]{Corollary}
\newtheorem*{SMITH}{Smith Conjecture}
\newtheorem*{LOOP}{Loop Theorem}
\newtheorem*{SPHERE}{Sphere Theorem}
\newtheorem*{DEHNSLEMMA}{Dehn's Lemma}
\newtheorem*{ALEXANDER}{Alexander's Theorem}

\theoremstyle{definition}
\newtheorem{definition}{Definition}[section]

\newtheorem{example}{Example}[section]

\theoremstyle{remark}
\newtheorem*{remark}{Remark}
\newtheorem*{note}{Note}
\newtheorem*{question}{Question}

\usepackage{xifthen}    %provides \ifthenelse and \isempty
\newcommand*\foreignlang[3][]{%
	\textit{\foreignlanguage{#3}{#2}}%
	\ifthenelse{\isempty{#1}}{}{ `#1'}%
}

\newcommand*{\checkhere}{\ClassWarning{}{"Check before publishing" warning}}

\newcommand*{\R}{\mathbb{R}}
\newcommand*{\Z}{\mathbb{Z}}
\newcommand*{\N}{\mathbb{N}}
\newcommand*{\C}{\mathbb{C}}
\renewcommand*{\P}{\mathbb{P}}
\newcommand*{\Hy}{\mathbb{H}}

\usepackage{mathtools}
\DeclarePairedDelimiter\abs{\lvert}{\rvert}

\DeclareMathOperator{\tr}{tr}
\DeclareMathOperator{\im}{im}
\DeclareMathOperator{\lcm}{lcm}
\DeclareMathOperator{\id}{id}

\newcommand{\simplies}{\mkern-5mu\implies\mkern-5mu}
\newcommand{\caponept}{\mathrel{\ \mathclap{\cap}\hspace{0.3em}\mathclap{\cap}\ }}

\newcommand{\proofstep}[1]{\par\medskip\par\textit{#1}.}

\usepackage{mathabx,epsfig}

\newcommand \tat {\foreignlanguage{french}{tête-à-tête}}
\newcommand \Tat {\foreignlanguage{french}{Tête-à-tête}}

\newcommand \mcg {\mathrm{Mod}}

\newcommand \LL {\mathcal{L}}
\newcommand \NN {\mathcal{N}}
\newcommand \EE {\mathcal{E}}

\newcommand \T[1]{\mathcal{T}_{#1}} % alternatively, use \tau

\newcommand \Tr{\mathcal{T}\mkern-5.5mu r}
\newcommand \Bi{\mathcal{B}i} % alternatively, use \mathfrak{T} and \mathfrak{B}.

\newcommand \DD{\mathrm{D}^*_2} % binary dihedral group or quaternion group
\newcommand \TT{\mathrm{T^*}} % binary tetrahedral group
\newcommand \OO{\mathrm{O^*}} % binary octahedral group
\newcommand \II{\mathrm{I^*}} % binary icosahedral group

\DeclareMathOperator{\ord}{ord}

\newcommand \code[1]{\texttt{#1}}

\newenvironment{enumerate-roman}
    {%
       \begin{enumerate}%
	}
	{%
	   \end{enumerate}%
	}
	
\newenvironment{enumerate-roman-packed}
    {%
       \begin{enumerate}%
	      \setlength{\itemsep}{1pt}
	      \setlength{\parskip}{0pt}
	      \setlength{\parsep}{0pt}
	}
	{%
	   \end{enumerate}%
	}
	
\newenvironment{enumerate-packed}
    {
    \begin{enumerate}
	  \setlength{\itemsep}{1pt}
	  \setlength{\parskip}{0pt}
	  \setlength{\parsep}{0pt}
	}
	{
	\end{enumerate}
	}

\begin{document}

%\addtocontents{toc}{\vskip -5ex}
%\addtocontents{lof}{\vskip -5ex}

% The cover page
%This cover pages uses: \usepackage[pass]{geometry}

%See also: http://tex.stackexchange.com/questions/17579/how-can-i-design-a-book-cover
\newgeometry{margin=2cm}
%\begin{titlepage}  %Don't put this here - it's not meant for this and starts page numbering.
\thispagestyle{empty}

\linespread{1}
\begin{center}

	\ 
	% See http://tex.stackexchange.com/questions/25249/how-do-i-use-a-particular-font-for-a-small-section-of-text-in-my-document/25251#25251. For example, pag for "Avant Garde" is nice.
	%% Normal font selection:
	% {\fontfamily{LinuxBiolinumT-OsF}\selectfont\scshape
	%% For Arxiv:
	{\fontfamily{jkpss}\selectfont\scshape
	
		\vspace{\stretch{35}}
	%	{\Huge\bfseries \Tat\\
		{\fontsize{25}{22}\selectfont \Tat\\
		Graphs and Twists\\}
		\vspace{4ex}
		{\large PhD thesis in Mathematics\\}
		
		%\vspace{-40mm}
		\vspace{\stretch{33}}
		
		%\begin{flushright}
			\def\svgwidth{0.5\textwidth}
			%% Creator: Inkscape 0.48.3.1, www.inkscape.org
%% Accompanies image file 'img-genus3order3.pdf' (pdf, eps, ps)
\begingroup%
  \makeatletter%
  \providecommand\color[2][]{%
    \errmessage{(Inkscape) Color is used for the text in Inkscape, but the package 'color.sty' is not loaded}%
    \renewcommand\color[2][]{}%
  }%
  \providecommand\transparent[1]{%
    \errmessage{(Inkscape) Transparency is used (non-zero) for the text in Inkscape, but the package 'transparent.sty' is not loaded}%
    \renewcommand\transparent[1]{}%
  }%
  \providecommand\rotatebox[2]{#2}%
  \ifx\svgwidth\undefined%
    \setlength{\unitlength}{505.14902344bp}%
    \ifx\svgscale\undefined%
      \relax%
    \else%
      \setlength{\unitlength}{\unitlength * \real{\svgscale}}%
    \fi%
  \else%
    \setlength{\unitlength}{\svgwidth}%
  \fi%
  \global\let\svgwidth\undefined%
  \global\let\svgscale\undefined%
  \makeatother%
  \begin{picture}(1,1.04643925)%
    \put(0,0){\includegraphics[width=\unitlength]{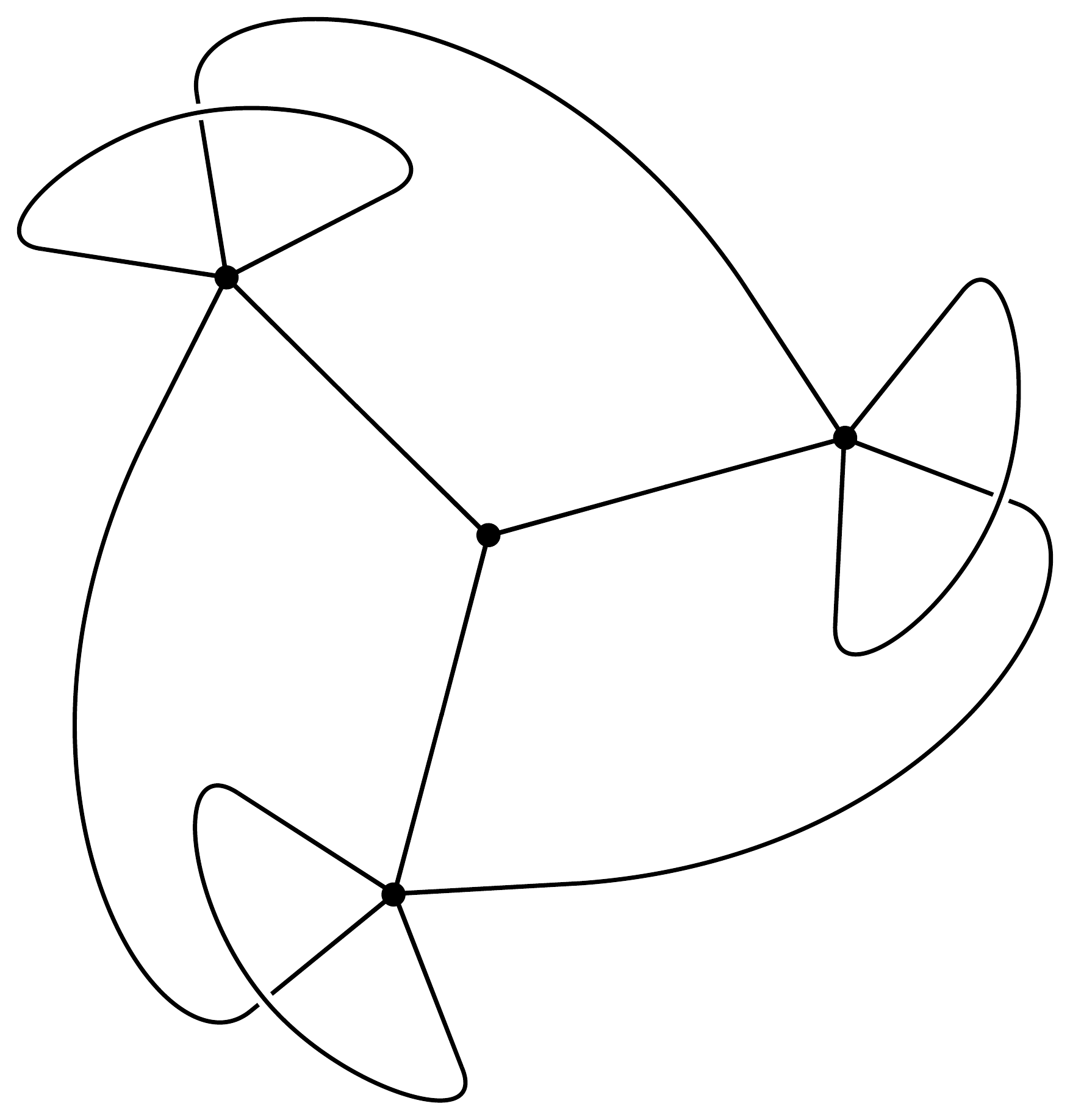}}%
  \end{picture}%
\endgroup%

		%\end{flushright}
		
		\vspace{\stretch{17}}
	
		{\Large Christian Graf\\}
		\vspace{1ex}		
		{\large University of Basel, 2014}
	
		\vspace{\stretch{10}}
	}

\end{center}
\restoregeometry

\newpage

\newpage\thispagestyle{empty}

\frontmatter

%%The "half title" or "bastard title"
% \begin{titlepage}
% 	\begin{center}
% 	    \vspace*{\stretch{1}}		
% 		{\large%\bfseries
% 		\scshape
% 			\Tat\\
% 			Graphs and Twists\\ % Don't omit the newline here, otherwise
% 								% TeX will get the centering wrong.
% 		}
% 		\vspace*{\stretch{2}}	
% 	\end{center}
% 	\newpage\thispagestyle{empty}	
% \end{titlepage}

% The front matter - half-title, title, table of contents, preface.

%The title page, according to
%http://philnat.unibas.ch/fileadmin/uploads/dokumente/Pflichtexemplare-2012.pdf
%See also: http://mirror.switch.ch/ftp/mirror/tex/info/memdesign/memdesign.pdf
\begin{titlepage}
	\begin{flushleft}
	    \vspace*{\stretch{1}}
		{\Large\bfseries
			\Tat\\
			Graphs and Twists\\
		}
		%\vspace{0.2\textheight}
		\vspace*{\stretch{3}}
	    {\bfseries
	   	    Inauguraldissertation\\
	   	}
	   	%% Coffee stain for draft versions
	   	%\cofeAm{0.9}{0.8}{180}{-100}{-100}	    
	    \vspace*{\stretch{1}}
	    zur\\
	    Erlangung der Würde eines Doktors der Philosophie\\
        \vspace*{\stretch{0.5}}
	    vorgelegt der\\
	    Philosophisch-Naturwissenschaftlichen Fakultät\\
	    der Universität Basel\\
		\vspace*{\stretch{2}}
	    von\\
		Christian Graf\\
		aus\\
		Laufenburg AG\\
		\vspace*{\stretch{2}}
		Basel, 2014
	\end{flushleft}
\end{titlepage}

%The back of the title page
{
	\newpage
	\thispagestyle{empty}	
	\raggedright
	\vspace*{\stretch{4}}
	Genehmigt von der Philosophisch-Naturwissenschaftlichen Fakultät\\
    auf Antrag von\\
    \ \\
    Prof. Dr. Norbert A'Campo, Universität Basel\\
    Prof. Dr. Michel Boileau, Université d'Aix-Marseille\\
    Prof. Dr. Alexandru Oancea, Université Pierre et Marie Curie 
    % Hinweis: Der/die Vorsitzende und allfällige externe Experten gehören nicht dazu – bitte nicht erwähnen.
    % Siehe http://www.unibas.ch/doc/doc_download.cfm?uuid=D030CCF29704A8CB76F49E238D348C47
    
    \vspace*{\stretch{1}}
    Basel, den 25. März 2014
    % (Datum der Genehmigung durch die Fakultät)
    % Hinweis: Dies ist das Datum der Fakultätssitzung – NICHT das Prüfungsdatum
    
    \vspace*{\stretch{1}}    
    \hspace{0.5\textwidth}Prof. Dr. Jörg Schibler\\
    \hspace{0.5\textwidth}Dekan\\
    \vspace*{\stretch{4}}  
}

%The table of contents
\microtypesetup{protrusion=false} % disables protrusion locally in the document
\tableofcontents % prints Table of Contents
\listoffigures
\microtypesetup{protrusion=true} % enables protrusion
\newpage

\chapter{Summary}
This is a thesis in the field of low-dimensional topology, more specifically about the mapping class group, knots and links, and 3-manifolds.

For the most part, we will define and examine so-called
\emph{\tat\ twists}, a rich and well-structured collection
of elements of the mapping class group that are described
by \emph{\tat\ graphs}.
Whereas Dehn twists are twists around a simple closed curve, \tat\ twists are twists around a graph.
We will see how to describe mapping classes of finite order, or periodic pieces of mapping classes, by \tat\ twists.

Another main result is a new criterion to decide whether a Seifert surface of a fibred knot or link is a fibre surface.

\section*{Organization of the text}
First, Chapter \ref{chap:introduction} will introduce \tat\ graphs and twists
and give some examples.

Chapter \ref{chap:properties} then establishes some basic results about those objects.
We will see a notation, using chord diagrams, that helps to classify and study
them.
And we will study \emph{elementary twists}, which can be seen as building
blocks of more complicated \tat\ twists.

In Chapter \ref{chap:periodic}, an interesting result is proven:
\Tat\ twists describe precisely the (freely) periodic diffeomorphism
classes of surfaces with boundary or punctures.
From this fact, we can deduce combinatorially some properties about the orders of such maps.
We will also see that another characterization of periodic maps is the existence of an invariant
spine of the surface.

The next chapter, Chapter \ref{chap:monodromies}, is concerned
with \tat\ twists as monodromies of knots and links in the
3-sphere, and more generally with open books of \tat\ twists.

Chapter \ref{chap:elastic} is rather independent from the others.
It provides a simple characterization of fibre surfaces, which can
be used to justify the examples given in the introduction, as
well as to give easy new proofs of statements about fibre surfaces.

Chapter \ref{chap:mcg} treats \tat\ twists in the context of
the mapping class group. Among other things, we will see how
to use them to generate the mapping class group, as well as some
statements about roots of mapping classes.

Finally, Chapter \ref{chap:software} describes a software that
I have used to do some experiments with \tat\ graphs and twists.

\chapter{Acknowledgments}
Of the people with a direct connection to my work, I first of all
want to thank my advisor, Norbert A'Campo.
His vast knowledge, his great willingness to share it, and his
inspiring way of teaching and thinking about mathematics have
helped and influenced me enormously.
He set me on track, and to him I owe a great part of my mathematical
knowledge.

I thank the reviewers of this thesis, Alexandru Oancea
and Michel Boileau, for their reports and for participating
in my thesis defense.
Thanks go also to David Masser who was chair of the defense and gave comments on the text.

This work has been partially supported by the Swiss National Science Foundation, in a joint project of Anna Beliakova and Norbert A'Campo,
and I thank Anna for all the administrative work, apart from providing a welcoming
environment in the University of Zurich.
Thanks to Alessandra Iozzi and Marc Burger, I could work at ETH Zurich for
the most part of my time as a PhD student,
an extremely valuable support for which I am very grateful.
I thank the Mathematical Institute of the University of Basel where I have been employed as an assistant for two semesters
and had a very good office during the entire time of writing of this thesis.
In 2013, eventually, I spent two months at the University of Bern thanks to the invitation of Sebastian Baader.

Special thanks go to him for his immense support and his collaboration,
and in particular for a lot of beneficial and encouraging discussions about
mathematics and life and all.

I cannot thank enough my parents for their unconditional love and support.
And I thank Maria Fernanda for having decided to share her life with me.

There are many to whom I am grateful for their help,
for their friendship, for fruitful discussions,
for teaching me mathematics, for organizing conferences,
or just for making life enjoyable.
Among them are people with whom I discussed parts of
my thesis like Jürg Portmann, Tamara Widmer, Jonas Budmiger,
Maria Hempel, Thomas Huber, Filip Misev, Immanuel Stampfli,
and others who accompanied me (mathematically and not)
on some part of the way.
To give some names is really quite unfair because there are many more that have been
important to me:
My office mates from ETH's \emph{G}-, \emph{J}-, and again \emph{G}-floor;
all the company I had in and around ETH, among other things for many lunches, plus some relay races;
my friends at the University of Zurich;
the low-dimensional topology group from Bern;
everyone, including infrequent guests, of the \foreignlang[]{Geometriekaffee}{german} (or \foreignlang[]{-tee}{german} as some might call it);
and of course my colleagues and friends from Basel's Mathematical Institute,
for their company and their kindness from the start until the very day of my PhD defense and beyond.
Thank you for making these places so interesting and friendly for working and living.

\mainmatter

%==================================================
\chapter{Introduction}\label{chap:introduction}

%==================================================
\section{\Tat\ graphs}
It is possible to see a \tat\ graph in the real world.
Imagine two strangers standing on the pavement on two sides of a street with bustling morning
traffic. %\footnote{Or the sidewalk.}
Their eyes meet, but then they both walk on, continuing towards
their right, safe from the cars.
Upon having walked for two hundred metres they once more look across the street
and, to their surprise, find themselves meeting again.
And not by coincidence: It would be bewildering for them to discover that the
same thing would have happened no matter where they started.

We will examine networks of streets with this property; networks which,
like real streets, can include some over- and underpasses.
Mathematically speaking we do the following: Take at a metric graph $G$ embedded in an
oriented surface $\Sigma$ that deformation retracts to $G$. Measure walking
distance by using the retraction $\rho\colon\Sigma\to G$ to pull back the metric
of the graph, which means that only movement in the direction of the edges is
taken into account. If the reunion of the two strangers described above takes place for some fixed
walking distance $l$ and for every starting position on the boundary of
$\Sigma$, $G$ is said to have the \emph{\tat\ property} with \emph{walk length}
$l$. The strangers do not turn back, meaning that their path is required to be
transverse to the fibres of the retraction $\rho$. Fortunately, when they see
each other again, the traffic lessens for a moment; so they get to meet, albeit
in the middle of the road. We call such a path, starting on $\partial\Sigma$,
continuing in $\Sigma\smallsetminus G$, and ending on $G$, a \emph{safe walk} of
length $l$. If $l$ is a negative number, the safe walks are understood to lead
to the left instead of to the right.

%==================================================
\section{\Tat\ twists}
\begin{figure}
\centering
\def\svgwidth{0.35\textwidth}
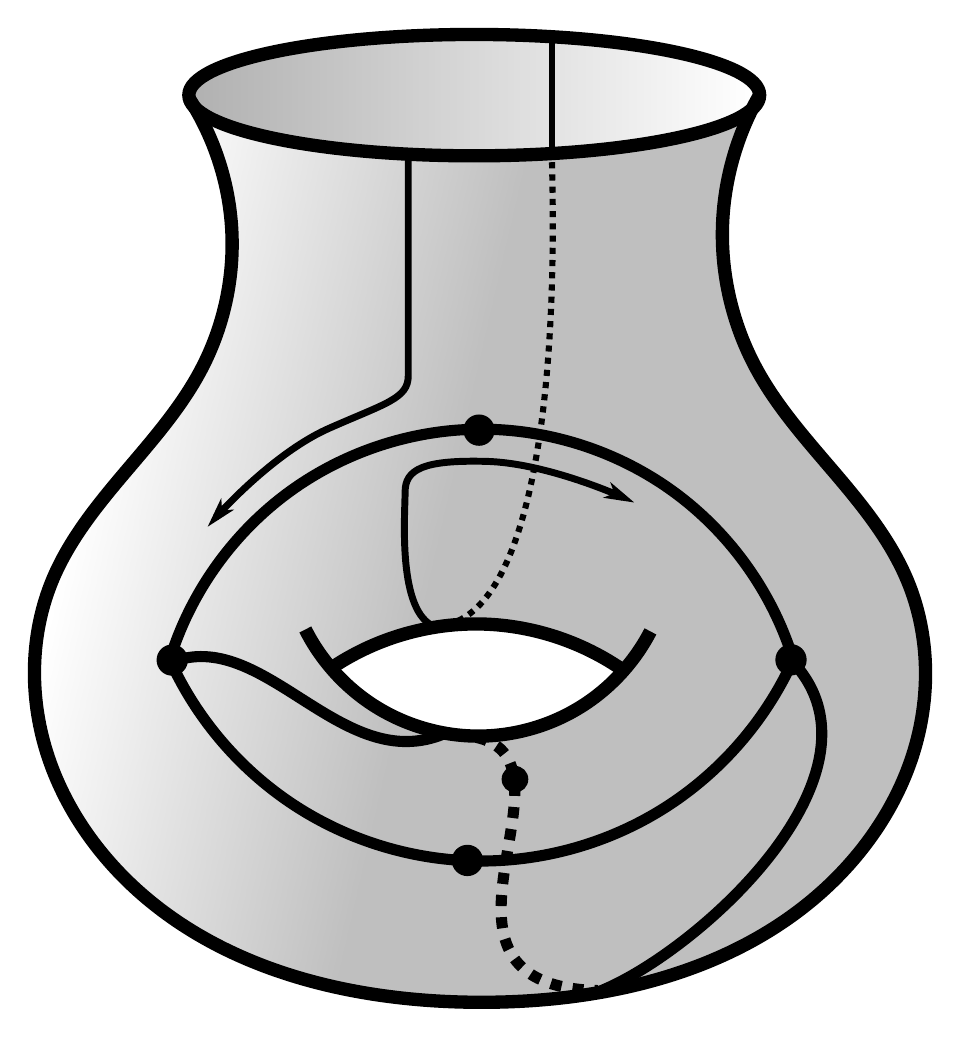
  \caption{\Tat\ graph on a one-holed torus}
  \label{fig:thetatorus}
\end{figure}
Norbert A'Campo, who coined these terms, defined in this way a natural
generalization of Dehn twists (\cite{ACampo2009}).
The cylinder is replaced by an arbitrary
surface with boundary $\Sigma$, and the simple closed curve along one twists by
an embedded graph $G$ as above.

%With latin modern: $\mathrm{\Theta}$
One example is the $\Theta$-graph in Figure \ref{fig:thetatorus}, which
is a deformation retract of the one-holed torus; see also Figure
\vref{fig:E-3-3} for two alternative views.  We parameterize all edges such
that they have length one (and will always do so in this text, unless stated
otherwise). Then one can check that this graph does indeed have the \tat\
property with walk length $2$.

An even simpler example is, of course, a circle consisting of two (unit-length)
edges. Here we have the \tat\ property with walk length $1$. This graph will
give us back the standard Dehn twist, and we define a diffeomorphism of $\Sigma$
accordingly:

\begin{figure}
\centering
\def\svgwidth{0.2\textwidth}
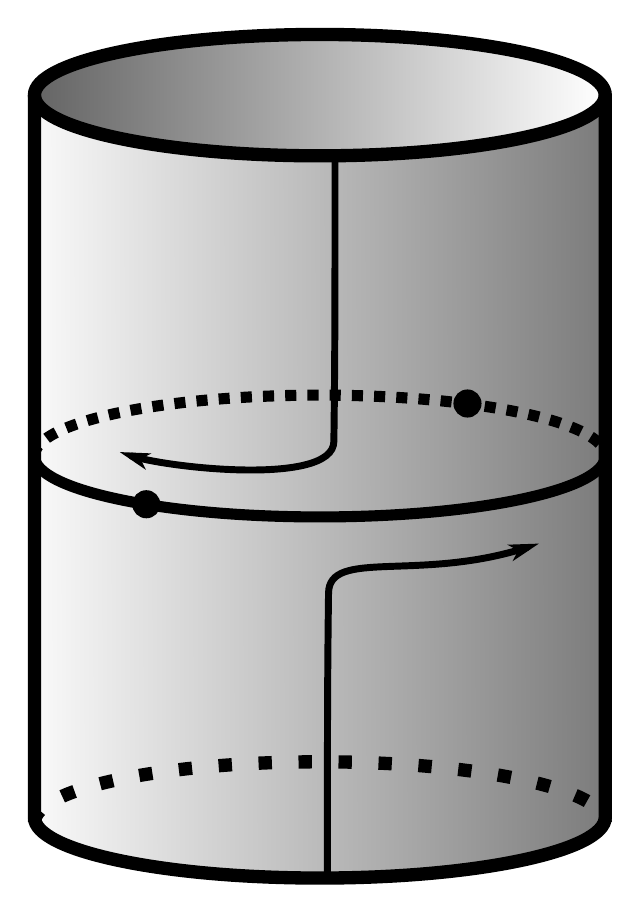
  \caption{\Tat\ graph for a Dehn twist}
  \label{fig:dehntwistgraph}
\end{figure}
Choose one properly embedded arc for each edge of $G$, in such a way that the
deformation retraction contracts it to a single point on the edge, where the arc
meets $G$ transversely. We call such an arc a \emph{crossing arc}. The \tat\ twist $\T{G,l}$ (or simply $\T{G}$)
then maps the two halves of the crossing arc to safe walks of length $l$ along the
graph. The union of all the transverse arcs cuts $\Sigma$ into a collection of
disks, so there is a unique way, up to isotopy, to complete $\T{G,l}$ to a
diffeomorphism of $\Sigma$.
%==================================================
\subsection{Naming matters and conventions}

\begin{figure}
	\begin{minipage}[b]{0.4\textwidth}
		\centering
		\def\svgwidth{0.5\textwidth}
		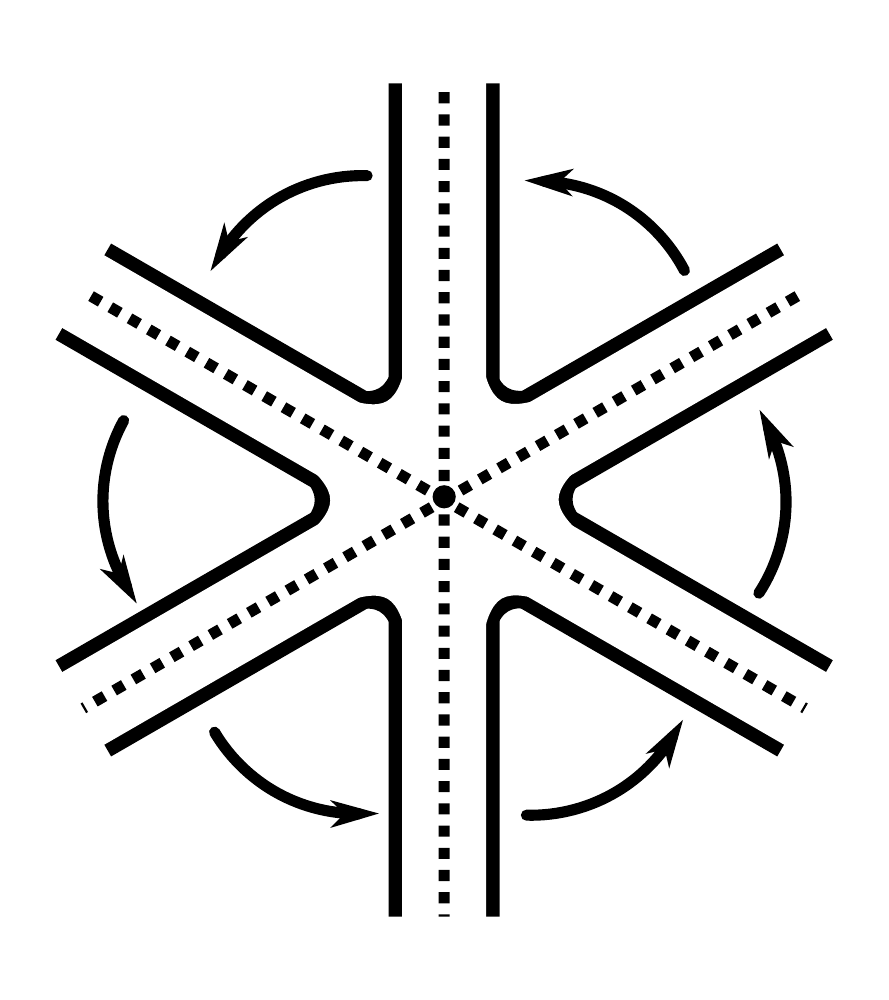
		\caption{Cyclic ordering}
		\label{fig:cyclicordering}
	\end{minipage}
	\hspace{0.5cm}
	\begin{minipage}[b]{0.5\textwidth}
		\centering
		\def\svgwidth{0.9\textwidth}
		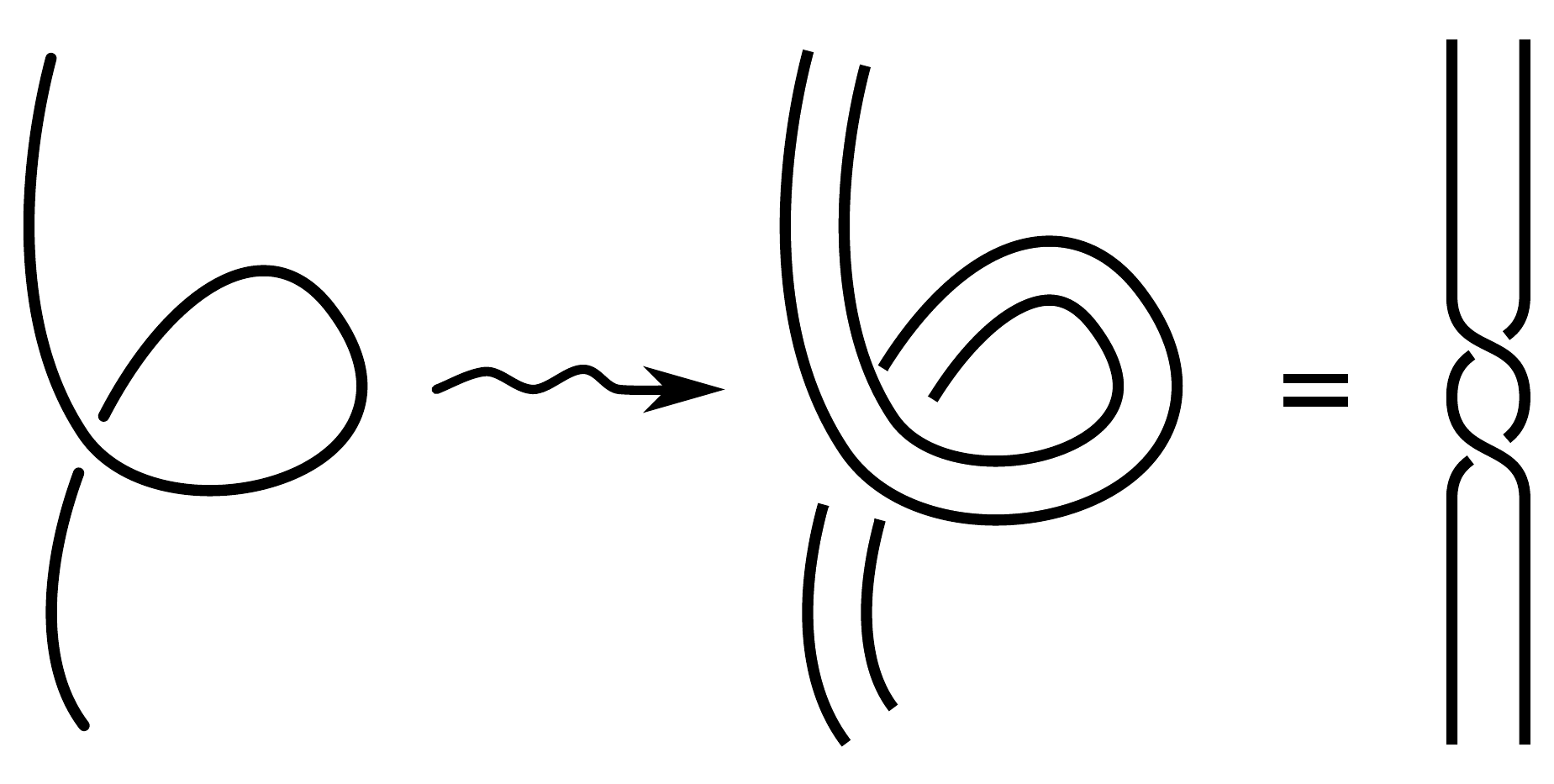
		\caption{Blackboard framing}
		\label{fig:blackboardframing}
	\end{minipage}
\end{figure}

A graph together with an oriented surface that deformation retracts to it is
often called \emph{ribbon graph}, \emph{fatgraph}, or also \emph{spine (of the
surface)}. An alternative description would be an
abstract graph and, additionally, for each vertex a cyclic ordering of the edges
adjacent to it, as illustrated in Figure \ref{fig:cyclicordering}.

\Tat\ graphs are ribbon graphs and thus come automatically equipped with a surface.
Therefore, notation like $\partial G$, when $G$ is a \tat\ graph, means ``the boundary
of the surface which gives $G$ its ribbon structure''. When a safe walk of length $l$
makes an entire turn around one chosen boundary component, we call this $l$ the
\emph{length of this boundary component}.

When the graph is obvious from the picture of the surface, it will often be
omitted. Whenever, on the other hand, the surface is omitted from the drawing,
\emph{blackboard framing} is used: The graph should be thickened inside the plane
of the paper, or blackboard, with the obvious modifications at crossings (see Figure
\ref{fig:blackboardframing}). Over- and undercrossings need not be distinguished and are
not always drawn.
\begin{figure}
\centering
\def\svgwidth{0.7\textwidth}
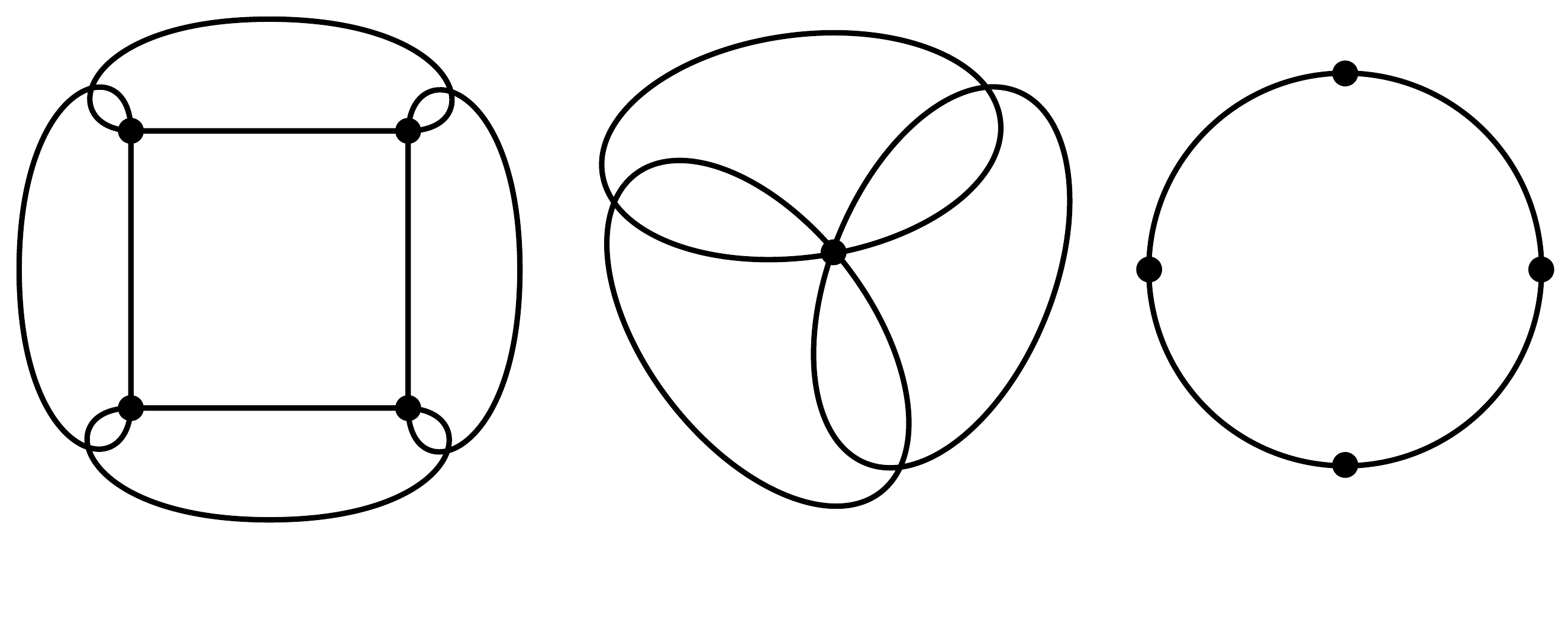
  \caption{Three more \tat\ graphs, with walk lengths $2$, $1$, and $2$.}
  \label{fig:moretat}
\end{figure}

%==================================================
\section{Torus link monodromies}\label{sec:introtoruslinks}
A very nice application of \tat\ twist, suggested by A'Campo, is a description
of the monodromy of torus knots and links. The theory of fibred links and
monodromies will be outlined later in Chapter \ref{chap:monodromies}.
The monodromy of a $(p,q)$-torus link is a mapping class $\phi_{p,q}$ which is
defined on a surface with $d = \gcd(p,q)$ boundary components and genus $g =
\frac{1}{2}((p-1)(q-1)-d+1).$
Its order, up to Dehn twists along the boundary, is $pq$.
$\phi_{p,q}$ can be described using the fact that a $(p,q)$-torus link is the
link of the singularity $x^p+y^q$ in $\C^2$.
\begin{figure}
\centering
\def\svgwidth{1\textwidth}
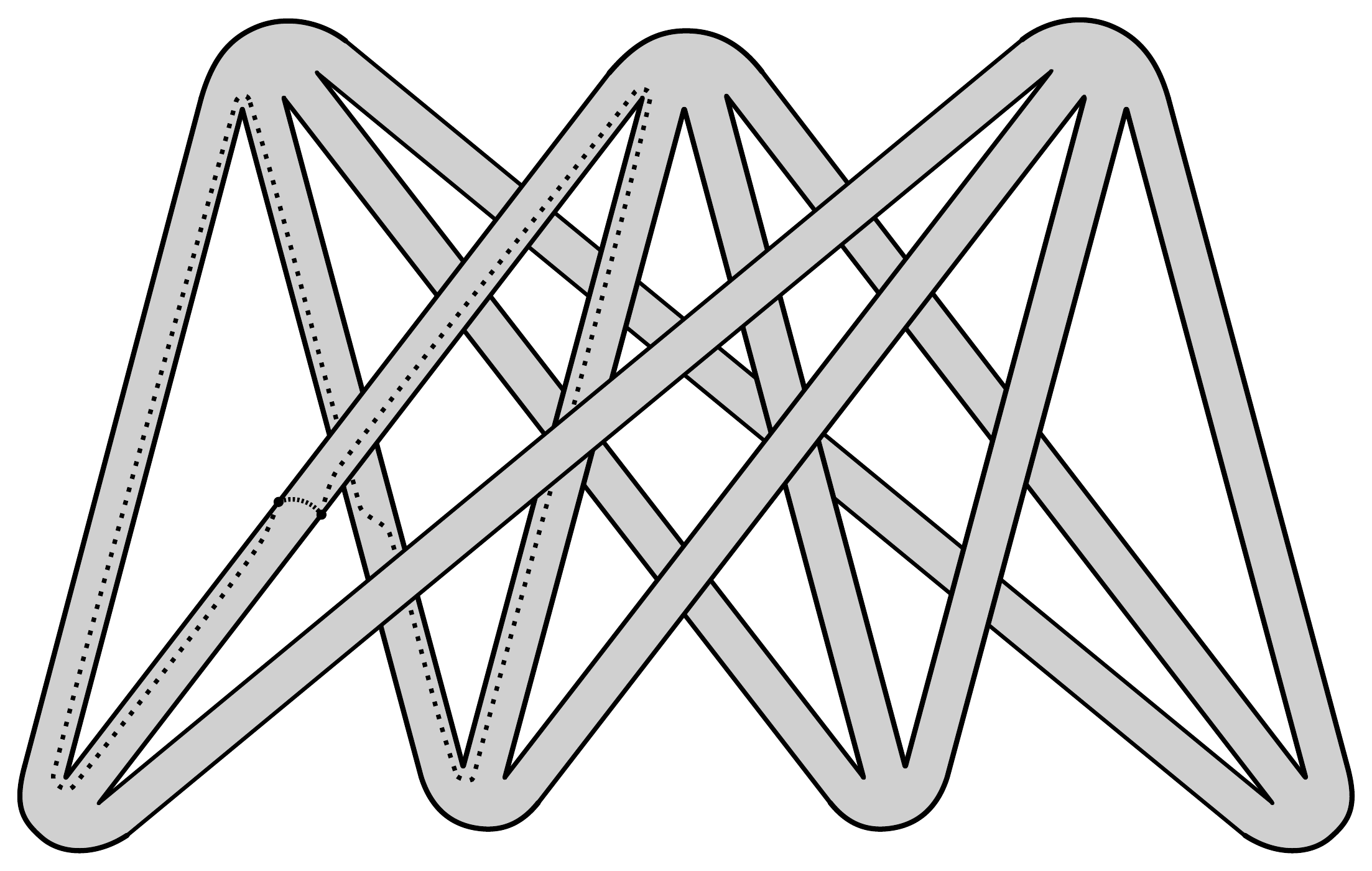
  \caption[\Tat\ twist for the $(3,4)$-torus knot]%
  {\Tat\ twist for the $(3,4)$-torus knot. The actual graph has
  been omitted as it is clear from the picture. One transverse arc is
  shown together with its image, which is composed of two safe walks of length two.}
  \label{fig:bipartite2}
\end{figure}

But \tat\ twists make the map much more explicit: $\phi_{p,q}$ is a \tat\ twists
along a complete bipartite graph $B_{p,q}$ with $p+q$ vertices.
These graphs have the \tat\ property for walk length $2$, and $\phi_{p,q} =
\T{B_{p,q},2}$.
As an example, see Figure \vref{fig:bipartite2}, where $(p,q) = (3,4)$.
Using general properties of \tat\ twist, described in the next chapter, we see
for example that $\phi_{p,q}$ permutes the $pq$ edges of the graph cyclically,  and
individually it permutes cyclically the $p$ vertices above and the $q$ vertices
below.

It was noted by Sebastian Baader that the particular embedding of the ribbon
graph that is chosen in the picture -- edges are stacked vertically according to
the number of their bottom (or top) endpoint -- actually makes it into a Seifert
surface for the link (\cite{Baader2011}); if one looks carefully, its
boundary unveils itself as the $(3,4)$-torus knot. This fact will be used in Chapter
\ref{chap:elastic} to prove that the \tat\ description is indeed correct.

%==================================================
\chapter{Properties and classification of \tat\ twists}\label{chap:properties}
Directly from the definition of \tat\ twists, it may seem mysterious which
graphs inside which surfaces could have the \tat\ property. But we can establish
properties for those twists that allow for a better understanding and also for a
systematic approach to listing and examining \tat\ graphs.

Later in this chapter, we will restrict ourselves to \tat\ graphs with one
boundary component and describe a notation for them, but generalizations to an
arbitrary number of boundary components are possible and often straightforward.

%==================================================
\section{Basic properties of \tat\ diffeomorphisms}
The following proposition establishes some properties of \tat\ twists that help
us imagine what they do.
\begin{proposition}\label{prop:basic}
A \tat\ twist $\T{G,l}$ can be represented by a diffeomorphism (which we also write
as $\T{G,l}$) such that
	\begin{enumerate-roman-packed}
		\item\label{itm:invariant}	$\T{G,l}(G)=G$,
		\item\label{itm:powers}		$\T{G,l}^n = \T{G,nl}^{}$ $(n\in\Z)$,
		\item\label{itm:finite}		$\T{G,l}$ is of finite order outside a
									tubular neighbourhood of the boundary of $G$'s
									surface.
	\end{enumerate-roman-packed}
\end{proposition}
\begin{proof}
\begin{figure}
\centering
\def\svgwidth{0.3\textwidth}
{\fontfamily{pplx}\selectfont
	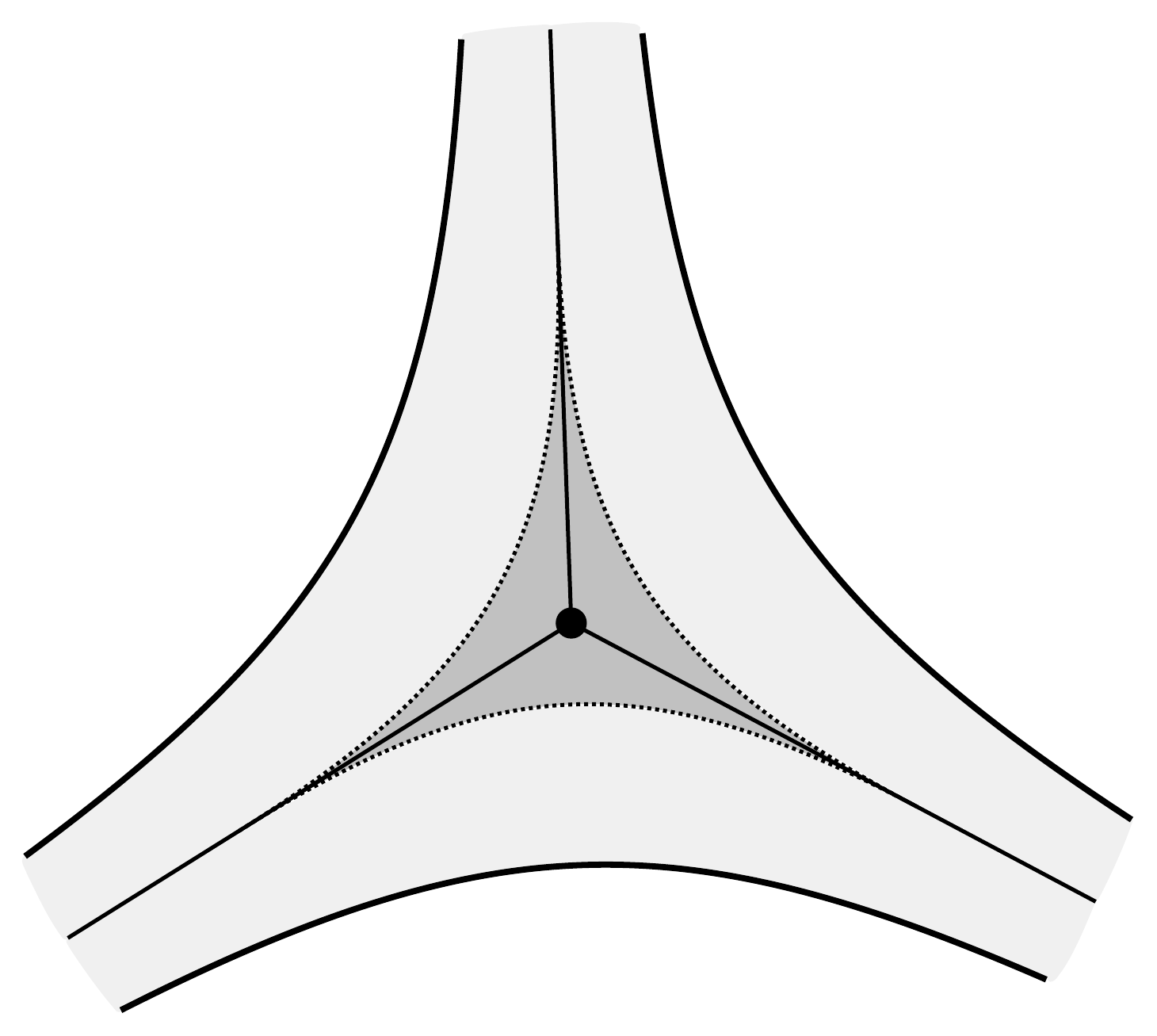
}
  \caption{Smoothing around the vertices}
  \label{fig:smoothvertex}
\end{figure}

We can see a \tat\ twist in a more explicit way, similar to a Dehn twist. To do
this, choose a Riemannian metric on the surface $\Sigma$ such that all edges of
$G$ are unit length geodesics. Around each vertex, choose a small rotationally
symmetric polygon whose vertices lie on $G$ with their adjacent edges tangent to
$G$; see Figure \ref{fig:smoothvertex}. Let $\bar{G}$ be $G$ together with these
polygons. $\Sigma\smallsetminus\bar{G}$ is a collection of annuli. Choose the
deformation retraction $\rho$ of $\Sigma$ to $G$ in such a way that it gives us
$\Sigma$ as a tubular neighbourhood of $\bar{G}$ and decomposes each component
of $\Sigma\smallsetminus\bar{G}$ as a product $S^1 \times [0,1]$, such that $S^1
\times \{0\}$ is a boundary component of $\Sigma$.
Assume that $S^1$ is parameterized as $[0,b]/{\{0,b\}}$, where $b$ is the length
of the respective boundary component, or of the cycle in $G$ around it, and
$\rho(\{m\}\times [0,1])$ is a vertex for every $m \in \N$, $0 \leq m \leq b$.

The \tat\ twist $\T{G,l}$ can now be realized as
\[
	(\theta,t) \mapsto (\theta + l\cdot h(t),t),
\]
where $h\colon [0,1] \to [0,1]$ is a smooth function which is zero on
$[0,\sfrac{1}{3}]$ and one on $[\sfrac{2}{3},1]$. At the same time, the polygons
are exchanged and/or rotated appropriately.

In this description, the first two statements of the proposition are obvious.
And when $b_i$, $1 \leq i \leq r$, is the length of the $i$th boundary
component, put $n = \frac{1}{l} \cdot
\lcm(b_1,\ldots,b_r,l)$.
Then $\T{G,l}^n$ consists of (possibly multiple) Dehn twists around the boundary
components of $\Sigma$.
\end{proof}
Trying to understand which graphs have the \tat\ property, one should note the
trivial cases:
\begin{remark}
Every ribbon graph has the \tat\ property at least for all multiples of
$l = \lcm(b_1,\ldots,b_n)$, the least common multiple of the lengths of all boundary components. The corresponding twists are compositions of Dehn
twists along the boundary.
\end{remark}

%==================================================
\subsection{Justifying the definition}
The simplicity of the definition of \tat\ twists that was given above lends
itself to two obvious generalizations regarding the walk length.

First, the original definition used by A'Campo assigns to each edge of the graph a
positive real length and chooses a uniform walk length of $\pi$.
Choosing $\pi$ is no restriction since we can rescale.
For the moment, call these graphs \emph{\tat\ graphs with real edge lengths}.
This definition is more general, but as it will turn out, produces the same isotopy
classes of \tat\ twists.

Second, one could also specify different walk lengths for safe walks starting at
different boundary components of the \tat\ graph. Call these graphs \emph{multi-speed \tat\ graphs}. When $G$ has
$r$ boundary components, we write them as $\T{G,\,\underline{l}} = \T{G,(l_1,\ldots,l_r)}$.
This definition is indeed more general: By assigning positive numbers to some
boundary components and negative ones to others, it allows walks in different
directions. The freedom of direction, however, is all that is generalized, as the following theorem shows.
\begin{remark}\label{rmk:nonzero}
There is a very special case which we treat first: If some $l_i$ is zero, then
all edges of $G$ adjacent to the $i$\/th boundary component are fixed pointwise.
Therefore when, say, the $j$th boundary component lies on the other side of such
an edge, $l_j$ must be a multiple of $b_j$. The same goes for all other boundary
components, provided $G$ is connected. In this case, $\phi$ is a composition of some
Dehn twists around boundary components.
Therefore, the $l_i$ are assumed to be nonzero in what follows.
\end{remark}

By ``\tat\ twist'', we will always mean the definition given in the introduction,
using a single walk length measured in number of edges. Multi-speed twists come in
handy in various situations, and their slightly greater generality will be used in
Chapter \ref{chap:periodic}.

The definitions are related in the following way:

\begin{theorem}\label{thm:definitions}
Let $\phi$ be a mapping class. The following are equivalent:
\begin{enumerate-roman-packed}
\item\label{itm:normal}    $\phi$ is a \tat\ twist,
\item\label{itm:real}      $\phi$ is a twist along a \tat\ graph with real edge lengths,
\item\label{itm:multiple}  $\phi$ is a twist along a multi-speed \tat\ graph, and all walk lengths have the same signs.
\end{enumerate-roman-packed}
Provided that $G$ is connected and not a circle, it can be chosen 
without bivalent vertices in case \ref{itm:real}, and without uni-
nor bivalent vertices in case \ref{itm:multiple}.
\end{theorem}
\begin{proof}
We show \ref{itm:real} $\simplies$ \ref{itm:multiple} and \ref{itm:multiple}
$\simplies$ \ref{itm:normal}; that \ref{itm:normal} implies the other two is
trivial. Then we show how to get rid of uni- and bivalent vertices in cases
\ref{itm:real} and \ref{itm:multiple}.
\proofstep{\ref{itm:real} $\simplies$ \ref{itm:multiple}}
Pick one boundary component and look at the cycle in $G$ that surrounds it.
$\phi$ induces a symmetry of that cycle, sending vertices to vertices. Starting
at any one vertex and counting the number of edges that are passed during
$\phi$'s safe walk, we get, for that boundary component, the correct walk length
for a multi-speed \tat\ twist (with unit-length edges). The signs of the walk
lengths are the same by assumption.
\proofstep{\ref{itm:multiple} $\simplies$ \ref{itm:normal}}
Let $G$ have $r$ boundary components of length $b_1,\ldots,b_r$ respectively,
and let $\phi$ be given as $\T{G,(l_1,\ldots,l_r)}$.

\begin{figure}
\centering
\def\svgwidth{0.5\textwidth}
{\fontfamily{pplx}\selectfont
	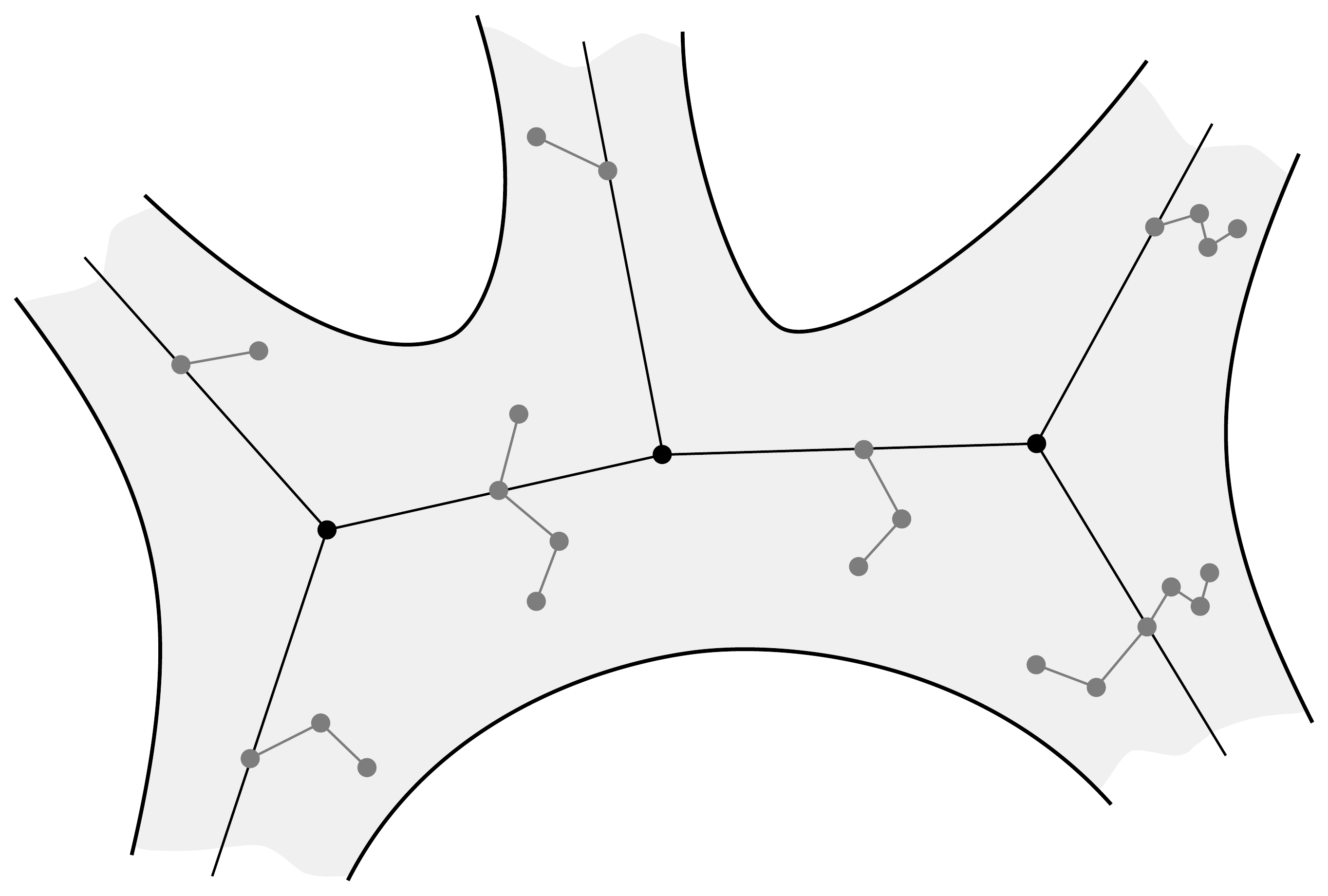
}
  \caption{A graph sprouting new twigs}
  \label{fig:twigs}
\end{figure}

Assume that no $l_i$ is zero; see the remark \vpageref{rmk:nonzero}. By assumption, all $l_i$ have the same
sign; assume that they are all positive (or change the orientation of the surface
to achieve this). We will see how one can, at
the cost of introducing uni- and bivalent vertices, modify the graph to give
back the same mapping class with one single walk length.

First, subdivide each edge once and
replace each $l_i$ by $2l_i$. The induced \tat\ twist remains the same. Each newly
introduced bivalent vertex has two nearby boundary components. Towards both of
them, we add a small linear graph which, if it goes towards the $i$th boundary
component, has length
\[
	s_i = \frac{1}{2l_i}\lcm(l_1,\ldots,l_r,2) - 1.
\]
$s_i$ can be zero, which means there is nothing to add. To walk along one edge in the original graph is the same as walking for distance
$2+2s_i$ in the changed graph. To describe the same \tat\ twist as before, replace
therefore the walk length $2l_i$ by
\[
	l = (2+2s_i)l_i = \lcm(l_1,\ldots,l_r,2),
\]
which is the uniform walk length we were looking for.

\proofstep{Removing uni- and bivalent vertices}
On a \tat\ graph with real edge lengths, a bivalent vertex can easily be removed
by giving the new combined edge a length which is the sum of the two pieces; that is,
provided the graph does not just consist of a single loop. In many cases, suitable
rescaling of edges may also eliminate the need of univalent vertices.

A multi-speed \tat\ graph $G$, if it is connected and not a circle, needs
neither uni- nor bivalent vertices.
Choose one boundary component and start a safe walk of the specified walk
length, say $l_i$, at a vertex which is not bivalent.
Let $w_i$ be the number of bivalent vertices passed by this safe walk.
$l'_i = l_i - w_i$ is the new walk length to be used at this boundary component
after all bivalent vertices have been removed from $G$.
Univalent vertices can be removed similarly.
\end{proof}
For a (standard) \tat\ twist, uni- and bivalent vertices may be necessary; see
the example \vpageref{fig:bivalentneeded}.

%==================================================
\subsection{Bounds for general \tat\ graphs and twists}\label{sec:generalbounds}
The Euler characteristic of a \tat\ graph with $b$ boundary components and genus $g$, $v$ vertices and $e$ edges is
	\[
		v-e = \chi(\Sigma) = 2-2g-b,
	\]
hence
	\begin{equation}\label{eq:euler}
		g = 1 + \frac{e-v-b}{2}.
	\end{equation}

Assume now that the graph has neither uni- nor bivalent vertices. As we have
seen, this can be achieved by permitting different walk lengths for different
components of the boundary, if there is more than one. Then
	\[
		 v \leq \frac{2}{3} e
	\]
and therefore, since $b$ is at least $1$,
    \begin{equation}
    	e \leq 6g-3.
    \end{equation}
\begin{figure}
\centering
\def\svgwidth{0.4\textwidth}
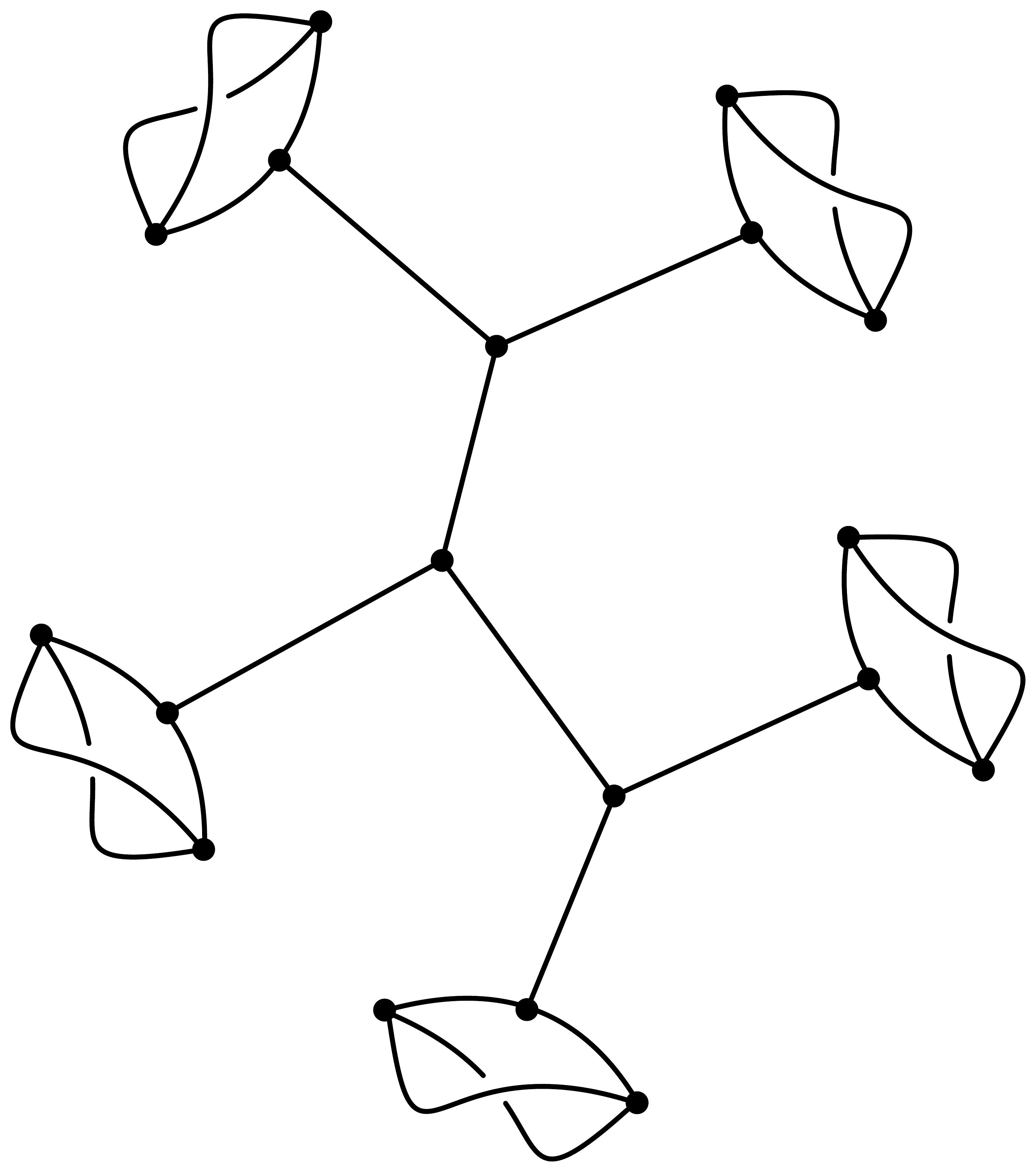
  \caption{Trivalent ribbon graph of genus $5$}
  \label{fig:bigtrivalent}
\end{figure}
Any graph with at least trivalent vertices can be made trivalent by inserting
new edges; then the inequality becomes an equality. Figure
\ref{fig:bigtrivalent} shows an example.

\begin{figure}
\centering
\def\svgwidth{0.7\textwidth}
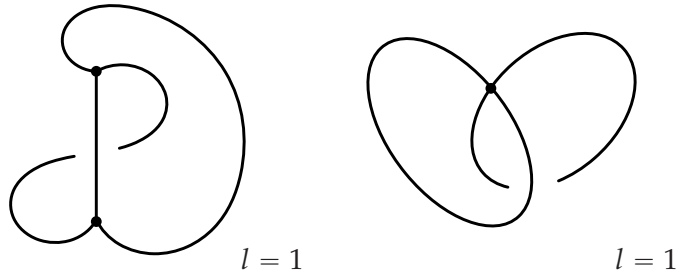
  \caption[The only \tat\ graphs of genus $1$]{The only \tat\ graphs of genus $1$: The graphs for the trefoil and the bifoil twist}
  \label{fig:bifoiltrefoil}
\end{figure}
\begin{remark}\label{rem:trefoilbifoil}
This bound shows that there is a finite number of \tat\ graphs of a given genus.
Hence, on each fixed surface there is only a finite number of \tat\ twists that
are not powers of others, up to conjugacy. For example, there are
only two \tat\ twists on the one-holed torus, and powers of them: The
\emph{trefoil twist} $\Tr$ -- the monodromy of the trefoil described above -- and the
\emph{bifoil twist} $\Bi$, depicted in Figure \ref{fig:bifoiltrefoil}.
\end{remark}

%==================================================
\section{Chord diagram notation}
To be able to systematically examine the zoo of \tat\ twists, or do computer
experiments, we need an appropriate notation for them. The goal is to encode a
pair of a surface $\Sigma$ with one boundary component and a \tat\ graph $G
\subset \Sigma$. Since by definition $\Sigma$ deformation retracts to $G$,
$\Sigma \smallsetminus G$ is homeomorphic to $S^1 \times \left] 0,1\right]$. So the
pair $(\Sigma,G)$ can be constructed from an annulus $S^1 \times \left[
0,1\right]$ by dividing one of its boundary components into $2e$ pieces, where
$e$ is the number of edges of $G$, and identifying them pairwise.

\begin{figure}
\centering
\def\svgwidth{0.3\textwidth}
{\fontfamily{pplx}\selectfont
	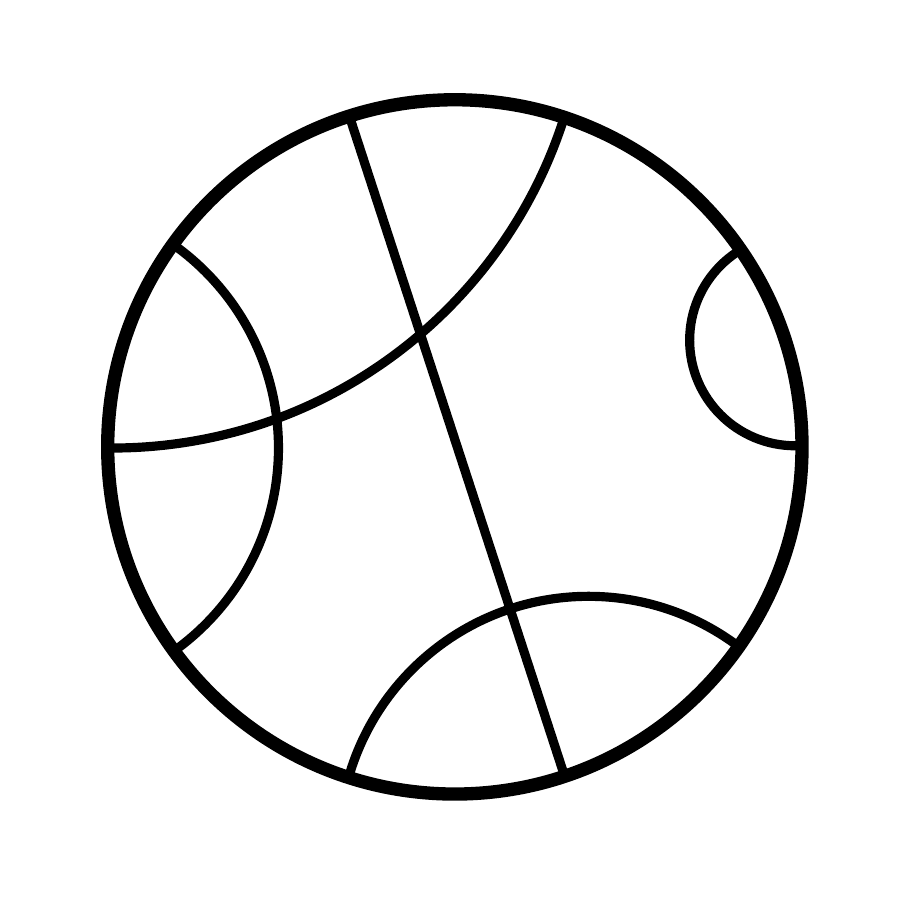
}
  \caption{A chord diagram with five chords}
  \label{fig:chorddiagram}
\end{figure}
\begin{definition}
A \emph{chord diagram} of size $n$ is a fixed-point free involution of the set
$\{1,\ldots,2n\}$, graphically represented by arcs (the \emph{chords}) that
connect labelled points on a circle.
We call two chord diagrams \emph{equivalent} if they only differ by a rotation
(keeping the labels fixed).
\end{definition}

\begin{figure}
\centering
\def\svgwidth{0.8\textwidth}
{\fontfamily{pplx}\selectfont
	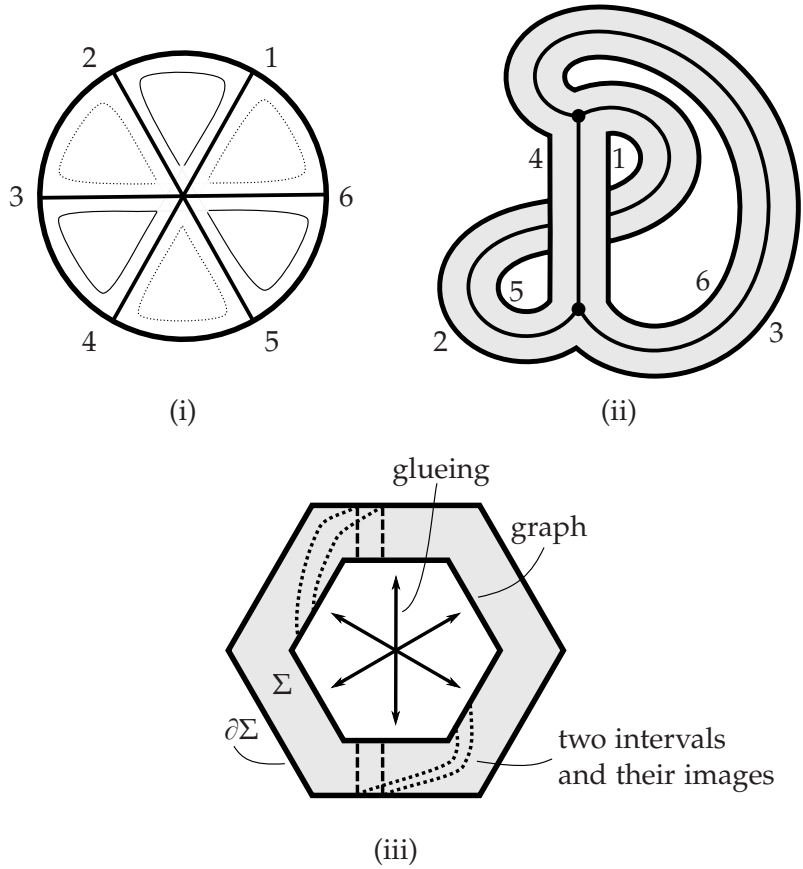
}
  \caption[Three illustrations of the relation between chord diagrams and graphs]%
  {\emph{(i)}~a chord diagram, with its two internal boundaries marked
              by one dotted and one solid line;
           \emph{(ii)}~the corresponding \tat\ graph, with the dotted internal boundary
              corresponding to the lower vertex, the solid one to the upper vertex;
           \emph{(iii)}~how the \tat\ graph is obtained from the diagram and how
              two properly embedded intervals are mapped by the \tat\ twist.}
  \label{fig:glueing}
\end{figure}

A chord whose endpoints are $r$ and $s$ will usually be given by the notation
$\{r,s\}$, and whenever convenient, $r$ and $s$ are to be understood as elements
of ${\Z}/{2n\Z}$. They may also be labelled by numbers from $0$ to $2n-1$,
for example when used for computations.

\subsubsection{Correspondence between chord diagrams and ribbon graphs}\label{sec:chorddiagramsandribbongraphs}
Equivalence classes of chord diagrams correspond to ribbon graphs with one
boundary component in a natural way. To build a ribbon graph from a chord
diagram, take a $2n$-gon and identify pairs of sides, reversing orientation, as
prescribed by the diagram, which makes a closed surface. When a similar smaller
$2n$-gon is removed, one gets a surface with boundary, and the glued edges form
an embedded graph. Alternatively, the edges to be glued can be put at the
inside, as in Figure \ref{fig:glueing}.

There is another way to construct these surfaces, which is sometimes useful even
though the graph can be seen less clearly in this way: Replace the circle of the
chord diagram by an annulus and glue bands to its inner boundary exactly as
indicated by the chords. Two such bands may cross, but whether one passes over the
other or vice-versa is not important. This makes a surface with potentially many
boundary components, but of the same genus as in the previous construction. One
boundary component is the outside of the annulus, the others we will call
\emph{internal boundaries of the chord diagram}. To get the same surface as
before, cap off all internal boundaries with disks. While the chords correspond
to the edges of the graph, the internal boundaries correspond to the vertices.

For a graph with one boundary component, where $e$ is the number of chords in its chord diagram and $v$ the number of internal boundaries, Formula \ref{eq:euler} becomes
	\begin{equation}
		g = \frac{1+e-v}{2}.
	\end{equation}

When we are given a ribbon graph with one boundary component and want to obtain
its chord diagram from it, we choose an arbitrary point on the boundary and,
moving along the boundary to the right, label each of the two sides of each edge
by consecutive numbers, as in the top half of Figure \ref{fig:glueing}.
The two numbers we see at an edge give us a chord.

\subsubsection{The \tat\ property in chord diagrams}
A chord diagram greatly helps recognizing a \tat\ property. This is illustrated
in the bottom part of Figure \ref{fig:glueing}, where the arrows indicate a
gluing. % glueing is correctly spelled, or spelt, but gluing is more common
We see here that the $\Theta$-graph from the previous picture has the
\tat\ property with walk length $1$ because paths which have the same endpoint
on the graph again share the same endpoint when they are composed with a safe
walk of length $1$. Thus the \tat\ property or, more precisely, the possible
walk lengths, show up as a rotational symmetry of the chord diagram. $G$ having
the \tat\ property with walk length $l$ means that the gluings are invariant
under rotation by $\frac{l}{2n}\cdot 2\pi = \frac{l\pi}{n}$.

%==================================================
\subsection{Building steps for ribbon graphs}\label{sec:buildingribbongraphs}
Chord diagrams lead us to the following observation:
\begin{proposition}\label{prop:ribbongraphs}
Two ribbon graphs with one boundary component are related to each other by a
sequence of the following two moves and their inverses:
\begin{enumerate-roman}
  \item stretching a vertex / collapsing an edge
	\begin{center}
		\def\svgwidth{0.5\textwidth}
		%% Creator: Inkscape 0.48.2, www.inkscape.org
%% Accompanies image file 'img-ribbongraphs1.pdf' (pdf, eps, ps)
\begingroup%
  \makeatletter%
  \providecommand\color[2][]{%
    \errmessage{(Inkscape) Color is used for the text in Inkscape, but the package 'color.sty' is not loaded}%
    \renewcommand\color[2][]{}%
  }%
  \providecommand\transparent[1]{%
    \errmessage{(Inkscape) Transparency is used (non-zero) for the text in Inkscape, but the package 'transparent.sty' is not loaded}%
    \renewcommand\transparent[1]{}%
  }%
  \providecommand\rotatebox[2]{#2}%
  \ifx\svgwidth\undefined%
    \setlength{\unitlength}{681.52104492bp}%
    \ifx\svgscale\undefined%
      \relax%
    \else%
      \setlength{\unitlength}{\unitlength * \real{\svgscale}}%
    \fi%
  \else%
    \setlength{\unitlength}{\svgwidth}%
  \fi%
  \global\let\svgwidth\undefined%
  \global\let\svgscale\undefined%
  \makeatother%
  \begin{picture}(1,0.34605877)%
    \put(0,0){\includegraphics[width=\unitlength]{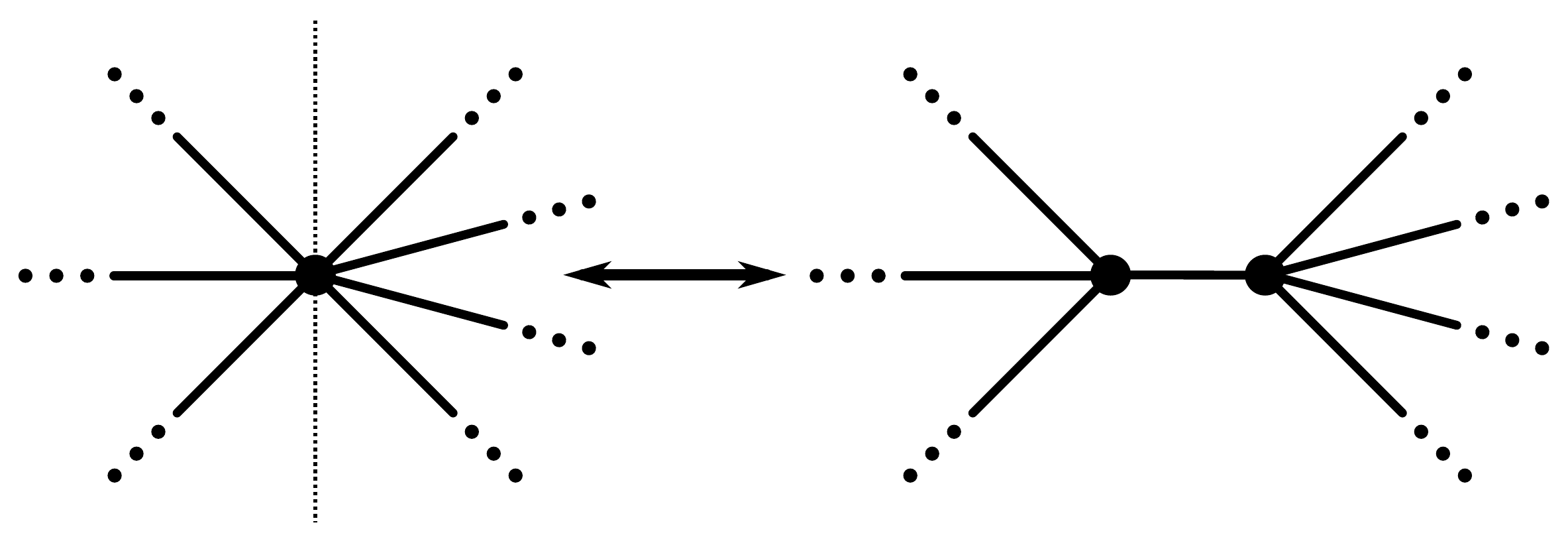}}%
  \end{picture}%
\endgroup%

		  %\caption{}
		 \label{fig:ribbongraphs1}
	\end{center}
  \item hitching two vertices / unhitching a vertex
    \begin{center}
		\def\svgwidth{0.5\textwidth}
		%% Creator: Inkscape 0.48.2, www.inkscape.org
%% Accompanies image file 'img-ribbongraphs2.pdf' (pdf, eps, ps)
\begingroup%
  \makeatletter%
  \providecommand\color[2][]{%
    \errmessage{(Inkscape) Color is used for the text in Inkscape, but the package 'color.sty' is not loaded}%
    \renewcommand\color[2][]{}%
  }%
  \providecommand\transparent[1]{%
    \errmessage{(Inkscape) Transparency is used (non-zero) for the text in Inkscape, but the package 'transparent.sty' is not loaded}%
    \renewcommand\transparent[1]{}%
  }%
  \providecommand\rotatebox[2]{#2}%
  \ifx\svgwidth\undefined%
    \setlength{\unitlength}{681.52104492bp}%
    \ifx\svgscale\undefined%
      \relax%
    \else%
      \setlength{\unitlength}{\unitlength * \real{\svgscale}}%
    \fi%
  \else%
    \setlength{\unitlength}{\svgwidth}%
  \fi%
  \global\let\svgwidth\undefined%
  \global\let\svgscale\undefined%
  \makeatother%
  \begin{picture}(1,0.288252)%
    \put(0,0){\includegraphics[width=\unitlength]{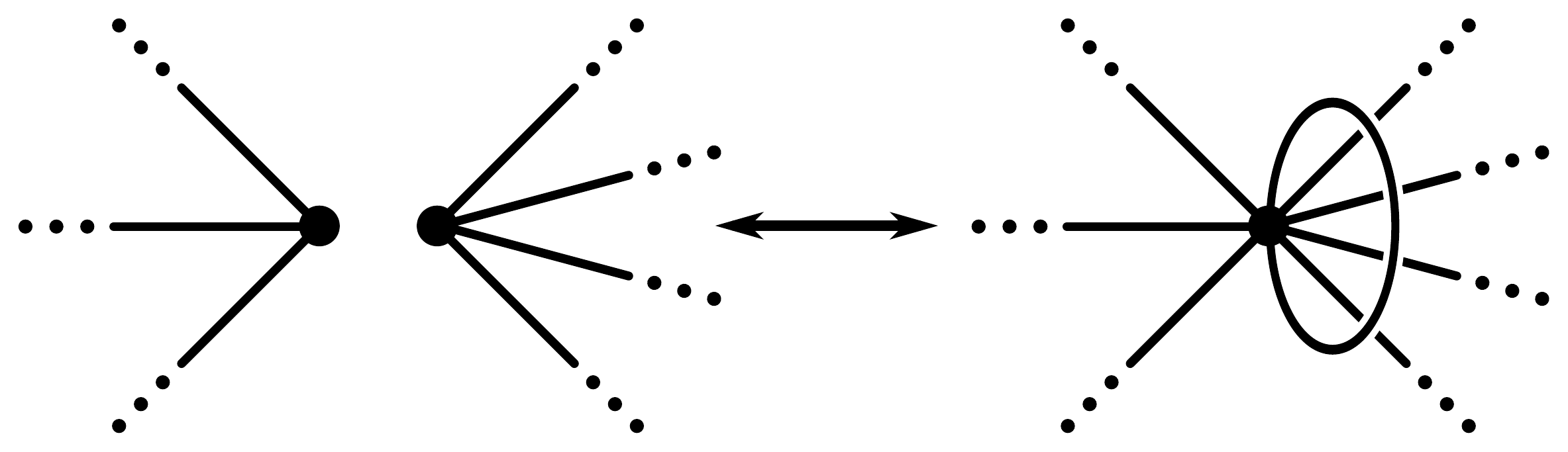}}%
  \end{picture}%
\endgroup%

		  %\caption{}
		 \label{fig:ribbongraphs2}
	\end{center}
\end{enumerate-roman}
\end{proposition}
Otherwise stated, every such ribbon graph can be built from the graph with just
one vertex and no edges by vertex stretching and hitching.
\begin{proof}
The graph with one vertex and no edges is represented by the empty chord
diagram. Whenever a new edge is added to some chord diagram, two things are
possible:
\begin{enumerate-roman-packed}
\item The chord separates an internal boundary component into two: This
corresponds to stretching the respective vertex.
\item The chord connects two internal boundaries: This corresponds to hitching
two vertices, and increases the genus of the surface by one.
\end{enumerate-roman-packed}
\end{proof}
\begin{remark}
A consequence from these considerations that will be used further on is: Adding new chords to
a chord diagram can only increase its genus.
\end{remark}

%==================================================
\subsection{Equivalence of \tat\ graphs}\label{sec:equivalence}
Some \tat\ twists are equivalent to others in the sense that they represent the
same mapping classes.
For example, we have already seen that edges can be subdivided and the walk
length adapted accordingly, if the subdivision is done equally for the entire
orbit of the edge.
On the level of chord diagrams, this corresponds to replacing a chord, as well as its
images under the given rotation, by two or more parallel ones.
If there is only one boundary component, orbits of univalent vertices can be removed
(or introduced) at will.
This corresponds to removing an orbit of chords that connect neighbouring labelled points.

These two modifications are in fact examples of a slightly more general process,
which is of course reversible:
\begin{proposition}
Let $\T{G,l}$ be a \tat\ twist with an edge orbit that consists of contractible
components. Then $\T{G',l'}$, where $G'$ is obtained from $G$ by contracting all
edges of this orbit and $l'$ is the suitably adapted walk length, defines
the same mapping class.
\end{proposition}
This process works because the contracted components are homeomorphic to disks and
the symmetry of the graph is not destroyed.

In the special case where the twist is the identity, or a composition of Dehn twists around boundary components, one can contract every edge that is not
a loop and end up with a bouquet of circles.
One might ask whether the collapse of a contractible edge orbit this is
the only kind of equivalence that is needed:
\begin{question}\label{q:collapsible}
	If $\T{G_1,l_1}$ and $\T{G_2,l_2}$ (possibly with multiple walk lengths)
	represent the same mapping class, is there a graph $G$ that is obtained
	from both $G_1$ and $G_2$ by collapsing contractible edge orbits?
\end{question}

\subsubsection{An example obtained by ``blow-up''}
An example of such a contraction is shown in Figure \ref{fig:sebastianexample}.
On the left, we see the complete bipartite graph which (with its minimal walk length
of $2$) describes the monodromy of the $(2,3)$-torus knot, the trefoil. If one
of the two top vertices is pulled down, it looks like the surface in the middle,
where the three bands going down have received a half twist. We ``blow up'' the
two vertices by replacing them with another \tat\ graph -- in this case the
same complete $(2,3)$ bipartite graph -- and update the walk length such that it
induces the same symmetry on the original edges. This is a general construction,
suggested by A'Campo, to create more complicated \tat\ graphs.

The edges coming from the original graph are still all in one orbit, but are now
separated into three connected components which can be contracted. When we do
so, we obtain a complete $(4,3)$ bipartite graph. Note, however, that the
induced walk length is the double of the minimal walk length for this graph.
Therefore the \tat\ twist we get is the square of the monodromy of the $(4,3)$
torus knot.
\begin{figure}
\centering
\def\svgwidth{0.7\textwidth}
{\fontfamily{pplx}\selectfont
	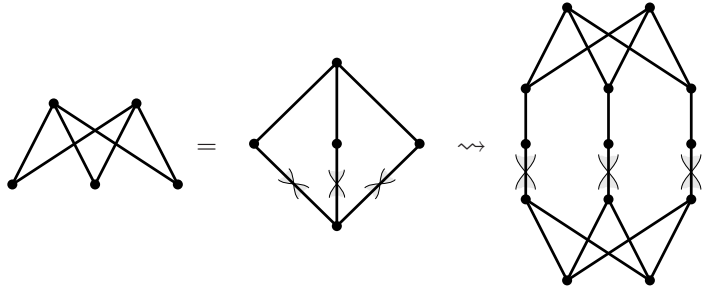
}
  \caption[A (2,3)-bipartite graph with its 2-vertex set ``blown-up'' to two
  (2,3)-bipartite graphs]%
  {A (2,3)-bipartite graph with its 2-vertex set ``blown-up'' to two
  (2,3)-bipartite graphs. The crosses on the edges represent a half twist. This
  \tat\ graph has an edge orbit with contractible components.}
  \label{fig:sebastianexample}
\end{figure}

\subsubsection{A necessary bivalent vertex}\label{sec:bivalent}
Figure \ref{fig:bivalentneeded} shows an example of a \tat\ graph with a walk
length of $3$ where the bivalent vertex cannot be removed, unless the twists
is described as a multi-speed twist with different walk lengths for the two boundary components.
\begin{figure}
\centering
\def\svgwidth{0.4\textwidth}
{\fontfamily{pplx}\selectfont
	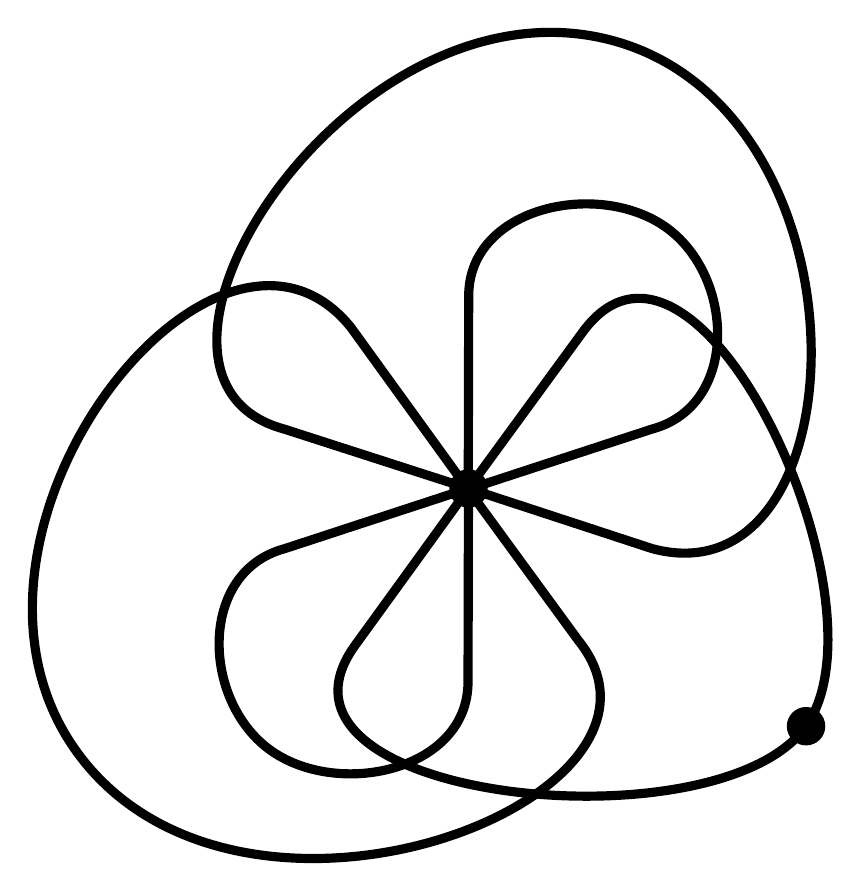
}
  \caption[A \tat\ graph that needs its bivalent vertex]%
  {A \tat\ graph with two boundary components, genus two, and a walk length of three, that needs its bivalent vertex}
  \label{fig:bivalentneeded}
\end{figure}

%==================================================
\section{Elementary \tat\ twists}\label{sec:elementary}
In this part we are going to study a class of \tat\ twists which can be seen as
building blocks for general \tat\ twists for surfaces with one boundary
component. Their combinatorics will also be important in Chapter \ref{chap:periodic}, where we will use them to study periodic diffeomorphisms.
\begin{definition}
A \tat\ twist $\T{G,l}$ with $\#\partial G=1$ is called \emph{elementary}
if it acts transitively on the set of edges of $G$.
\end{definition}
Most graphs we have seen up to now are of this type. For example, the twists
along $(p,q)$-bipartite graphs that represent the monodromy of $(p,q)$-torus
knots act on their $p\cdot q$ edges by cyclic permutation.

We first prove a classification for elementary twists:
\begin{theorem}
Elementary \tat\ twists $\T{G,l}$ have underlying graphs $G$ from a two-parameter
family $E_{n,a}$, $n,a\in\N$, $a\leq n$, $a$ odd if $a<n$. Its members are described by chord
diagrams with $n$ chords and constant chord length $a$, with chords of the
form $\{2k,2k+a\}$ and $\{2k+1,2k+1-a\}$. For the twists, $l=1$ or $2$ automatically.

The diffeomorphism $\T{G,l}$ can always be represented as some $\T{E_{n,a},2}$
with $a<n$.

Conversely, every \tat\ twist with walk length $1$ or $2$ and one boundary component
is elementary (and is hence of the form $E_{n,a}$).
\end{theorem}
Here, the \emph{length} of a chord from $r$ to $s$ is
$\min(\abs{r-s},n-\abs{r-s})$. This also corresponds to the length of the
shortest safe walk from one side of the edge represented by the chord to the
other -- in terms of the introduction: the shortest way to get to the other side
of the road without crossing it. We will use the notation $E_{n,a}$ for both the
graph and the chord diagram.
% Why did the chicken not cross the road? To get to the other side.
\begin{figure}
\centering
\def\svgwidth{0.3\textwidth}
{\fontfamily{pplx}\selectfont
	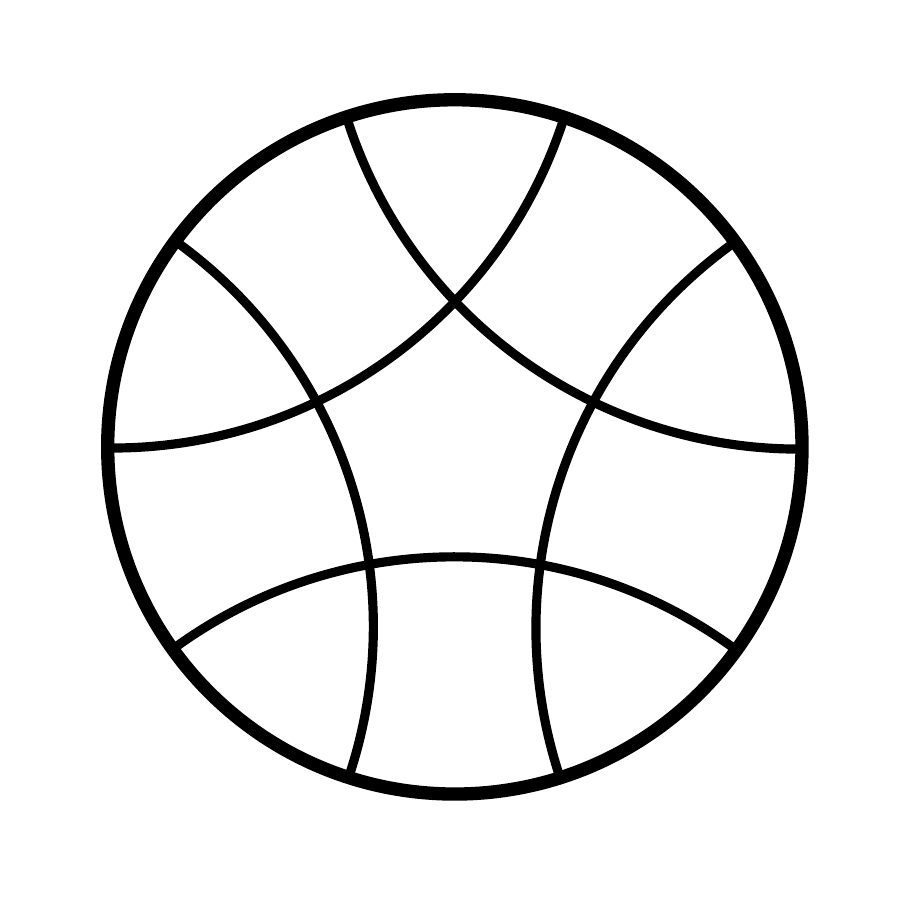
}
  \caption{The chord diagram $E_{5,3}$}
  \label{fig:chorddiagram-E-5-3}
\end{figure}
\begin{proof}
Let $\T{} = \T{G,l}$ have only one edge orbit. This means that for each pair $c_1$,
$c_2$ of chords in the chord diagram of $G$ there is a rotation bringing $c_1$
to $c_2$. The lengths of $c_1$ and $c_2$, and of all chords, are necessarily
equal. Call this length $a$, and let $\{r,r+a\}$ be one chord. Next to it, there
must be either a chord $\{r+1, r+1+a\}$, or a chord $\{r+1,r+1-a\}$.

If the former is the case, then the rotation by one step, sending $r$ to
$r+1$, must be a symmetry of the diagram. Hence also the rotation $r\mapsto r+a$
is a symmetry, so $\{r+a, r+2a\}$ must be a chord as well. We end up in the case where
$a=n$, meaning all chords are diameters and $l=1$.

Assume now that $a<n$, which implies that $\{r+1,r+1-a\}$ is a chord. The next
one, by analogous reasoning, will be $\{r+2,r+2+a\}$. We get a chord diagram
with a two-step symmetry that sends $r$ to $r+2$. $a$ must be odd in this case.
\begin{figure}
\centering
\def\svgwidth{0.75\textwidth}
{\fontfamily{pplx}\selectfont
	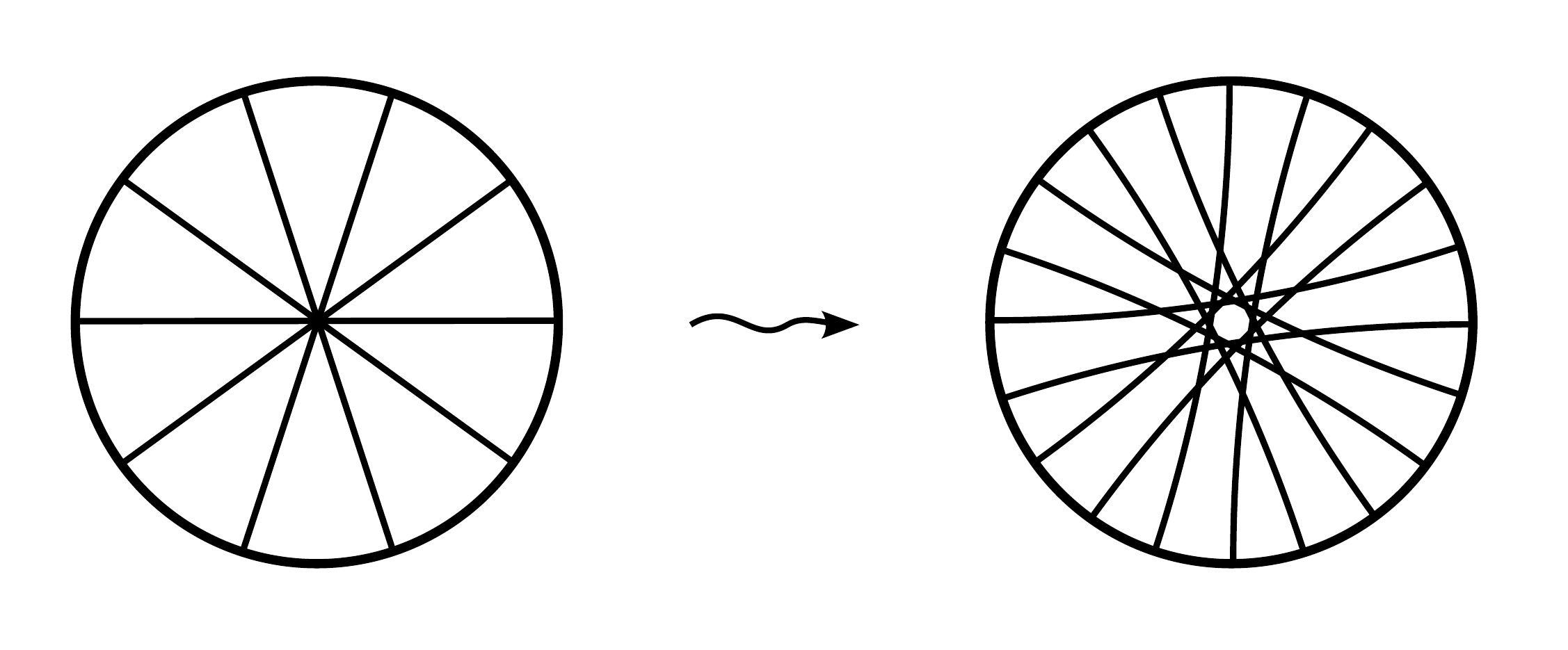
}
  \caption{Changing $E_{5,5}$ to $E_{10,9}$}
  \label{fig:chorddoubling}
\end{figure}

The case where $l=1$ is special in that powers of $\T{}$ will not only act
transitively on chords, or edges, but can also map them to themselves while
reversing orientation. This can be changed by replacing each chord by two
parallel chords, which corresponds to subdividing each edge once. Hence whenever
we like, we are free to replace $\T{E_{n,n},1}$ with $\T{E_{2n,2n-1},2}$ (see
Figure \ref{fig:chorddoubling}).

For the converse statement, remark first that it is clear that whenever a chord
diagram has a rotational symmetry sending $r$ to $r+1$, all chords must be
diameters, that is to say, of the form $\{r,r+n\}$.

The interesting case is $l=2$ and $a<n$. Now each chord has a well-defined
``first'' and ``second'' endpoint, counting clockwise. Since each of the $2n$
points of the diagram is either a first or a second endpoint, and the rotation
has two equal orbits on the level of points, it acts transitively on the set of
first endpoints, hence on chords.
\end{proof}

%==================================================
\subsection{Bounds for elementary \tat\ twists}
Because vertices correspond to internal boundaries of a chord diagram, we can
also observe:
\begin{remark}\label{rem:bipartite}
All elementary \tat\ graphs $E_{n,a}$ where $a<n$ are bipartite.
\end{remark}
This is because a chord divides the unit disk into two parts, the bigger of
which could be called the ``outside'', the smaller one the ``inside''. A vertex
of $E_{n,a}$ lies either on the outside or on the inside of the chords that
bound it. So it falls into one of two classes; and edges, as they correspond to
chords, connect only vertices of one class to vertices of the other.
\begin{figure}
\centering
\def\svgwidth{0.35\textwidth}
{\fontfamily{pplx}\selectfont
	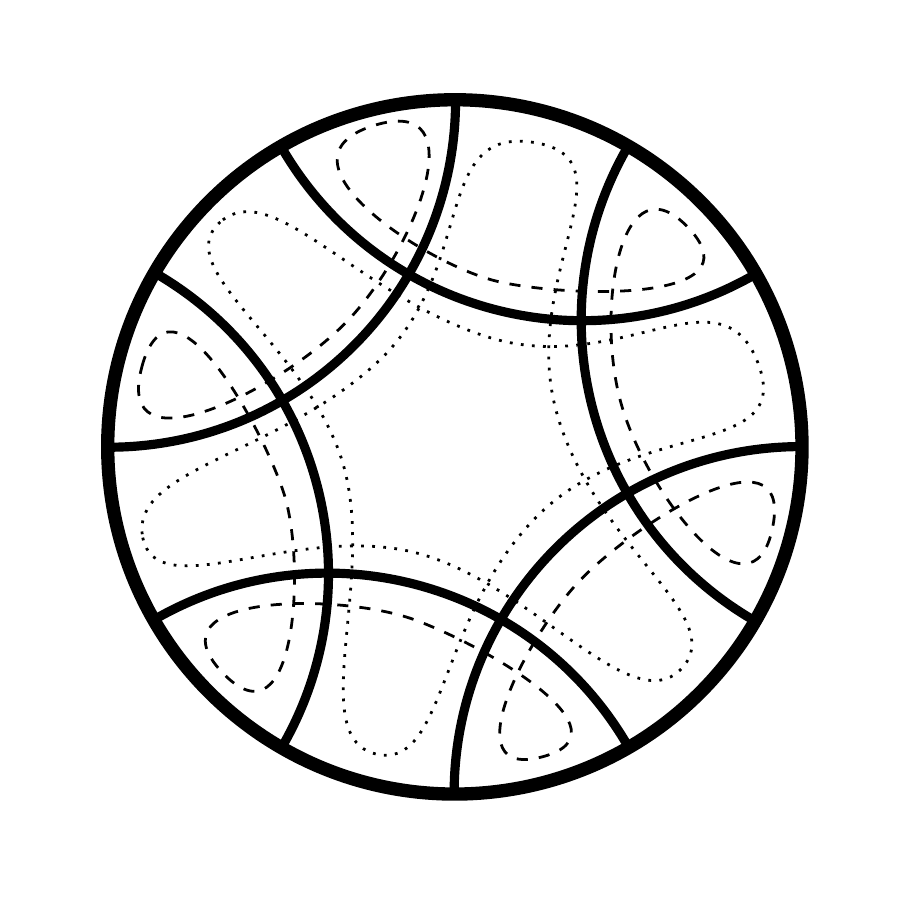
}
  \caption[$E_{6,3}$ with its two outer vertices and its one inner vertex]%
  {$E_{6,3}$ with its two outer vertices (dotted) and its one inner vertex (dashed)}
  \label{fig:chorddiagramvertices}
\end{figure}
Let us call the two classes of vertices ``outer'' and ``inner'' vertices for the
moment. We can easily count their number: Following the chords along an outer
vertex, we encounter first endpoints at a distance of $\frac{a+1}{2}$. Since
there are $n$ first endpoints, the number of outer vertices is the greatest
common divisor of $\frac{a+1}{2}$ and $n$, for which we write
$\gcd\left(\frac{a+1}{2},n\right)$.
Likewise, the number of inner vertices is $\gcd\left(\frac{a-1}{2},n\right)$. Thus
we have:
\begin{lemma}\label{lem:elementarygraphgenus}
The elementary \tat\ graph $E_{n,a}$ with $a<n$ has
	\[
		v = \gcd\left(\frac{a-1}{2},n\right) + \gcd\left(\frac{a+1}{2},n\right)
	\]
vertices. Its genus is
	\[
		g(E_{n,a}) = \frac{1 + n - \Bigl(\gcd\left(\frac{a-1}{2},n\right) + \gcd\left(\frac{a+1}{2},n\right)\Bigr)}{2}.
	\]
When $a=n$, there are just one or two vertices depending on whether $n$ is even or odd, and
$g(E_{n,n}) = \lfloor n/2 \rfloor$.
\end{lemma}
The case $a=n$ with only diametral chords is easily understood; its surface is
obtained from gluing opposite edges and its \tat\ twist $\T{E_{n.n},1}$
corresponds to the rotation of a $4g$- or $4g{+}2$-gon by one click.

Assume again $a<n$. The case $a=1$ is not interesting as it just produces a disk
with a star of $n$ univalent vertices. The case $E_{n,n-1}$, for $n$ even,
corresponds to $E_{n/2,n/2}$ with all edges subdivided once. If we want to
forbit uni- and bivalent vertices, we have thus to restrict $a$ to an odd number
between $3$ and $n-2$. The order of $\T{E_{a,n},2}$ is equal to the number of
edges $n$. The following lemma restricts $n$ depending on the genus; it is much stronger
than the calculations for general graphs in Section \ref{sec:generalbounds}.
\begin{lemma}\label{lem:elementary3g+3}
Let $3\leq a\leq n-2$ in the elementary \tat\ graph $E_{n,a}$, and $g = g(E_{n,a})$ be its genus.
Then
	\[
		n \leq 
		\begin{cases}
		   3g+3, & g \equiv 0,1\\
		   3g, & g \equiv 2
 		\end{cases}
 		\pmod 3.
	\]
Moreover, both inequalities are sharp, i.\,e.\@ there exists an elementary \tat\
graph whose twist is of order $3g+3$ or $3g$, respectively.

When $g \geq 4$, this graph is unique and given by the pair
	\[
		(n,a) =
			\begin{cases}
				(3g+3,2g+1), & g \equiv 0 \\
				(3g+3,2g+3), & g \equiv 1\\
				(3g,2g-1), & g \equiv 2
			\end{cases}
		\pmod{3}.
	\]
\end{lemma}
\begin{proof}
Set $k=\frac{a+1}{2}$, which is an integer. We then have to bound the sum $v = \gcd\left(k-1,n\right) +
\gcd\left(k,n\right)$ from above, for $k$ between $2$ and $\frac{n-1}{2}$. Since
$k-1>0$ and since $k-1$ and $k$ have no common divisor,
$\gcd\left(k-1,n\right)\cdot\gcd\left(k,n\right) \leq n$. If, say, $\gcd\left(k-1,n\right) =
\frac{n}{r}$, we get that
	\[
		v = \gcd\left(k-1,n\right) + \gcd\left(k,n\right) \leq \frac{n}{r} + {r}.
	\]
Since $k\leq\frac{n-1}{2}$, both summands are at most $\frac{n}{3}$. Therefore
$3\leq r \leq \frac{n}{3}$, and using that $\frac{n}{r} + {r} \leq \frac{n}{3} + {3}$
(with equality only if $r=3$ or $r=\frac{1}{3}$), we get that
	\[
	n = 2g + v - 1 \leq 2g + \frac{n}{3} + 3 - 1,
	\]
hence
	\[
		\frac{2}{3}n \leq 2g + 2
	\]
and the inequality is proved for $g \not\equiv 2 \pmod 3$.

Let us assume that $n = 3 g_0 + 3$ for some $g_0$ and try to actually find an $a$, or $k$, such
that $g = g(E_{n,a}) = g_0$. Either $k$ or $k-1$ must be equal to $\frac{n}{3}$, the other
must be divisible by $3$. If $g_0 \equiv 0 \pmod 3$, take
	\[
		k = \frac{n}{3} = g_0+1 \equiv 1 \pmod 3,
	\]
which means that $a = 2g_0+1$. If $g_0 \equiv 1 \pmod 3$, take
	\[
		k - 1  = \frac{n}{3} = g_0 + 1 \equiv 2 \pmod 3,
	\]
which means that $a = 2g_0+3$. Those choices are unique.

Now to the case that $g_0 \equiv 2 \pmod 3$. When $n = 3g_0$, we can take
	\[
		k = \frac{n}{3} = g_0 \equiv 2 \pmod 3,
	\]
so $k-1$ is not divisible by $3$ and $v = g_0+1$ as required in this case. We have $a = 2g-1$.

There is no way to choose $k$ such that $v = \frac{n}{3} + 3$, so $n = 3g+3$
cannot be achieved. $3g+2$ and $3g+1$ are not divisible by $3$, so in these cases
$\gcd\left(k-1,n\right)$ and $\gcd\left(k,n\right)$ will be strictly smaller than
$\frac{n}{3}$.

To prove that there are no twists of order $3g+2$ and $3g+1$ and no further twists of
order $3g$, it remains to check that no $r$ bigger than $3$ can work. An easy way to
check this is to observe that when $r \geq 4$ and $v \leq \frac{n}{4} + 4$, then $n \geq 3g$ implies that
$n \leq 36$. One can check these cases by computer or by hand.

The only collision occurs for $g=3$, where it happens that $3g+3 = 4g = 12$, and we find
both $E_{12,7}$ and $E_{6,6}$ of order $12$.
\end{proof}

Elementary \tat\ twists realize the highest possible orders among \tat\ twists with one
boundary component: As we have seen in Proposition \ref{prop:ribbongraphs}, adding chords can
only increase the genus. When a chord diagram with more than one edge orbit has a rotational
symmetry of order $n$, all its edge orbits individually have a rotational symmetry of order
$n$, or possibly $2n$ for orbits that consist of diameters. We can therefore conclude:
\begin{corollary}\label{cor:tatorders}
Let $\T{}{}$\, be a \tat\ twist whose graph is of genus $g$ with one boundary component. Then its order is
either $4g+2$, $4g$, or
	\[
		\ord(\T{}\ ) \leq 
		\begin{cases}
		   3g+3, & g \equiv 0,1\\
		   3g, & g \equiv 2
 		\end{cases}
 		\pmod 3.
	\]
In all of these cases, as soon as $g \geq 4$, there exists a unique conjugacy class
of \tat\ twists realizing the given order. This class is described by an elementary
\tat\ twist $E_{n,a}$ with $(n,a)$ as in Lemma \ref{lem:elementary3g+3}
above.\qed
\end{corollary}
Given $g \geq 2$, observe that when we take a twist of order $4g$ or $4g+2$, adding a second
edge orbit cannot thwart uniqueness. Those twists have chord diagrams consisting of diameters.
When we add a second orbit consisting of diameters also, we get a twist of order $2g$ or $2g+2$,
which is smaller than $3g$. Only when $g=2$, we find a second twist of order six: $\T{E_{6,6},4}
= \T{E_{6,6},2}^2$, in addition to $\T{E_{6,3},2}$. When we add a second orbit which does not consist
of diameters, the restrictions from Lemma \ref{lem:elementary3g+3} apply to that orbit.

The last thing to note is this: Even when the second edge orbit does not
increase the genus, it will not produce a new conjugacy class. Because in that
case, as seen in Proposition \ref{prop:ribbongraphs}, new edges are introduced
at vertices without breaking the symmetry of the graph. This will not change the
isotopy class of the \tat{} twist.

%==================================================
\subsection{Chord diagrams for torus knots}
The \tat\ graphs that describe the monodromies of torus knots were described in the
introduction; as mentioned in the beginning of this chapter they have walk length
$2$ and are therefore described by an elementary \tat\ twist.
Figure \vref{fig:bipartite2} can serve as an example.
If we want to describe such a graph in the form $E_{n,a}$ we have to calculate the
chord length, i.\,e.\@ check how long a safe walk takes to ``cross the street''.
It is an even number. A safe walk of length $2$ corresponds to the twist and exchanges
cyclically the vertices at the top as well as those at the bottom.
This leads to a simple calculation and to the following statement: 
\begin{proposition}
The monodromy of a $(p,q)$-torus knot, $p<q$, is a \tat\ twist with walk length
$2$ around the elementary \tat\ graph $E_{pq,2mp-1}$, where $m\cdot p
\equiv 1\pmod{q}$.\qed
\end{proposition}
The chord length $a = 2mp-1$ in the proposition may be ``the long way around''
the surface.
In that case, it can of course be replaced by $2pq-a$.
In the example, which shows the $(3,4)$-torus knot, we have $12$ edges or
chords,
and since $3 \cdot 3 \equiv 1 \pmod{4}$, the chord length is
$2\cdot3\cdot3 - 1 = 17$, or better $24-17=7$.
One may check this by labelling the two sides of the edges as in Figure
\vref{fig:glueing}.

%==================================================
\section{Fixed points}
All \tat\ twists leave the boundary of the surface $\Sigma$ on which they live
pointwise fixed. But these are not ``essential'' fixed points as they can be
removed by a small perturbation, like a small translation along $\partial\Sigma$
in the direction induced by the orientation of $\Sigma$ if the walk length is positive.

Fixed points in $\Sigma\smallsetminus(\partial\Sigma \cup G)$ can likewise be removed
by composing with a diffeomorphism that pushes all such points slightly away
from $\partial\Sigma$ and towards $G$. So essential fixed points should be
searched on $G$, which is mapped to itself by \tat\ twists, as we have seen.

Both twists with fixed points as well as without occur. Most famously, a Dehn
twist has no (essential) fixed points, and neither has any power of it, although
in this case we must modify the diffeomorphism in such a way that the circle $G$
is not any more mapped to itself. But also the monodromy of the trefoil,
$\T{E_{3,3},1}$ has none: It permutes the three edges cyclically and
interchanges the two vertices.
\begin{figure}
  \centering
  \def\svgwidth{0.7\textwidth}
  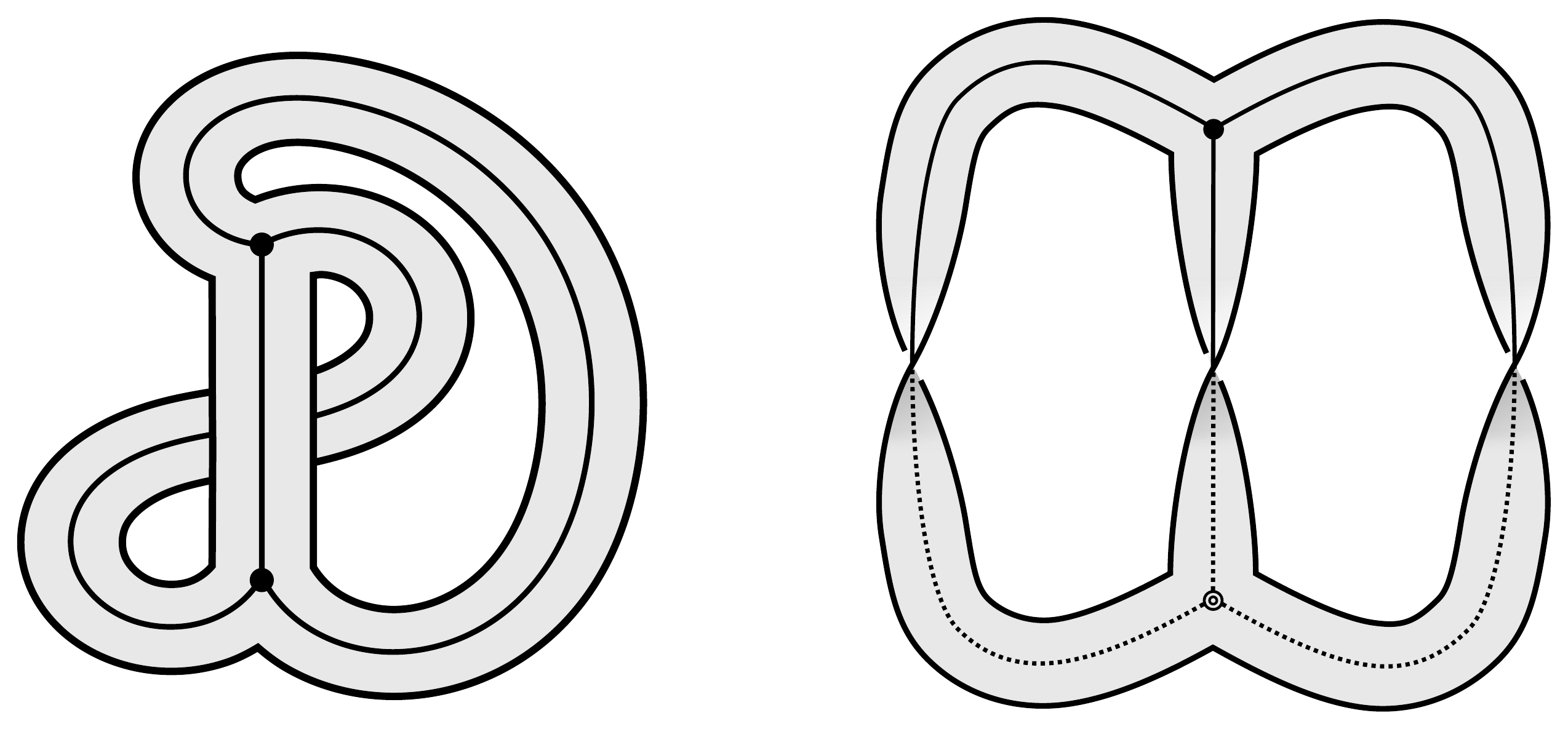
  \caption{$E_{3,3}$}
  \label{fig:E-3-3}
\end{figure}

On the other hand, $\Bi = \T{E_{2,2},1}$ maps the only vertex of $E_{2,2}$ to
itself. And this fixed point is indeed essential: The Lefschetz number of $\Bi$
is
\[
   \Lambda(\Bi) = 1 - \tr(B) = 1
\]
where $B$ denotes the induced action of $\Bi$ on $H_1(\Sigma)$, in this
case given by a matrix conjugate to
$\begin{psmallmatrix*}[r] 0 & -1 \\
                           1 &  0 \end{psmallmatrix*}$.

%==================================================
\chapter{\Tat\ twists and periodic diffeomorphisms}\label{chap:periodic}
A mapping class $\phi$ is called \emph{periodic} or \emph{of finite order} if
there is some $k>0$ such that $\phi^k$ is the isotopy class of the identity. Nielsen has
shown in 1942 (\cite{Nielsen1942}) that such mapping
classes contain a representative -- a diffeomorphism -- $f$ such that
$f^k$ is actually equal to the identity. Moreover, whenever the surface has
negative Euler characteristic, one can find a hyperbolic metric such that the
diffeomorphism is an isometry for this metric. This is true for closed surfaces, as
well as for surfaces with punctures and for surfaces with boundary where one
allows the boundary to rotate.

%==================================================
\section{Periodic diffeomorphisms on surfaces with boundary}\label{sec:periodic}
\begin{figure}
  \centering
  \def\svgwidth{0.5\textwidth}
  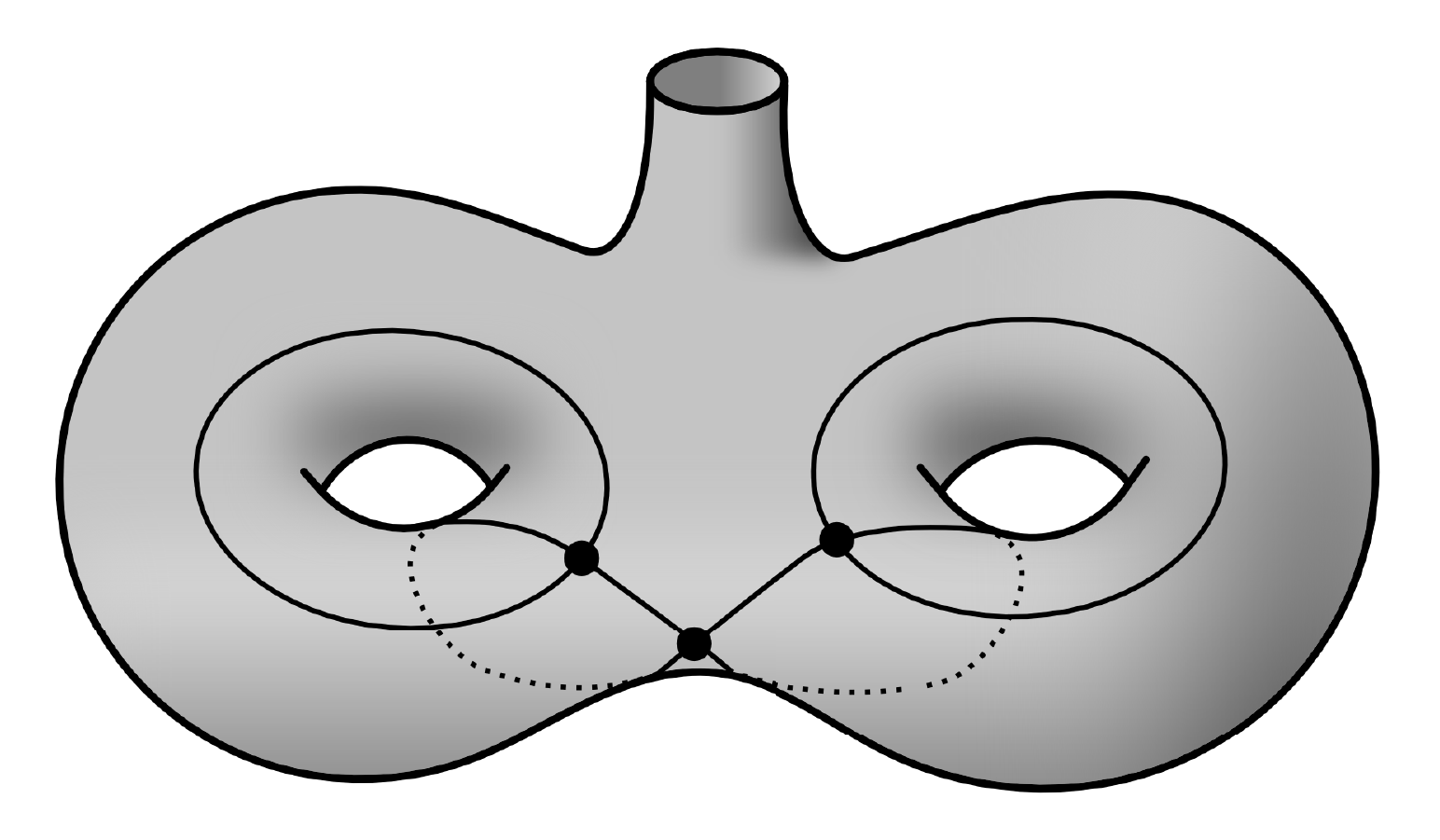
  \caption[A \tat\ graph that describes a rotation by 180 degrees]%
  {A \tat\ graph with walk length $6$, describing the map of order $2$ one gets by
   rotating the whole surface by 180 degrees along a vertical axis, while keeping the top boundary fixed}
  \label{fig:rotation180}
\end{figure}
If one requires the boundary of a surface to be pointwise fixed by the
diffeomorphisms and the isotopy, as we always do here, there are no periodic
mapping classes apart from the identity. Using Nielsen's theorem, this can be
seen geometrically: Via the exponential map, an isometry is determined by
the image of a point and a tangent vector at the point. Points on the boundary
are fixed, and if the isometry is to be periodic an inward pointing
tangent vector must be fixed as well, so the map is the identity. \Tat\ graphs
will give us another, topological, proof of this.

We can therefore use the term ``periodic'' in a more general way and call a
mapping class, or diffeomorphism, that fixes the boundary \emph{freely periodic} or simply
\emph{periodic} if, for some $k>0$, $\phi^k$ is isotopic to the identity, where
the isotopy is allowed to move the boundary.

Another way to say this: When $\Sigma$ is a surface with $b$ boundary components
and $\dot{\Sigma}$ is $\Sigma$ with its boundary collapsed to punctures, there
is a central extension
   \[
      0 \to \Z^b \to \mcg(\Sigma) \overset{c}{\to} \mcg(\dot{\Sigma}) \to 0
   \]
where the subgroup $\Z^b$ is generated by the Dehn twists along the boundary
components. Thus we call $\phi$ periodic (or of finite order) if $c(\phi)$ is
periodic in the ordinary sense. This usage is quite common in the context of
monodromies of singularities (although sometimes ``monodromy'' can refer
to the action on homology only).

We have seen in Proposition \ref{prop:basic} that \tat\ twists are periodic in
the above sense. One may now ask how one could recognize whether a periodic map
is induced by a \tat\ twist, and how to find a \tat\ graph for it in this case.
The answer is surprisingly simple: Every periodic diffeomorphism comes from a
\tat\ graph.
\begin{theorem}\label{thm:equivalentdefs}
Let $\Sigma$ be a (compact, connected, oriented) surface with nonempty boundary and $\phi\in\mcg(\Sigma)$
a mapping class. Then the following are
equivalent:
	\begin{enumerate-roman-packed}
		\item\label{itm:istat}				$\phi$ is a multi-speed \tat\ twist,
		\item\label{itm:hasinvariantgraph}  there is an embedded graph that fills $\Sigma$ and
			  								and is invariant under $\phi$,
		\item\label{itm:hasinvariantspine}  $\Sigma$ has a spine that is invariant
											under $\phi$,
		\item\label{itm:isperiodic}			$\phi$ is (freely) periodic.
	\end{enumerate-roman-packed}
\end{theorem}
We say that a graph $G \subset \Sigma$ \emph{fills} $\Sigma$ if its complement
$\Sigma \smallsetminus G$ consists only of disks and boundary-parallel annuli.

During the proof, we will see how to explicitly find a \tat\ graph, given a
periodic diffeomorphism. The invariant graph in \ref{itm:hasinvariantspine}
will be the graph around which one twists. It can also be useful to view the bigger
filling graph from \ref{itm:hasinvariantgraph} as a \tat\ graph,
one whose embedding corresponds to a subsurface of $\Sigma$.

\begin{proof}
We have already seen in Proposition \ref{prop:basic} that \tat\ twists are periodic and that
they leave the defining \tat\ graph invariant, so \ref{itm:hasinvariantgraph}, \ref{itm:hasinvariantspine},
and \ref{itm:isperiodic} follow from \ref{itm:istat}.

\proofstep{\ref{itm:hasinvariantspine}$\implies$\ref{itm:istat}}
\begin{figure}
  \centering
  \def\svgwidth{0.5\textwidth}
  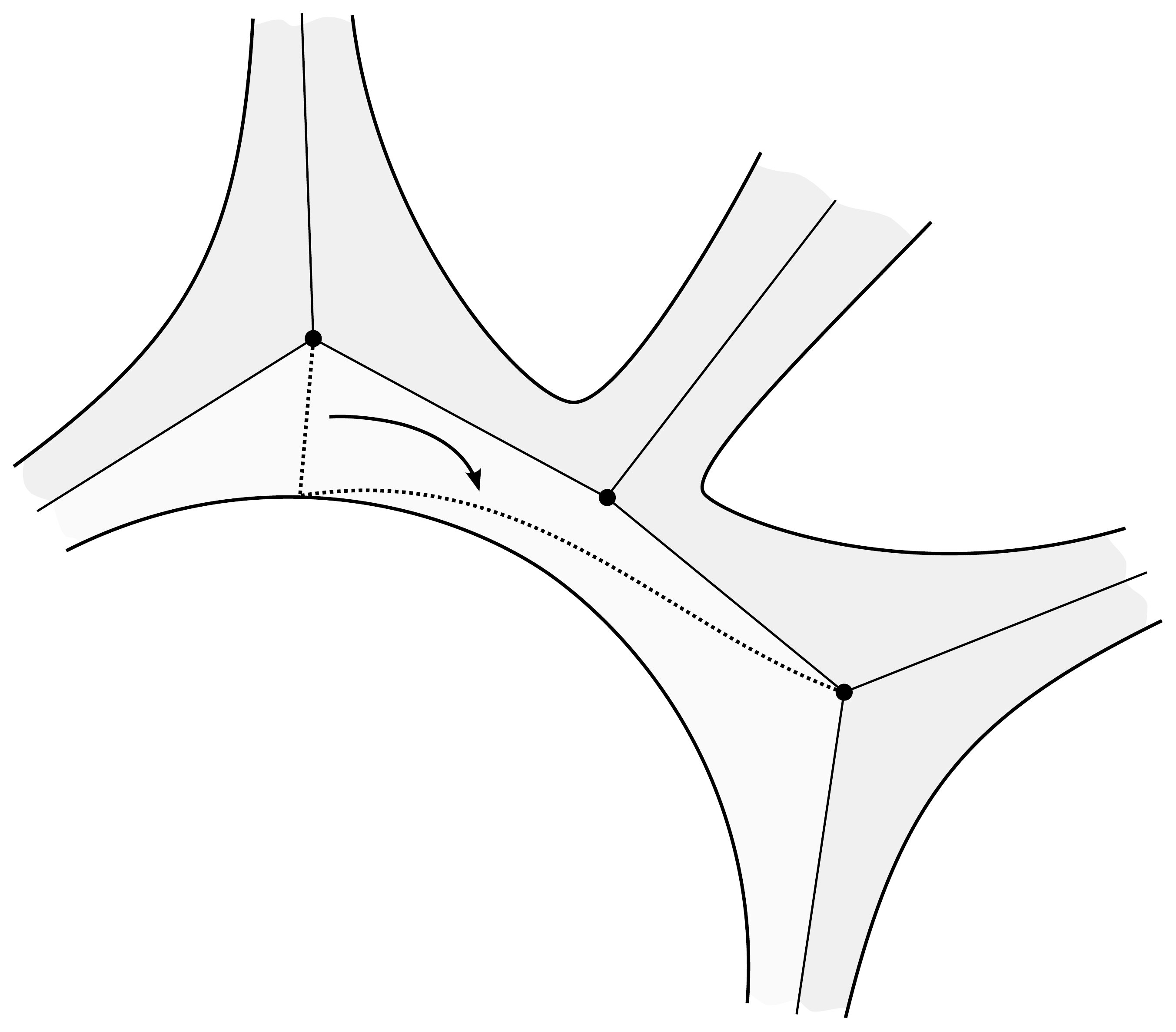
  \caption{The image of one arc determines the map on the annulus up to isotopy}
  \label{fig:onearcdetermines}
\end{figure}
We first argue why having an invariant graph as in \ref{itm:hasinvariantspine}
makes $\phi$ a \tat\ twist. To see this, assume without loss of generality that
the graph has no bivalent vertices and is contained in the interior of $\Sigma$.
Pick one boundary component; see Figure \ref{fig:onearcdetermines}.
Between that boundary component and $G$ there is an annulus $A$.
Choose a vertex of $G$ that is adjacent to $A$, and an arc going from the
boundary component to the vertex. By assumption, $\phi$ fixes the boundary of
$\Sigma$, and in particular the endpoint of the arc which lies on the boundary.
And being a diffeomorphism that leaves $G$ invariant, it sends vertices to
vertices.

Therefore it must send the second endpoint to some other adjacent vertex, or
perhaps the same, possibly winding around the annulus a few times. Up to
isotopy, the image of the arc determines the mapping class on $A$, and similarly
on all of $\Sigma$. The images of the chosen arcs are safe walks of some lengths
$l_i$, and with these walk lengths we have described $\phi$ as a multi-speed
\tat\ twist $\T{G,\underline{l}}$. Note that this process accurately recovers
the amount of ``twisting around the boundary''. If we prefer to have a \tat\
twist $\T{G',l}$ with a single walk length for all boundary components, we have
to modify $G$ as in Theorem \ref{thm:definitions}.

\proofstep{{\ref{itm:isperiodic}$\implies$\ref{itm:hasinvariantgraph}}}
Now we prove that for any periodic map, as in \ref{itm:isperiodic}, there is an
invariant filling graph.

Finding such a graph is easy once we have Nielsen's theorem cited above.
Represent the mapping class $\phi$ by a diffeomorphism $f$ which is of finite
order. $f$ will still not interchange the boundary components, but will in
general not fix them pointwise. Now choose any graph $G_0$ that fills
$\Sigma$, or is even a spine for it.
The union
	\[
		G_1 = \bigcup_{k=0}^{\ord(\phi)-1} f^k(G_0)
	\]
becomes a graph when intersection points between iterates are considered vertices.

For the sake of completeness we should ensure that $G_1$ is indeed a finite
graph. Since $f$ can be realized as an orientation-preserving isometry of some
Riemannian metric (by averaging any metric, or by using the hyperbolic metric
from Nielsen's theorem), its fixed points are isolated. Thus we can require that
$G_0$ not meet any fixed point. Likewise, points with a period smaller than the
order of $f$ are also isolated since they are fixed points of a power of $f$
which is not the identity. We choose $G_0$ to be disjoint from these as well.
In particular, the vertices of $G_0$ will be disjoint from the vertices of
$f^k(G_0)$ for all $k$ between $1$ and $\ord(f)-1$.
Therefore, possibly after a small perturbation of the edges of $G_0$, all intersections between
$G_0$ and $f^k(G_0)$ happen between interior points of edges.
Around each vertex there is a small open neighbourhood which is disjoint from
all its iterates.
Remove these neighbourhoods from $G_0$ and call the result $E_0$, a
smooth compact submanifold (with boundary) of $\Sigma$. $f^k$ leaves none of its
points fixed; therefore it is possible to move $E_0$ by a small isotopy to make
it intersect $f^k(E_0)$ transversely; see the lemma
\vpageref[below]{lem:transversality}. The edge endpoints need not be moved since
they are already disjoint from their iterates, and they can be connected back to
$E_0$ by paths that are themselves disjoint from their iterates. Since transversality is an
open condition, we can achieve transversality simultaneously for all $k$.

The new graph $G_1$ we obtain is certainly invariant under $f$, but usually not
a deformation retract.
However, it fills $\Sigma$, meaning that all its facets (the connected components of
$\Sigma \smallsetminus G_1$) that do not touch the boundary of $\Sigma$ are disks.
We see this because of two facts: First, the facets of $G_1$ are contained in
disks and boundary-parallel annuli inside $\Sigma$, namely the facets of $G_0$, which are unions
of facets of $G_1$.
And second, $G_1$ is connected.
To convince oneself of this, one can look at the edges of $G_1$ surrounding a
boundary component and see that among those one encounters edges from all
iterates since $f$ is of finite order.
Then, since $G_0$ is connected, $G_1$ is connected as well. From the two facts
we conclude that all internal facets are disks.

\proofstep{\ref{itm:hasinvariantgraph}$\implies$\ref{itm:hasinvariantspine}}
The invariant filling graph we found may be too big to be a spine, so in this
step we modify it to remove all facets apart from the boundary annuli.

\begin{figure}
  \centering
  \def\svgwidth{1\textwidth}
  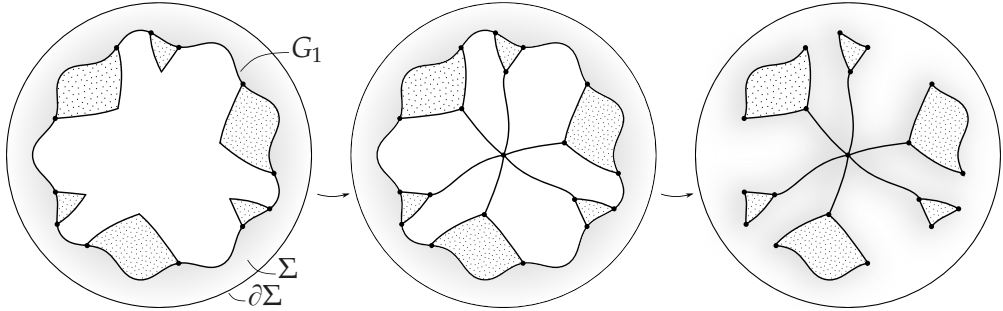
  \caption{Invariant collapse}
  \label{fig:leveebreach}
\end{figure}

The strategy is to collapse them from the boundary. This is done as in
the picture \vpageref{fig:leveebreach}, by pushing in edges that are adjacent to the
boundary annuli. The white polygon symbolizes any facet of $G_1$ that touches
one chosen boundary annulus. The rest of the graph, which may be complicated, is
symbolized by the dotted area. Call the polygon $P$. Among the edges of $P$ that
touch the boundary annulus, some may be in the same orbit under $f$, but since
$f$ leaves the boundary components invariant, the edges are not sent to any other part of $P$.
Let $i>0$ be the smallest number such that $f^i$ sends the polygon to
itself. $f^i$ acts on the boundary of $P$ as (conjugate to) a rotation. We can
modify $f$ inside $P$ by an isotopy such that $f^i|_P$ is conjugate to a
rotation (see also the note following the proof). Choose a regular Euclidean
polygon $P_E \subset \R^2$ of the same type and a diffeomorphism $\eta \colon P
\to P_E$. Assume that $\eta \circ f^i \circ \eta^{-1}$ is a Euclidean rotation.
Now we can collapse: Remove from $P$ all edges that touch the selected boundary component. Then
take a radius of $P_E$ whose preimage in $P$ goes to any of its remaining edges or
vertices. Add all its images under powers of $f$ to get a new invariant graph
with one facet less. Repeat until there are no more disk components and get an invariant
spine.
\end{proof}
\begin{note}
By Lemma \ref{lem:kerekyarto}, $f^i$ is actually already conjugate to
a rotation, so it does not even need to be modified.
However, by allowing for the modification we can avoid using the lemma if we
want and get more control on what happens on the edges of $P$.
\end{note}
We have used this transversality lemma in the proof:
\begin{lemma}\label{lem:transversality}
Let $M$ be a manifold and $A \subset M$ a compact submanifold (which may have
boundary). Let $f \colon M \to M$ be a diffeomorphism without fixed points on
$A$. Then there is diffeomorphism $h \colon M \to M$, arbitrarily close to the
identity, such that $h(A)$ and $f(h(A))$ intersect transversely. %"transversely"
%is better than "transversally" and is also used a lot in mathematics.
\end{lemma}
\begin{proof}
Since $f$ has no fixed points on $A$, each point of $A$ has a neighbourhood $U$
such that $f(U) \cap U = \varnothing$. Out of these neighbourhoods we choose a
finite subcover $(U_i)_{i=1}^{n}$. Furthermore, choose compact sets $K_i \subset
U_i$ such that the union of the interiors
$\mathring{K}_i$ still covers $A$. Using standard tranversality theory we can find a small
isotopy $H_1 \colon [0,1] \times A \to M$ with the following properties:
\begin{align*}
	H_1(0,\cdot) &= \id_A,\\
	H_1(t,\cdot)|_{A \smallsetminus U_1} &= \id_{A \smallsetminus U_1} &&\text{for all\ } t \in [0,1],
\intertext{such that, when we define $h_1 \colon A \to M$ as $h_1 = H_1(1,\cdot)$, we have that
$h_1(A) \cap K_1$ is transverse to $f(A) \cap K_1$. We require furthermore that}
	h_1(A \cap K_1) &\subset U_1.
\end{align*}

By assumption,
$f(U_1) \cap U_1 = \varnothing$,
which implies that
$f(A) \cap K_1 = f(h_1(A)) \cap K_1$.
Hence we have achieved the desired transversality locally.

We continue constructing maps $h_i$ in a similar way. Assume we already have
$h_1$ to $h_{i-1}$ such that $(h_{i-1} \circ \ldots \circ h_1)(A) \cap K_{r}$ is
transverse to $f((h_{i-1} \circ \ldots \circ h_1)(A)) \cap K_{r}$  for all $r$
between $1$ and $i-1$. Build a homotopy $H_i$ analogously to $H_1$, but taking
care to choose it small enough such that all $h_i((h_{i-1} \circ \ldots \circ
h_1)(A)) \cap K_r$ and $f((h_i(h_{i-1} \circ \ldots \circ h_1)(A)) \cap K_r$
remain transverse. This is possible since the sets $M \smallsetminus K_r$ are open and
because being transverse is an open condition. In the end, we obtain a diffeomorphism
$h = h_n \circ \ldots \circ h_1$ that fulfils the requirements of the lemma.
\end{proof}

The following lemma was stated by Kerékjártó in 1919 (\cite{Kerekjarto1919}),
but without satisfactory proof. See the article of Constantin and Kolev
(\cite{ConstantinKolev2003}) for a complete treatment. When $f$ is a diffeomorphism,
there is a quick geometric proof, given below.
\begin{lemma}\label{lem:kerekyarto}
Let $D$ be the unit disk in $\R^2$. Then any orientation-preserving
homeomorphism $f \colon D \to D$ of finite order is conjugate to a rotation.
\end{lemma}
\begin{proof}[Proof when $f$ is a diffeomorphism.]
Choose any Riemannian metric $g$ on the interior $\mathring{D}$ and average it by taking
$g'=g+f^*g+(f^2)^*g + \ldots + (f^{k-1})^*g$ where $k$ is the order of $f$. By
the uniformization theorem for Riemann surfaces, there is a conformally equivalent metric $h$ such that
$(\mathring{D},h)$ is either isometric to the hyperbolic disk $\Hy$ or to the complex plane $\C$.
Via this isometry, $f$ becomes an automorphism of $\Hy$ or $\C$. For both cases,
conformal automorphisms of finite order are conjugate to rotations about the origin.

The diffeomorphism which conjugates $f$ to a rotation is the composition of the isometry
and the conjugacy inside the automorphism group.
\end{proof}

An easy consequence of the theorem is the following proposition that has been
mentioned before:
\begin{corollary}
Let $\Sigma$ be a surface with boundary. Then the only mapping class of finite
order (in the strict sense) that fixes the boundary is the identity.
\end{corollary}
\begin{proof}
Assume $\Sigma$ is neither a disk nor a cylinder.
Since the mapping class is given by a multi-speed \tat\ twist
with nonzero walk length around a spine of the surface, some power of it
consists of Dehn twists around the boundary components.

By basic facts about Dehn twists (see e.\,g.\@ \cite[Chapter~3]{FarbMargalit2012}),
one sees that this product is of infinite order. For example, one can study its effect
on curves that live on the double of $\Sigma$, which is obtained from two copies of $\Sigma$ by
identifying the corresponding boundary components.

The cylinder is different because it has two isotopic boundary components, but
there the statement is clear since the only \tat\ twists on a cylinder are
powers of Dehn twists, which have infinite order.
\end{proof}

Looking at the proof of Theorem \ref{thm:equivalentdefs}, one sees that one
can in fact construct an invariant spine not only for the powers of a finite
order diffeomorphism, but also for any finite subgroup of the diffeomorphism
group. From this we can conclude the following (well-known) fact:
\begin{corollary}
On a connected (oriented, compact) surface with boundary or punctures, any finite subgroup of the
orientation-preserving diffeomorphism group is cyclic.
\end{corollary}
\begin{proof}
All elements of the subgroup are multi-speed \tat\ twists (with free boundary) along the same
graph. Since Dehn twists along the boundary are trivial when the isotopy can move the
boundary, they are described by walk lengths $(l_1, \ldots, l_r)$ with natural numbers $l_i$ defined
modulo $b_i$, the length of the $i$-th boundary component. When the surface is
connected, $l_2$ up to $l_r$ are already determined by $l_1$. The statement follows.
\end{proof}

Nielsen's theorem has a more general and more difficult version that was
proved by Kerckhoff using, among other things, ``earthquake maps''. It says that
in all finite subgroups of the mapping class group, we can represent mapping
classes by concrete diffeomorphisms. Nielsen's theorem says the same for
finite cyclic subgroups.
\begin{theorem}[Kerckhoff, \cite{Kerckhoff1983}]
Let $\kappa \colon \mathrm{Diff}(\Sigma) \to \mcg(\Sigma)$ be the canonical
quotient map from the diffeomorphism group of a surface $\Sigma$ to its
mapping class group. Let $G \subset \mcg(\Sigma)$ be a finite subgroup. In
that case, the restricted map $\kappa_| \colon \kappa^{-1}(G) \to G$ has
a section.
\end{theorem}
Therefore, all finite subgroups of the mapping class group of a surface
with punctures are cyclic.

As an alternative to the above proof of Theorem \ref{thm:equivalentdefs}, one
can use geometry to find an invariant graph. I owe the idea for such a proof to
Marc Lackenby.
\begin{proof}[Geometric proof]
Start with a finite subgroup of the diffeomorphism group of a surface with boundary.
If the surface is neither the disk nor the cylinder, it has negative Euler
characteristic.
By averaging and using the uniformization theorem, we are able to find a
complete hyperbolic metric such that all members of the subgroup act
as isometries, as in the proof of \ref{lem:kerekyarto}.
The boundary components become cusps of the surface.
Lift each cusp to the universal cover of the surface, where it will
be a point in the boundary the hyperbolic disk.

Around each of these lifts, choose a horoball which is small in the sense
that its projection down to the surface is still embedded and such that it does
not touch any other horoball.
Then we let all horoballs grow at constant speed.
Think of (projections of the) the horoballs as paint, one colour for each cusp,
that is poured into the white surface and spreads out smoothly.
As time passes, more of the surface is painted; the rest is still white.
At some point, one of two accidents will happen: The projection of a horoball will
fail to be embedded, or two horoballs will touch.
We assume that at places where paint arrives from two sides, it does not continue
further and does not mix.
These places will form the graph in the end.

More precisely, starting at the time of the accident, there will be a self-inter\-section
of a horocycle, or an intersection of two horocycles, respectively.
We mark all such intersections of horocycles, as long as they occur on the boundary between painted and white regions of the surface.
These markings will form a graph of the surface.
It is obvious from the construction that the surface deformation retracts to it.
Moreover, since it only depended on the hyperbolic metric (and a choice of initial
horoballs), it is invariant under the chosen subgroup.
\end{proof}
Note that resorting to geometry in the proof allows for some additional
statements to be made about the graph.
For example, we can select one of the cusps, choose a horoball there, and then
choose very small horoballs around all of the other cusps.
In that case, the paint from the selected cusp will fill the entire surface
except for the tips of the other cusps, where it meets their respective paint.
Thus we get the following corollary:
\begin{corollary}\label{cor:tatcusps}
Every (multi-speed) \tat\ twist is equivalent to (i.\,e.\@ produces the same mapping class as)
one of the following form:

There is one special boundary component which we call \emph{central}. The graph
only contains edges that meet the central boundary component on at least one side.
That is, the other boundary components meet nowhere. The cycles that surround the noncentral boundary components are
embedded polygons with bi- or trivalent vertices.\qed
\end{corollary}

We can therefore think that a general \tat\ twist is derived from a \tat\ graph
with a single boundary component, with some of its vertices, which are fixed points
for the twist, blown up to a circle; see Figure \ref{fig:thetatorus2} for an example of such a graph.

\begin{figure}
  \centering
  \def\svgwidth{0.37\textwidth}
  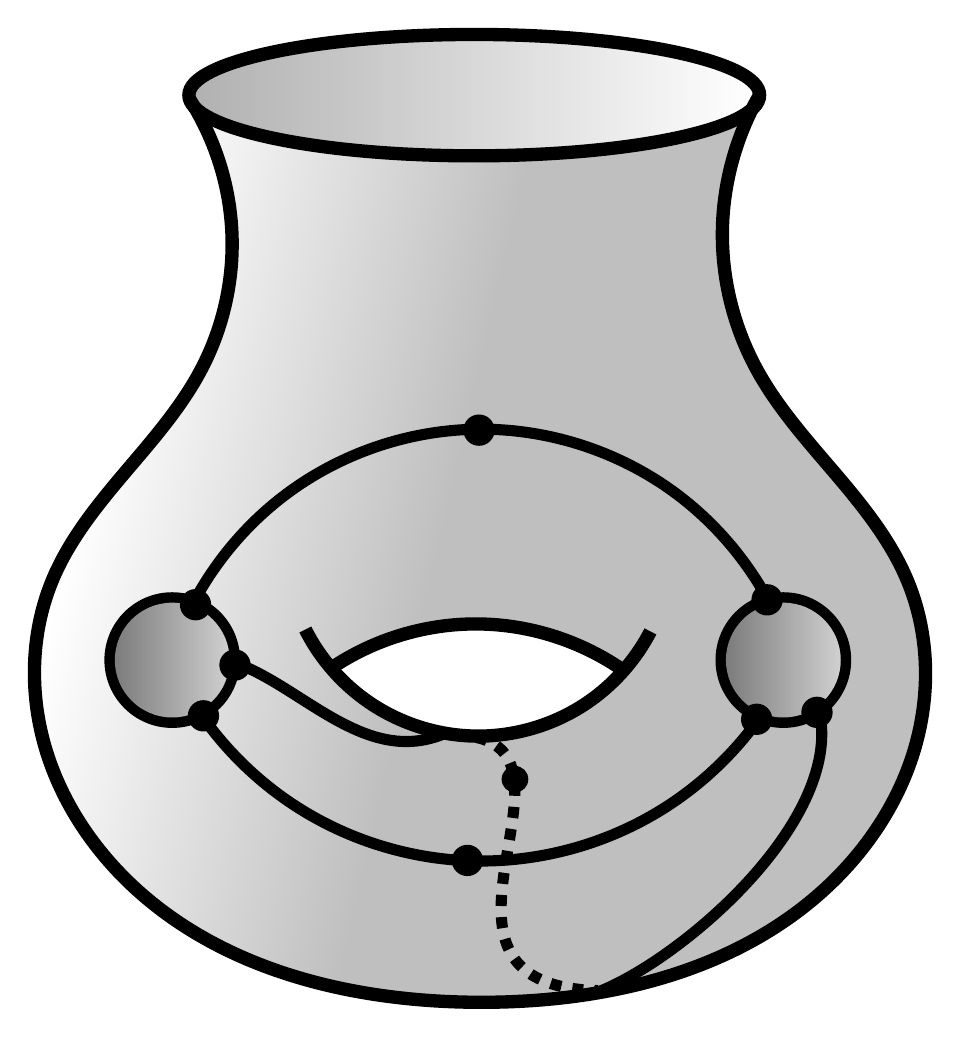
  \caption{\Tat\ graph on a 3-holed torus}
  \label{fig:thetatorus2}
\end{figure}

%==================================================
\section{Periodic diffeomorphisms on closed surfaces}
\Tat\ graphs (possibly multi-speed) can be embedded into closed surfaces; see also Chapter \ref{chap:mcg}.
If the \tat\ graph fills the surface, meaning that the surface is obtained
by capping off boundary components with disks, it induces a map of finite order.
In this case, we see a fixed point in each of these disks.
Vice versa, when a finite-order diffeomorphism has a fixed point, we can
remove an invariant disk around the fixed point to get a surface with boundary
(see e.\,g.\@ \cite{ConstantinKolev2003} for why such a disk exists, even for homeomorphisms), and in that
case it is described by a \tat\ twist.

A map of finite order can also appear when two boundary components of the \tat\ graph are glued
together.
For this to happen, the Dehn twists along the glued boundary components 
that appear in some power of the \tat\ twist must cancel themselves.

\begin{figure}
  \centering
  \def\svgwidth{0.5\textwidth}
  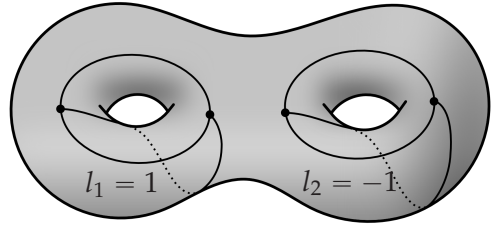
  \caption{Periodic map of order $6$, without fixed points, on a surface of genus $2$}
  \label{fig:nofixedpoint2}
\end{figure}
The same is possible when two or more \tat\ graphs are embedded disjointly such that
their boundary twists cancel;
see Figure \ref{fig:nofixedpoint2} for an example.
In that and in the former case, the diffeomorphism will not necessarily have fixed points,
but will have an invariant circle.
Again, vice versa, whenever a finite-order diffeomorphism has an invariant
circle whose two sides are not interchanged
we can cut along the circle, and the induced diffeomorphism will be described completely
by two or one \tat\ twists, depending on whether the invariant curve is
separating and essential or not.
If it is essential, the diffeomorphism is reducible according to the Nielsen-Thurston classification.

Note that when a circle is invariant, but its two sides are interchanged, then the circle itself
undergoes a reflection and there are actually two fixed points on it. 

However, we are left with the finite-order maps without invariant circles,
and it is not clear how to apply \tat\ twists to describe those as well.

%==================================================
\section{Bounds for periodic diffeomorphisms}
In 1895, Anders Wiman proved the so-called ``$4g$+$2$ theorem'': On a surface of
genus at least $2$, the order of a periodic diffeomorphism is at most $4g+2$
(\cite{Wiman1895}).
To be precise, Wiman proved the statement for automorphisms of algebraic curves
and used the branched covering structure coming from the polynomial equation.
Using the theorem above and the results about orders from Corollary
\ref{cor:tatorders}, we get a topological proof for Wiman's theorem for the case
of surfaces with at least one boundary component.

\begin{remark}
As noted in the previous section, the situation is a bit different for closed surfaces
and the proof does not apply for all diffeomorphisms.
It does apply when the diffeomorphism has a fixed point, or more generally an
invariant circle, in which case the map can be described by \tat\ twists.
When the action of the diffeomorphism is free, meaning that none of its iterates apart
from the identity has a fixed point, the order of $f$ is even smaller; see Lemma
\vref{lem:nofixedpoint}.
\end{remark}

We can copy the corollary about \tat\ twists with its more precise information
about the highest and second-highest orders and get the following ``$3g$+$3$
theorem'':
\begin{corollary}\label{cor:periodicorders}
On an orientable surface with boundary which is neither a disk, a sphere, or a torus, let $f$
be a (freely) periodic orientation-preserving diffeomorphism. Then its order is
either $4g+2$, $4g$, or
	\[
		\ord(f) \leq 
		\begin{cases}
		   3g+3, & g \equiv 0,1\\
		   3g, & g \equiv 2
 		\end{cases}
 		\pmod 3.
	\]
In all of these cases, as soon as $g \geq 4$, there exists a unique conjugacy class
of diffeomorphisms of the given order.

This conjugacy class is described by an elementary \tat\
twist $E_{n,a}$ with $(n,a)$ as in Lemma \ref{lem:elementary3g+3}.
\end{corollary}
For a closed surface, this elementary twist corresponds to the rotation of a
polygon by two clicks, with its sides glued as specified by the chord diagram.
\begin{proof}
All that is left to prove is a subtle point shown in Lemma
\vref{lem:invariantorder}: On a surface of negative Euler characteristic, as in
the corollary, two periodic diffeomorphisms that are isotopic have the same
order. This is of course wrong on the disk, the sphere and the torus where, for example, a
rotation by one third and a rotation by one quarter are isotopic. But apart
from these cases, we can start with a diffeomorphism, find an invariant spine
for it, conclude that the isotopy class of the diffeomorphism is a \tat\ twist
around that spine, and that its order is the same as the order of the \tat\
twist.
\end{proof}
\begin{lemma}\label{lem:invariantorder}
On an orientable surface which is neither the disk, the sphere, or the torus,
let $f$ and $g$ be two periodic orientation-preserving diffeomorphism which are isotopic. Then their
order is the same, meaning: If $k$,$l > 0$ are minimal such that $f^k = g^l = \id$, then $k=l$.
\end{lemma}
Note that, unlike stated in \cite[p.~200]{FarbMargalit2012}, this does not follow
from the fact that nontrivial elements of the mapping class group act nontrivially
on homology. The latter is true on the torus, but on the torus there are many
diffeomorphisms of finite order which are isotopically trivial.

\begin{remark}
Also, the following very similar statement, where ``equal'' is replaced by ``isotopic'', is trivial:

Let $f$ and $g$ be as above. Let $k$ and $l$ be minimal such that $f^k$ and $g^l$ are
isotopic to the identity. Then $k=l$.
\end{remark}
Before we prove the lemma, we prove some preliminary facts about diffeomorphisms
of finite order. In the three following statements, let $f$ be an
orientation-preserving finite-order diffeomorphism of an orientable surface.

\begin{lemma}[\cite{ConstantinKolev2003}]
Let $x$ be a fixed point of $f$ and $N$ a neighbourhood of $x$. Then there
exists a disk $D$ that contains $x$ in its interior, is contained in $N$, and is
mapped to itself: $f(D)=D$.
\end{lemma}
\begin{proof}
See \cite{ConstantinKolev2003} for the proof, which uses the Jordan-Schoenflies theorem.
\end{proof}

\begin{lemma}
If $f \neq \id$, its fixed points are isolated.
\end{lemma}
\begin{proof}
Around a fixed point, choose an invariant disk $D$ as in the previous lemma. By
the lemma of Kerékjártó (Lemma \ref{lem:kerekyarto}), $f$ acts on $D$ by
rotation.
\end{proof}

\begin{lemma}
The fixed-point index of $f$ at every fixed point is $+1$.
\end{lemma}
\begin{proof}
This follows from the previous lemma. Alternatively, do the following: Around a
fixed point $x$, choose again a small invariant disk $D \ni x$ which we imagine
inside a local chart. The fixed-point index measures the rotation of the vector
$y - f(y)$ while $y$ moves along a small simple closed curve around $x$, for
which we take $\partial D$. Since $f(\partial D) = \partial D$, without fixed
points on $\partial D$, the fixed-point index is $+1$.
\end{proof}

Lemma \ref{lem:invariantorder} is a consequence of the following statement:
\begin{lemma}
Let $h$ be an orientation-preserving finite-order diffeomorphism of a compact
surface $\Sigma$ of negative Euler characteristic. Assume that $h \simeq \id$.
Then $h = \id$.
\end{lemma}
\begin{proof}
Assume that $f \neq \id$. Because the fixed points of $f$ are isolated, there
are only finitely many of them. Since $h \simeq \id$, the Lefschetz number
$\Lambda(f)$ satisfies $\Lambda(f) = \Lambda(\id) = \chi(\Sigma) <0$.
By the Lefschetz fixed point formula, $f$ would have a fixed point of negative
index, which is a contradiction.
\end{proof}

\begin{proof}[Proof of Lemma \ref{lem:invariantorder}.]
Assume that $k > l$ and let $h = f^{k-l}$. Then, by the remark, $h = f^k f^{-l}
\simeq f^k g^{-l} = f^k = \id$, but $h \neq \id$. And also $h^k = (f^{k-l})^k =
(f^k)^{k-l} = \id$. By the previous lemma, this is impossible.
\end{proof}

\subsubsection{Diffeomorphisms on closed surfaces that act freely}
As mentioned above, the corollary applies to all diffeomorphisms that
have an invariant circle, even if they have no fixed points.
An example of such a diffeomorphism is
drawn in Figure \ref{fig:nofixedpoint}: Take the depicted surface of genus $g$
and rotate it around the central hole to get a diffeomorphism without fixed
points of order $g-1$.
\begin{figure}
  \centering
  \def\svgwidth{0.4\textwidth}
  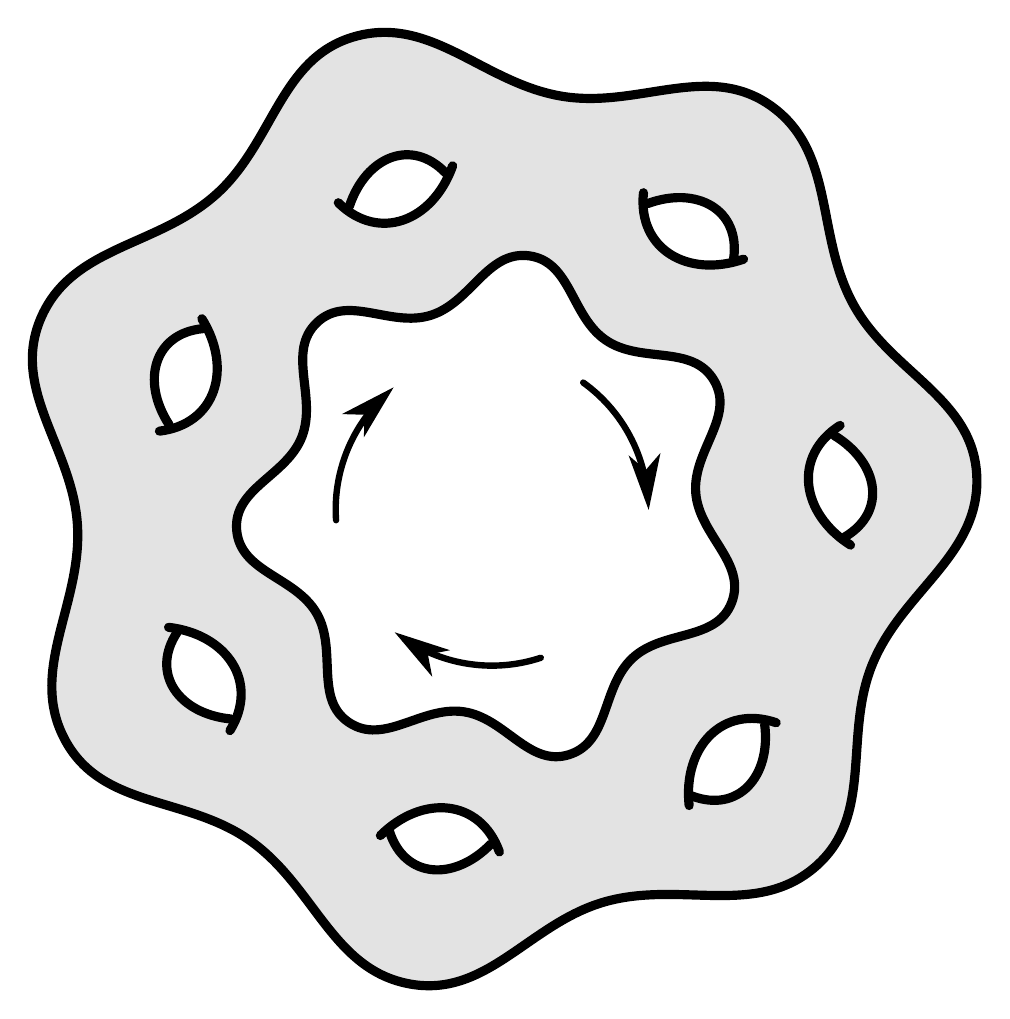
  \caption{A diffeomorphism (of order 7) without fixed points}
  \label{fig:nofixedpoint}
\end{figure}
This is actually the highest possible order if $f$ has no points whose orbit
is smaller than the order of $f$:
\begin{lemma}\label{lem:nofixedpoint}
Let $\Sigma$ be a closed orientable surface of genus $g \geq 2$ and let $f$ be
an orientation-preserving periodic diffeomorphism that acts freely on the surface.
Then the order of $f$ is at most $g-1$.
\end{lemma}
\begin{proof}
When $f$ is periodic of order $k$ and acting freely, $f$ induces a covering
\[
	\pi \colon \Sigma \to \faktor{\Sigma}{f}.
\]
For the Euler characteristic, we have
\[
	\chi\left(\faktor{\Sigma}{f}\right) = \frac{1}{k} \, \chi(\Sigma).
\]
Since the Euler characteristic of $\Sigma$ is negative by assumption, i.\,e.\@
smaller than $-2$, the same is true for $\faktor{\Sigma}{f}$ and hence
\[
	2 \leq \abs[\Big]{\chi\left(\faktor{\Sigma}{f}\right)} = \abs[\Big]{\frac{1}{k} \, \chi(\Sigma)},
\]
which implies that $k \leq \frac{1}{2}\abs{\chi(\Sigma)} = g-1$.
\end{proof}

%==================================================
\chapter{Tête-à-tête twists as monodromies}\label{chap:monodromies}
This chapter shows \tat\ twists in action as monodromies of fibred knots, both in
the 3-sphere as well as in other manifolds.

%==================================================
\section{Open books, fibred links and monodromy}\label{sec:openbooks}
We start with some definitions that are standard in low-dimensional topology; first
that of an open book decomposition. It provides a very fruitful connection
between mapping classes that fix the boundary of a surface and 3-manifolds.
\begin{definition}
Let $M$ be a manifold. An \emph{open book decomposition} of $M$ is a pair $(L,
\pi)$, where $L \subset M$ is a link in $M$ called the binding, and $\pi\colon M
\smallsetminus L \to S^1$ is a fibre bundle map. The fibres (usually called
\emph{pages} of the open book) are open orientable surfaces. Their closures are
homeomorphic to a fixed compact surface $\Sigma$ and have $L$ as their boundary.
\end{definition}

An open book decomposition comes with a mapping class $\phi \in \mcg(\Sigma)$,
the \emph{monodromy (diffeomorphism)}, that fixes the boundary of $\Sigma$. It
can be constructed by choosing a smooth vector field on M, transverse to the
pages, that on $L$ is zero and on $M \smallsetminus L$ projects to the vector
$\partial\theta$ on $S^1$, which here denotes the unit circle in $\C$. When we identify
$\Sigma$ with the closure of the fibre $\pi^{-1}(1)$ and follow the flow of the vector field
for time $2\pi$, we get a diffeomorphism of $\Sigma$. Any two such vector fields
are isotopic, therefore the monodromy $\phi$ is well-defined up to isotopy.

A link which is the binding of some open book decomposition is called a
\emph{fibred link}, especially if the manifold is the 3-sphere.

We will need some notation for the rest of the fibres as well: For $\theta \in S^1$,
denote by $\Sigma_\theta = \overline{\pi^{-1}(\theta)}$ the closure of the fibre over $\theta$, and for
$t \in [0,1]$ such that $\theta = \exp(2 \pi i t)$, denote by $\Phi_t \colon
\Sigma = \Sigma_1 \to \Sigma_\theta$ the diffeomorphism given by the flow of the vector field for time $t$.

%==================================================
\subsection{Bookbinding}
There is a tendency to speak of ``open book decompositions'' in the above sense,
but just of \emph{open books} when the same object is described by different
data. Namely, instead of the triple $(M, L, \pi)$, we specify the compact
orientable surface $\Sigma$ together with a mapping class $\phi \in
\mcg(\Sigma)$. From this, we can construct a 3-manifold $M = M_{(\Sigma,\phi)}$ with an
open book decomposition whose fibre is $\Sigma$ and whose monodromy is $\phi$:

Take $\Sigma \times [0,1]$ and identify $\Sigma \times \{1\}$ with $\Sigma
\times \{0\}$ by sending $\{p,1\}$ to $\{\phi(p),0\}$. This produces the
\emph{mapping torus} of $\phi$, whose boundary is given the structure of a
trivial circle bundle over $\partial\Sigma$ by this construction since $\phi$ is
the identity on $\partial\Sigma$. Now collapse this boundary to circles by
identifying all $\{q,s\}$ with $\{q,s'\}$ where $q$ is in $\partial\Sigma$ and
$s$, $s'$ are in $S^1$. Alternatively, one can fill the boundary components with
full tori such that the fibre circles are contractible. In any way, we get a
closed 3-manifold $M$ together with a link $L$ -- the contracted boundary or the
souls of the glued tori, respectively -- and a fibration of $M \smallsetminus L$ over
$S^1$ with fibre $\Sigma$, in other words an open book with monodromy $\phi$,
illustrated by Figure \ref{fig:openbookconstruction}.
\begin{remark}
If the binding is connected, it is always a homologically trivial knot since
it bounds a surface.
\end{remark}
\begin{figure}
\centering
\def\svgwidth{0.5\textwidth}
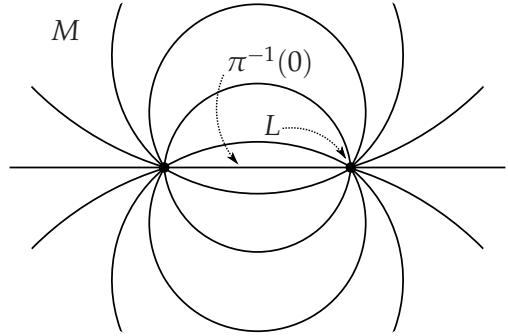
  \caption[The fibres of an open book decomposition of the 2-sphere]%
  {The fibres of an open book decomposition of the 2-sphere.
  For a 3-manifold, the lines represent surfaces and $L$ is a link.}
  \label{fig:openbookconstruction}
\end{figure}

%==================================================
\section{Seifert manifolds}
A closed \emph{Seifert manifold} is a closed 3-manifold that is foliated by
circles, with the additional requirement that every leaf of the foliation has a
neighbourhood which is leaves-preserving diffeomorphic to a \emph{standard
fibred torus}. A Seifert fibration need not be a true fibration in the usual
sense; the name is used nevertheless and the leaves are usually called fibres.

A \emph{standard fibred torus} is obtained from a solid cylinder $D^2 \times
\left[0,1\right]$ by gluing top to bottom by some rational rotation. That is to
say, given two coprime integers $a$ and $b$, $a \geq 1$, we identify $(z,1)$
with $(e^{\frac{2 \pi i b}{a}}z,0)$. The vertical lines of the cylinder,
$\left\{(z,t) \mid t \in \left[0,1\right] \right\}$ where $z \in D^2$ is fixed,
become circles under this gluing. If $a$ is not one, the middle fibre
$\left\{(0,t) \mid t \in \left[0,1\right] \right\}$ (which is somehow
``shorter'') is called a \emph{singular fibre}.

Seifert manifolds are allowed to have boundary, also fibred by circles, hence
consisting of tori.

%==================================================
\subsection{\Tat\ twists produce Seifert manifolds}
As we have seen, a \tat\ twist is of finite order, in the sense that some
$k^\textrm{th}$ power of it is isotopic to a Dehn twist along the boundary of
the surface. We have also seen how one can represent this periodic map by an
actual diffeomorphism which is periodic on the Seifert surface minus a small
neighbourhood of the boundary. This implies:
\begin{proposition}\label{prop:seifertcomplement}
The open book produced by a \tat\ twist with nonzero walk length is a Seifert
manifold.
\end{proposition}
\begin{proof}
Represent the \tat{} twist by an actual finite-order diffeomorphism $f$ defined on
$\mathring\Sigma$, which is $\Sigma$ minus a small tubular neighbourhood of its
boundary.
Then for every point $p \in \mathring\Sigma$, $\{p\} \times [0,1]$ is glued to $\{f(p)\}
\times [0,1]$ in the open book $M_\phi$, then to $\{f^2(p)\} \times [0,1]$, and
so on, until it closes up to a circle inside $M_\phi$.
The only points on $\mathring\Sigma$ which are possibly of lower order are the
vertices of the \tat{} graph, where we get singular fibres.

This makes the complement of the binding a Seifert manifold. Since the walk length
is not zero, it is possible to extend the Seifert structure to the solid tori around
the binding. If the walk length were zero, however, the circles of the Seifert
fibration would bound disks in these solid tori, which would make the extension
impossible.
\end{proof}
Later in this chapter (see Section \ref{sec:tatswithopenbooks}), we will study
the manifolds that arise as open books for \tat\ twists a little further.

%==================================================
\section{Fibred knots with \tat\ monodromies}
In this section, we study fibred knots and links in $S^3$.

%==================================================
\subsection{Trivial monodromies}\label{sec:trivialmonodromy}
Let $M$ be an irreducible 3-manifold, $L \subset M$ a fibred link and, as
before, $(\Sigma_\theta)_{\theta \in S^1}$ a family of (closures of) fibres, $L = \partial
\Sigma_\theta$, and
	\[
		\left(\Phi_t \colon \Sigma_0 \to \Sigma_{\exp(2 \pi i t)}\right)_{t \in
[0,1]}
	\]
a smooth family of diffeomorphisms such that the mapping class $\phi$ of
$\Phi_1$ is the monodromy. We abbreviate $\Sigma_1$ as $\Sigma$. Then we have
the following:

\begin{lemma}
Let $\gamma \subset \Sigma$ be a properly embedded arc such that $\phi(\gamma)$
is isotopic to $\gamma$. Then $\gamma$ is separating.
\end{lemma}
\begin{proof}
Regardless of the assumption, the monodromy family gives rise to a disk $D =
\bigcup_{t \in [0,1]} \Phi_t(\gamma)$ whose interior is embedded
in $M \smallsetminus \Sigma$. Assuming now that $\phi(\gamma)$ is isotopic to
$\gamma$, we can arrange $D$ to be an embedded sphere that intersects $\Sigma$
in $\gamma$ only. Since $M$ is irreducible, the sphere, and hence $\gamma$,
separate $\Sigma$.
\end{proof}
\begin{corollary}
The only link in $S^3$ with trivial monodromy is the unknot.
\end{corollary}

Like always, it is important to be aware of the type of monodromy we study. For
example, Eisenbud and Neumann give us in their book~\cite{EisenbudNeumann1985} a
list of ``links with trivial geometric monodromy''. One example is in the
introduction, its fibre surface is a knotted three-holed sphere. However,
constructing the open book over the three-holed sphere with trivial monodromy
produces the manifold $S^2 \times S^1 \# S^2 \times S^1$. We can see this by
looking at the surface cross $S^1$ and first collapsing one of the boundaries to
a circle. The resulting manifold is the 3-sphere with two full tori removed
along a two-component unlink. Collapsing these two boundaries as well is
equivalent to gluing in two full tori whose meridian goes along the canonical
longitudes (a $0$-Dehn filling).

So why is this not a contradiction? The point here is simply that the monodromy
in these examples is only isotopically trivial if we do not require the boundary
to be fixed during the isotopy. This determines the link complement, but does
not say much about the open book as a whole. In fact the Hopf link would be the
simplest nontrivial example of this kind as its monodromy is a Dehn twist which
is trivial in that sense.

%==================================================
\subsection{Knots with \tat\ monodromy}
The monodromy of torus knots were the first examples of \tat\ twists that A'Campo
considered. Since there are various modifications to \tat\ graphs which produce
new, more complicated, \tat\ graphs, one can wonder what other knot monodromies
can be described by them. But as it turns out, fibred knots with \tat\ monodromies
are precisely the torus knots:
\begin{theorem*}
Let\/ $K$ be a fibred knot whose monodromy is represented by a \tat\ twist.
Then\/ $K$ is a torus knot.
\end{theorem*}

In fact one can say:
\begin{theorem}\label{thm:finmon}
Let\/ $K$ be a fibred knot whose monodromy is of finite order, i.\,e.\@ has a power
which is a product of Dehn twists along the boundary of the fibre surface.
Then\/ $K$ is a torus knot.
\end{theorem}
It was not obvious where to find this result in the literature; but it was, for example, stated by Burde and Zieschang (\cite{BurdeZieschang1966}).
The theorem is a direct consequence of the following theorem by Seifert, which he
proved in his second dissertation (\foreignlang[Topology
of three-dimensional fibred spaces]{Topologie dreidimensionaler gefaserter
Räume}{german}, \cite{Seifert1933}) where he founded the theory of Seifert fibred spaces:
\begin{theorem}
Any fibre of a Seifert fibration of the 3-sphere is a torus knot.
\end{theorem}
\begin{proof}[Sketch of the proof]
In chapter 11, Seifert classifies all possible Seifert fibrations of the
3-sphere. The result follows from this classification.

First, one can prove that the space of fibres of a Seifert fibred 3-manifold
is a surface (\foreignlang[decomposition surface]{Zerlegungsfläche}{german}),
the orbit surface, which is closed if the manifold is closed. It comes equipped
with a projection map, the continuous map which maps a point of the manifold to
the point representing the fibre it lies on.
Note that, in general, an orbit surface cannot be seen as a surface that
lies inside the 3-manifold.

We can lift any path on the orbit surface to a path in the manifold.
A homotopy of the lifted path projects to a homotopy on the orbit surface.
Therefore the orbit surface of a simply-connected manifold is also simply
connected, hence a sphere in the case of the 3-sphere.

An important part of Seifert's text is the classification of Seifert fibred spaces.
He defines a set of invariants (up to elementary modifications)
\[\label{eq:Seifert}
	(\textrm{\textsl{type}}; \textrm{\textsl{surface}} \mid b;
	(\alpha_1, \beta_1); \ldots ; (\alpha_r, \beta_r))
\]
where \textsl{surface} is the genus or number of cross-caps of the orbit surface and
\textsl{type} is some information about orientation that can take one of six
possible values. The pairs $(\alpha_i, \beta_i)$ describe the exceptional
fibres. If one bores out all exceptional fibres and replaces them by regular ones,
one gets a circle bundle over the orbit surface. The invariant $b$, like the Euler class,
distinguishes between the different circle bundles over the given surface. Seifert shows
that these invariants completely determine the Seifert fibre space up to isomorphism.

In chapter 10, Seifert derives a presentation for the fundamental group of a
closed Seifert fibred 3-manifold from the invariants. Hence in the case of the
sphere with $n$ exceptional fibres, everything is encoded in $b$ plus $n$
rational numbers.

Using his presentation, Seifert notes that some quotient of the fundamental
group is a polygon group, that is to say, the symmetry group of a
black-and-white tessellation by $n$-gons. If $n\geq 4$, this tessellation
necessarily lives in Euclidean or hyperbolic space, hence the group is infinite.
For $n=3$, there is a finite number of families of finite groups corresponding
to tessellations of the sphere. But if the group is to be trivial, $n$ must be at
most $2$.

So there are at most two exceptional fibres, and Seifert describes such
fibrations in chapter 3. They are given by the orbits of the rotation
\[
 \begin{pmatrix}
   \cos(mt)  & \sin(mt) &           &          \\
   -\sin(mt) & \cos(mt) &           &          \\
             &          & \cos(nt)  & \sin(nt) \\
             &          & -\sin(nt) & \cos(nt)
  \end{pmatrix},
  \quad\gcd(m,n)=1,
\]
of $S^3$, seen as the unit sphere in $R^4$. Since the data described above
completely determines the manifold up to fibre-preserving homeomorphism, we now
know all possible Seifert fibrations of the 3-sphere. The exceptional fibres
are always unknotted here, therefore the regular fibres are torus knots.
\end{proof}

The following example shows that Theorem \ref{thm:finmon} does not hold for
links:
\begin{example}
The link of the singularity $x(y^2-x^4)$ is not a torus link, but its monodromy
is of finite order.
\end{example}
The link in question is
\[
   L = \left\{ (x,y)\in\C^2 \mid x(y^2-x^4)=0, \abs{x}^2+\abs{y}^2=1 \right\}.
\]
Its genus is $2$ and it has $3$ components. The genus of the $(p,q)$-torus link
$T(p,q)$ is calculated by
\(
   g = \frac{1}{2}((p-1)(q-1)-d+1),
\)
where $d = \gcd(p,q)$, which must be $3$ in this case. Therefore, $L$ cannot be a
torus link because the genus of $T(3,3)$ is $1$ and the genus of $T(3,6)$ is
already $4$. $L$ does, however, contain a $(2,4)$-torus link due to the factor
$y^2-x^4$, in fact it has a diagram as in Figure \ref{fig:finiteorderlink}.
\begin{figure}
\centering
\def\svgwidth{0.3\textwidth}
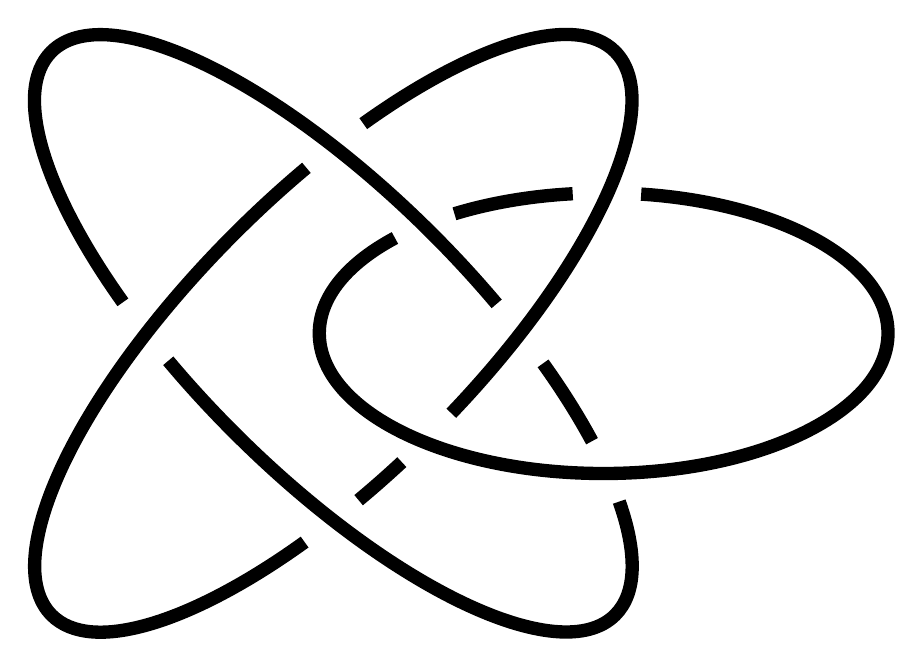
  \caption{The link of the singularity $x(y^2-x^4)$}
  \label{fig:finiteorderlink}
\end{figure}

Setting
\[
F(x,y) = \frac{f(x,y)}{\abs{f(x,y)}}
\]
with $f(x,y) = x(y^2-x^4)$, we get the projection of a fibration
${S^3 \smallsetminus L \xrightarrow{F} S^1}$.  %or $S^3 \smallsetminus L \overset{F}\to S^1$
Let us describe the monodromy of L as a diffeomorphism of $S = F^{-1}(1)$. First note that
$f\left(e^{\sfrac{2 \pi it}{5}}x, e^{\sfrac{4 \pi it}{5}}y\right) = e^{2\pi it} f(x,y)$,
so whenever a point $(x_0,y_0)$ is in $S$, then
$f\left(e^{\sfrac{2 \pi it}{5}}x_0, e^{\sfrac{4 \pi it}{5}}y_0\right) = e^{2\pi it}$
for all $t \in \R$. Thus we get an isotopy
$h\colon S \times \left[0,1\right] \to S^3 \smallsetminus L$
with $F(h_t(x_0,y_0)) = e^{2\pi it}$, and a self-map $h_1\colon S \to S$.

\begin{figure}
\centering
\def\svgwidth{0.5\textwidth}
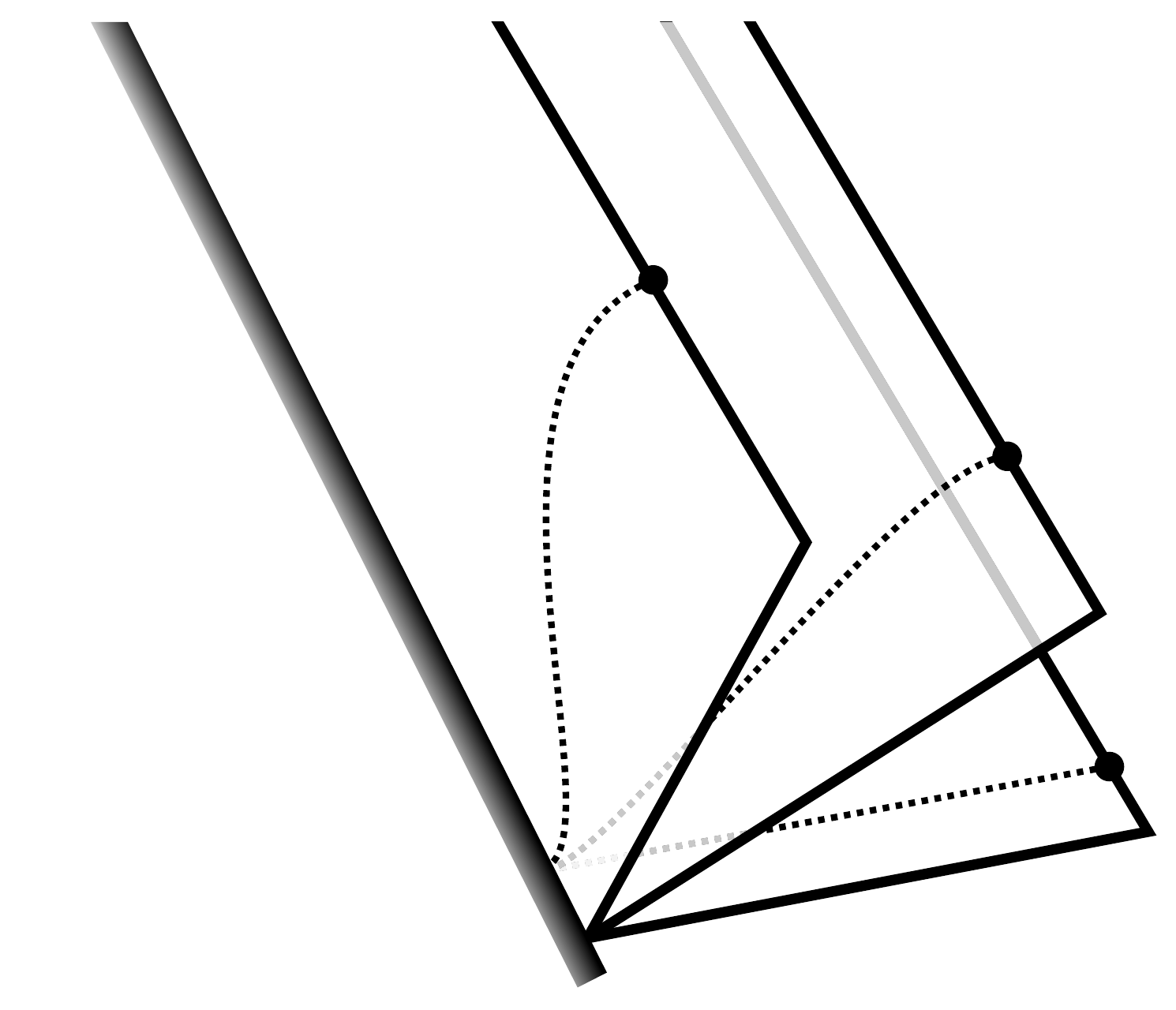
  \caption[How to slow down a monodromy near the boundary]%
  {How to slow down a monodromy near the boundary using a smooth family
           of intervals that connect to a fixed point of the boundary}
  \label{fig:openbookslowdown}
\end{figure}
$h_1$ is not exactly the monodromy because it does not extend to the identity on
$\partial \bar{S}$. To fix this, we can modify the isotopy in a small collar
neighbourhood of the boundary and smoothly interpolate it to the identity on
$\partial \bar{S}$. By this modification, we get the monodromy $\phi$ from
$t=1$. Because $h_1^5 = \id$, $\phi^5$ is the product of some Dehn
twists along the boundary of $\bar{S}$.

$f$ is an example of a \emph{quasihomogeneous polynomial}. Being
quasihomogeneous is exactly the property we used above, namely that there
are weights $\omega_1$, $\omega_2$ and $\omega$ in $\Z$ such that for every
$\lambda \in \C$ we have
	\[
		f(\lambda^{\omega_1} x, \lambda^{\omega_2} y) = \lambda^{\omega} f(x,y).
	\]
Since all of these have periodic monodromy, the result from Section \ref{sec:periodic}
applies and we get:
\begin{corollary}
Let $\phi$ be the monodromy of a quasihomogeneous polynomial in $\C[x,y]$. Then
$\phi$ is given by a \tat\ twist.
\end{corollary}

%==================================================
\section{{\Tat}s with open books}\label{sec:tatswithopenbooks}
As we have seen in the previous section, only few \tat\ twists are monodromies
of knots or links, by which we meant knots or links in $S^3$. But one can of
course ask whether other \tat\ twists could be monodromies of knots or links in
different 3-manifolds, and if yes, in which.

The answer to the first question is trivially ``yes'', as we have seen in the beginning
of this chapter: To every mapping class fixing the boundary of a surface one can construct the open book for it,
making it into a monodromy. However, asking what manifolds arise leads to plenty of
interesting examples.

Alexander has shown in 1923 (\cite{Alexander1923}) that every closed orientable
3-manifold can be equipped with an open-book decomposition and can thus be
obtained by constructing the open book corresponding to a diffeomorphism of a
surface; so a priori this is no restriction on the type of manifolds produced
by \tat\ twists. The open book can even be chosen to have connected boundary, as
shown by González Acuña (\cite{GonzalezAcuna1975}) and Myers (\cite{Myers1978}).

%==================================================
\subsection{Seifert symbols}
We have, however, the strong restriction of Proposition
\ref{prop:seifertcomplement} above:
The open book produced by a \tat{} twist (with nonzero walk length) is a
Seifert manifold.
One could also write down the Seifert symbol (as shown \vpageref{eq:Seifert}) from
the \tat\ graph:
There can be one exceptional fibre for each boundary component, where different
coefficients correspond to different Dehn fillings of the mapping torus of the
diffeomorphism.
Also at each vertex an exceptional fibre can occur.
Their coefficients depend on the amount of rotation that occurs when the vertex
it is mapped to itself by a (minimal) power of the diffeomorphism. We will not
pursue this course in the following section, but rather study the fundamental
group of the open book that we construct.
However, note the following:
\begin{proposition}
The open book belonging to an elementary \tat\ twist $\T{E_{n,a},2}$, for $a<n$,
is a Seifert fibred manifold with base $S^2$ and at most 3 exceptional fibres.
\end{proposition}
\begin{proof}
Since the \tat\ twist acts transitively on edges, the quotient of the graph by
the twist consists of a single edge that connects two vertices, one corresponding
to ``inner'' and one to ``outer'' vertices of the chord diagram (see the remark
\vpageref{rem:bipartite}).
The base manifold then consists of a disk which is the thickening of this edge, with possibly
exceptional fibres at the two vertices, and a disk that represents the Dehn filling,
with another possible exceptional fibre in its middle.
\end{proof}

%==================================================
\subsection{A presentation of the fundamental group}\label{sec:fundamentalgroup}
In this section we will see how to find a presentation for the fundamental group
of open books coming from \tat\ twists and how to use this to recognize some of
the manifolds constructed in this way.

When we build a 3-manifold $M_\phi$ from a surface $\Sigma$ and a diffeomorphism
$\phi$ as an open book, we can use the theorem of Seifert and van Kampen to
calculate its fundamental group. To do this, we split $M_\phi$ into two parts by
removing a closed surface, the double of $\Sigma$, given by
	\[
		D\Sigma =
			\pi^{-1}(0) \,\cup\, \pi^{-1}(\sfrac{1}{2}) \,\cup\, L,
	\]
where $L$ is the binding of the open book decomposition of $M_\phi$ and
$\pi\colon M_\phi\smallsetminus L \to S^1=\R/\Z$ is its projection map.
$M_\phi \smallsetminus D\Sigma$ then consists of the two parts
$\pi^{-1}\big((0,\sfrac{1}{2})\big)$ and
$\pi^{-1}\big((\sfrac{1}{2},1)\big)$, whose closure is in each case
homeomorphic to $\Sigma \times [0,1]$ and whose fundamental group is
thus just $\pi_1(\Sigma)$.
The theorem of Seifert and van Kampen now gives us $\pi_1(M_\phi)$ as an
amalgamated product of the form $\pi_1(\Sigma) *_{\pi_1(D\Sigma)}
\pi_1(\Sigma)$.

When the boundary of $\Sigma$ is connected, we can explicitly write down generators and
relations for the fundamental group.

The presentation we use here depends on the choice of a basepoint on the boundary,
or a choice of a ``first'' -- or rather ``zeroth'' -- endpoint in the chord diagram.
We then label the endpoints from $0$ to $2n-1$.
Recall the construction of the ribbon graph from a chord diagram described in
the second paragraph of Section \ref{sec:chorddiagramsandribbongraphs}:
Starting with an annulus, we glue a band for every chord,
and a disk for every internal boundary component of the diagram.

We use  one generator for each chord, which
corresponds to starting at the basepoint, walking anticlockwise to the first endpoint of the chord, traversing it, and walking back clockwise to the basepoint.
When the chord from endpoint $i$ to endpoint $j$, $j>i$, is traversed, we call this generator $c_i$.
There is one additional generator, $\omega$, which corresponds to walking once around the whole annulus in anticlockwise direction.

Each internal boundary component provides one relation, namely the
product of $c_i$'s and $c_i^{-1}$'s that gets trivial by gluing in the corresponding disk.
When, as we walk along one internal boundary, we pass by the basepoint, we
add $\omega$ to the product.
Using these relations, we get a presentation of $\pi_1(\Sigma)$.

The group $\pi_1(M)$ has the additional relations that come from gluing the two
copies of $\Sigma \times [0,1]$. Those actually have their boundaries pinched so
they look like in Figure \ref{fig:lens_thickening} on page \pageref{fig:lens_thickening} in the next chapter. \checkhere
We glue using $\phi$ on one side and the identity on the other.
The additional relations just state that a generator $c_i$ becomes equal to its
image under $\phi$.
When $l$ is the walk length of the twist, let
$i' = i + l \mod 2n$
and
$j' = j + l \mod 2n$,
understood as numbers in $\{0, \ldots, 2n-1\}$.
Let
$r_i = \big\lfloor\tfrac{i + l}{2n}\big\rfloor$
and
$r_j = \big\lfloor\tfrac{j + l}{2n}\big\rfloor$,
which denote the number of times the endpoints $i$ and $j$ are rotated past the
basepoint.
The image of $c_i$ is then $\omega^{r_i}c_{i'}\omega^{r_j}$ if $i' < j'$,
and $\omega^{r_i}c_{j'}^{-1}\omega^{r_j}$ otherwise.
\begin{example}
\begin{figure}
\centering
\def\svgwidth{0.4\textwidth}
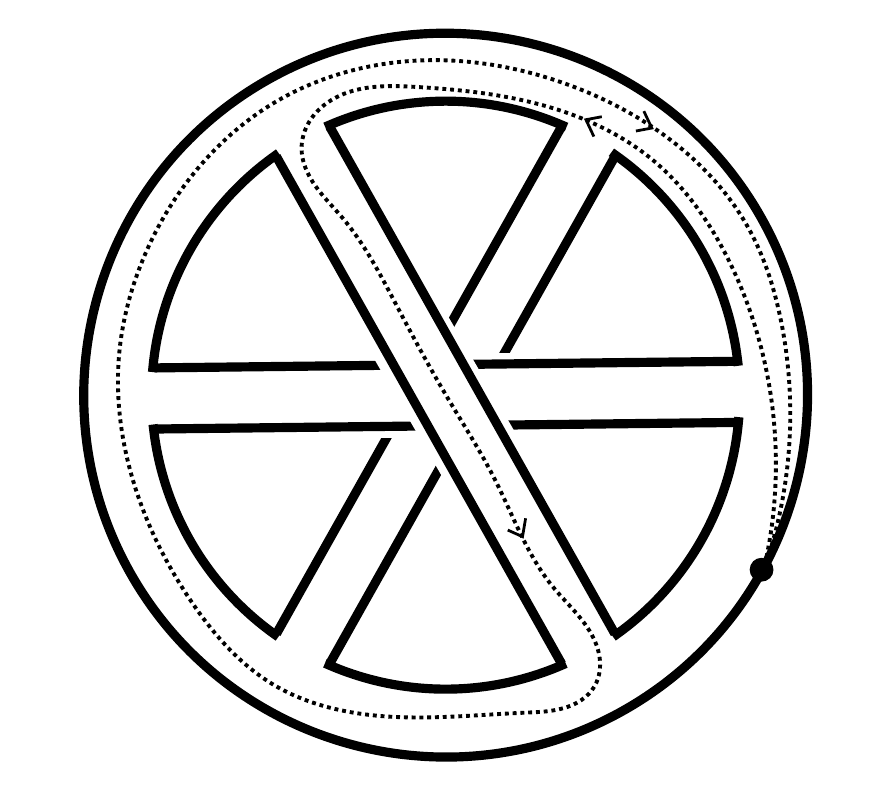
  \caption{The generator $c_2$ in the chord diagram $E_{3,3}$}
  \label{fig:fundamentalgroup}
\end{figure}
We calculate the fundamental group for the now familiar example of $E_{3,3}$
with walk length $1$, which should be trivial since this twist represents the
monodromy of the trefoil. Figure \ref{fig:fundamentalgroup} shows the chord
diagram, with chords replaced by bands, together with the basepoint and one
generator.

Writing down the generators and relations as described before, with
$\Tr = \T{E_{3,3},1}$,
we get that
\[
	\pi_1(M_{\Tr}) = \langle c_0, c_1, c_2, \omega \mid R_b \cup R_m \rangle,
\]
where $R_b$, the relations given by the internal boundaries of the chord diagram
(or the vertices of the graph), is
\begin{align*}
	R_b &= \left\{c_1^{} c_2^{-1} c_0^{-1}, c_0^{} c_1^{-1} c_2^{} \omega \right\},\\
	\intertext{and $R_m$, the relations given by the gluing map, is}
	R_m &= \left\{c_0 = c_1, c_1 = c_2, c_2 = c_0^{-1} \omega^{-1}\right\},
\end{align*}
with the obvious abuse of notation that $c_0 = c_1$ actually means $c_0 c_1^{-1}$.

From the relations in $R_m$ we immediately see that all $c_i$ are equal
to $c_1$, hence trivial by the first relation in $R_b$, and $\omega$ is
trivial as well. We therefore obtain the trivial group and, by the Poincaré
conjecture, the manifold is the 3-sphere, as we already knew.
\end{example}

\subsubsection{Homology of open books}
Using the Mayer-Vietoris sequence for the same pair that we used in the
calculation of the fundamental group --
$A = \Sigma \times \left[0,\sfrac{1}{2}\right]$
and
$B = \Sigma \times \left[\sfrac{1}{2},1\right]$
with intersection $D\Sigma$
-- we get that
\[
 H_1(M_\phi)
 	\cong
 \faktor{H_1(A)\oplus H_1(B)}{\ker(\iota_*-\kappa_*)}
 	=
 \faktor{H_1(A) \oplus H_1(B)}{\im(\mu_*,\nu_*)}
\]
where
$\iota\colon A \hookrightarrow M_\phi$
and
$\kappa\colon B \hookrightarrow M_\phi$
as well as
$\mu\colon D\Sigma \hookrightarrow A$
and
$\nu\colon D\Sigma \hookrightarrow B$
are the respective inclusion maps.
This is standard and corresponds to the abelianized version of the theorem of
Seifert and van Kampen.

In our case, we also get the isomorphism
\[
   H_1(M_\phi) \cong \ker(\mu_*,\nu_*).
\]
This map works as follows: If a cycle in $H_1(D\Sigma)$ is trivial in both
$H_1(A)$ and $H_1(B)$, it is in each one the boundary of a two-chain, which
combine to a two-cycle in $M_\phi$. Poincaré duality gives us then an element in
$H_1(M_\phi)$.

%==================================================
\subsection{Examples of open books given by \tat\ twists}
Using the above presentation of the fundamental group, we can calculate
some examples using the computer algebra system GAP.\footnote{\code{http://www.gap-system.org}}
Chapter~\ref{chap:software} will treat a computer program written in Java that produces
ready-made GAP code that describes the fundamental group for a given
\tat\ twist with one boundary component.
GAP provides many commands to examine the group.
Sometimes, the commands \code{IdGroup} or \code{StructureDescription}
will be able to identify the group in a list and output a description.
All the examples in this section (apart from three which have more than one boundary
component) have been found using a computer;
when there is a theoretical justification for the outcome, this
confirms the validity of the program.

\subsubsection{Spherical manifolds}
The GAP command \code{IsFinite} is sometimes able to tell whether the group
is finite.
If it is, we know that the manifold is a so-called
$\emph{spherical manifold}$, that is, that the fundamental group acts faithfully
by isometry on the 3-sphere and that the manifold is homeomorphic
to the quotient of the 3-sphere by this action. 
This statement was Thurston's \emph{elliptization conjecture}, a consequence
of the geometrization conjecture now proved by Perelman.
Unless the fundamental group is cyclic, there is a unique such quotient.

If it is cyclic, however, the manifold is a lens space $L(p;q)$ and there is a
family of them that share the same fundamental group.
But it is known which lens spaces are homeomorphic.
For example, all lens spaces of the form $L(2;q)$ are homeomorphic to
real projective space $\R\P^3$, hence whenever we see a manifold whose
fundamental group is $\Z/2\Z$, we know that it is $\R\P^3$.
The simplest \tat\ twist that produces $\R\P^3$ is the bifoil twist
$\Bi = \T{E_{2,2},1}$.

The book of Orlik about Seifert manifolds (\cite{Orlik1972}) gives a description
of all possible fundamental groups of spherical manifolds.
Powers of $\Bi$ also produce interesting examples:
The fundamental group of the open book of $\Bi^2$ is the binary dihedral group
$\DD$ of order 8, also called the dicyclic group $\mathrm{Dic}_2$ or
the quaternion group $\mathrm{Q}$.
From $\Bi^3$ we obtain the binary octahedral group, a group of order 48.

\subsubsection{Powers and branched covers}
Constructing the open book $M_{\phi^k}$ of the power $\phi^k$ of a
diffeomorphism $\phi$ corresponds to taking a cyclic branched cover of the manifold
$M_\phi$.
In particular, if $\phi$ is the monodromy of a knot, $M_{\phi^k}$
is a $k$-fold branched cover of $S^3$ branched over the knot.

For example, the twist with walk length $2$ around the elementary \tat\ graph
$E_{30,29}$ corresponds to the $(2,15)$-torus knot monodromy.
Its square produces a manifold whose fundamental group is $\Z/15\Z$, and which
therefore is a lens space $L$.
The double cover branched over the $(4,15)$-torus knot, call it $L'$, is itself
a double branched cover of $L$ because of the symmetry of the $(4,15)$-torus knot.
GAP can calculate its homology, which is still $\Z/15\Z$, but is not able to
calculate its fundamental group.
In 1983, Hodgson and Rubinstein have shown (\cite{HodgsonRubinstein1983}) that a
lens space occurs as the double branched cover of a unique knot in $S^3$, which
is in fact a 2-bridge knot, therefore $L'$ cannot be a lens space.
(Only the $(2,n)$-torus knots are also 2-bridge knots.
The bridge number of a $(p,q)$-torus knot is $\min(p,q)$; see
\cite{Schultens2007}.)
GAP can also calculate the homologies of the double branched covers over
the torus knots $(8,15)$ and $(16,15)$, which remain $\Z/15\Z$.

In the bifoil twist example above, we have seen ``quaternion space'' as a double,
and ``binary octahedral space'' as a triple branched cover over real projective
space.

There is an obvious symmetry for any diffeomorphism $\phi$:
\[
	M_{\phi^{-k}} = -M_{\phi^{\,k}},
\]
where $-M$ denotes $M$ with reversed orientation.
Any calculation with a negative walk length $-l$ will therefore result in a
fundamental group which is isomorphic to the one coming from the positive
walk length $l$.

\subsubsection{Homology spheres}
%http://faculty.sites.uci.edu/rstern/files/2011/03/27_Diff_geo_Inv_Hom_3_Spheres.pdf
%http://mathoverflow.net/questions/4798/classification-of-homology-3-spheres
Using the GAP command \code{AbelianInvariants}, one finds the abelianization of the
fundamental group, that is, the first homology group $H_1(M_\phi)$.
If the fundamental group is perfect, meaning that $H_1(M_\phi)$ is trivial,
$M_\phi$ has the same homology as the 3-sphere and is called a homology sphere.
See the book of Saveliev (\cite{Saveliev2002}) and the introduction of
\cite{FintushelStern1991} for overviews.

We know that all $M_\phi$ are Seifert fibred.
Examples of Seifert fibred homology spheres are the so-called \emph{Brieskorn spheres} $\Sigma(p,q,r)$,
the links of the singularities $x^p+y^q+z^r = 0$, where $p$, $q$, and $r$ are
pairwise coprime positive integers.

The only perfect fundamental group of a 3-manifold that is finite is the binary icosahedral
group, the fundamental group of the Poincaré sphere.
Therefore, due to the elliptization
conjecture, the Poincaré sphere is the only homology sphere with finite fundamental group.
It arises as the 5-fold cyclic branched covering of the trefoil, that is, as
the open book of $\Tr^5$, the fifth power of the trefoil twist $\Tr = \T{E_{3,3},1}$.
The Poincaré sphere is also a Brieskorn sphere, namely $\Sigma(2,3,5)$.
More generally, the $k$-fold cyclic branched covering of the $(p,q)$-torus knot
is the Brieskorn sphere $\Sigma(p,q,k)$.
Indeed we find another homology sphere, $\Sigma(2,3,7)$ by looking at the open
book that belongs to $\Tr^7$, or $\Sigma(2,3,11)$ from
$\Tr^{11}$.

\subsubsection{Dehn surgery}
A \tat\ twist has a power which is isotopic to a composition of
Dehn twists around its boundary.
Two open books obtained from diffeomorphisms which differ by
boundary twists are related by some Dehn surgery along the boundary link.
In the special case where a \tat\ twist $\T{}$ is a torus knot monodromy, hence $M_{\T{}} \cong S^3$,
and $k$ is its order, $M_{\T{}^{1+nk}}$ is the result of $1/n$-surgery
on the knot.
In the general case, the framing would be determined by the algebraic intersection
number with the surface.

Take the example of the left-handed trefoil, which is the binding in the open
book for $\Tr^{-1}$.
The open book of $\Tr^5 = \Tr^6 \Tr^{-1}$ is the result of $+1$-surgery on the left-handed
trefoil, which is known to be the Poincaré sphere -- another way to see why
the calculation for $\Tr^5$ mentioned above is correct.

Like in this example, the homology never changes under this kind of surgery; this
can also be seen without referring to surgery from the fact that boundary twists
act trivially on the homology of the surface.
Therefore, and since the homologies of $M$ and $-M$ are the same, all the
$H_1(M_{\T{}^{r+nk}})$
as well as all the
$H_1(M_{\T{}^{-r+nk}})$
are isomorphic, where $\T{}$ is a twist of order $k$ and $n$ is an arbitrary integer.

When $\T{}$ is a \tat\ twist on a surface $\Sigma$ we get in particular that
\[
H_1(M_{\T{}^{nk}}) = H_1(M_{\id_\Sigma}) = H_1(\Sigma),
\]
which is free abelian.
In these cases, unless $\Sigma$ is a disk, the fundamental group is infinite.

\subsubsection{Selection of examples}
The following table contains a rather arbitrary selection of calculations.
The columns are the \tat\ graph with its genus and number of boundary components,
the walk length, which power of the twist with
minimal walk length is considered, the fundamental group of the open book,
its homology, and some remarks.

The symbol $\DD$, as before, denotes the binary dihedral group with 8 elements, or the quaternion group;
$\TT$ the binary tetrahedral group, which is of order 24 and isomorphic to $\mathrm{SL}(2,3)$;
$\OO$ the binary octahedral group, which is of order 48;
$\II$ the binary icosahedral group, which is of order 120 and isomorphic to $\mathrm{SL}(2,5)$.
$\infty$ denotes any infinite group, and $\Z_n$ is short for $\Z/n\Z$.
$1$, as well as $0$, denote the trivial group.
Where the fundamental group is not indicated, GAP was not able to tell
whether it is finite or not, most likely because it is not.

Graphs \textit{(i)}, \textit{(ii)}, and \textit{(iii)} are the ones from Figure \vref{fig:moretat}.

\newcommand \ZZ[2] {{\Z_{#1}}^{\mkern-1.5mu#2}} % For powers of finite cyclic groups
\vspace{2ex}
{\fontfamily{pplx}\selectfont     % disable old style numbers here
{\renewcommand{\arraystretch}{1.1} % make row spacing a bit bigger
\begin{longtable}{lrrlll}
\toprule
\multicolumn{3}{c}{twist} & \multicolumn{2}{c}{open book}\\
\cmidrule(r){1-3} \cmidrule(r){4-5}
graph       & $l$ &power& $\pi_1$   & $H_1$         & remark      \\
\ $(g,b)$   &     &       &         &               &             \\
\midrule
\endhead
$E_{2,2}$   & 1   & 1   & $\Z_2$    & $\Z_2$      & $\R\P^3$    \\
\ $(1,1)$   & 2   & 2   & $\DD$     & $\ZZ{2}{2}$ &             \\
            & 3   & 3   & $\OO$     & $\Z_2$      &             \\
            & 4   & 4   & $\infty$  & $\Z^2$      &             \\
            & 5   & 5   &           & $\Z_2$      &             \\
$E_{3,3}$   & 1   & 1   & $1$       & $0$         & $S^3$, trefoil \\
\ $(1,1)$   & 2   & 2   & $\Z_3$    & $\Z_3$      & lens space $L(3,1)$ \\
            & 3   & 3   & $\DD$     & $\ZZ{2}{2}$ &             \\
            & 4   & 4   & $\TT$     & $\Z_3$      &             \\
            & 5   & 5   & $\II$     & $0$         & Poincaré sphere \\
            & 6   & 6   & $\infty$  & $\Z^2$      &             \\
$E_{5,3}$   & 2   & 1   & $\Z_5$    & $\Z_5$      & lens space  \\
\ $(2,1)$   & 4   & 2   &           & $\Z_5$      &             \\
            & 6   & 3   &           & $\Z_5$      &             \\
            & 8   & 4   &           & $\Z_5$      &             \\
            & 10  & 5   & $\infty$  & $\Z^4$      &             \\
$E_{6,3}$   & 2   & 1   & $\Z_3$    & $\Z_3$      & lens space  \\
\ $(2,1)$   & 4   & 2   &           & $\ZZ{3}{2}$ &             \\
            & 6   & 3   & $\infty$  & $\Z^{2}$    &             \\
            & 8   & 4   &           & $\ZZ{3}{2}$ &             \\
            & 10  & 5   &           & $\Z_3$      &             \\
            & 12  & 6   & $\infty$  & $\Z^4$      &             \\
$E_{7,3}$   & 2   & 1   & $\Z/7\Z$  & $\Z_7$      & lens space  \\
\ $(3,1)$   & 4   & 2   &           & $\Z_7$      & also for powers 3,4,5,6 \\
            & 14  & 7   & $\infty$  & $\Z^6$      &             \\
$E_{7,5}$   &     &     &           &             & like for $E_{7,3}$ \\
\ $(3,1)$   &     &     &           &             &             \\
$E_{12,7}$  & 2   & 1   & 1         & 0           & $S^3$, $(3,4)$-torus knot \\
\ $(3,1)$   & 4   & 2   & $\DD$     & $\Z_3$      &             \\
            & 6   & 3   &           & $\ZZ{4}{2}$ &             \\
            & 8   & 4   &           & $\ZZ{3}{3}$ &             \\
            & 10  & 5   &           & 0           & homology sphere \\
            & 12  & 6   & $\infty$  & $\Z^2+\ZZ{2}{2}$ &       \\
            & 14  & 7   &			& 0           & homology sphere \\
            & 16  & 8   &           & $\ZZ{3}{3}$ &             \\
            & 18  & 9   &           & $\ZZ{4}{2}$ &             \\
            & 20  & 10  &           & $\Z_3$      &             \\
            & 22  & 11  &           & 0           & homology sphere \\
            & 24  & 12  & $\infty$  & $\Z^6$      &              \\
$E_{12,11}$ & 2   & 1   & $\Z_2$    & $\Z_2$      & $\R\P^3$     \\
\ $(3,1)$   & 4   & 2   & $\mathrm{D}^*_6$ & $\ZZ{2}{2}$ & order 24 \\
$E_{12,5}$  & 2   & 1   & $\Z_2$    & $\Z_2$      & $\R\P^3$     \\
\ $(4,1)$   & 4   & 2   &           & $\ZZ{2}{2}+\Z_3$ &         \\
$E_{12,3}$  & 2   & 1   & $\Z_6$    & $\Z_6$      &              \\
\ $(5,1)$   &     &     &           &             &              \\
$E_{60,19}$   & 2   & 1   & $\Z_2$    & $\Z_2$      & $\R\P^3$     \\
\ $(24,1)$    & 4   & 2   &           & $\ZZ{2}{2}+\Z_3$ &         \\
$E_{52,13}$   & 2   & 1   & $\Z_{26}$ & $\Z_{26}$   & lens space   \\
\ $(25,1)$    & 4   & 2   &           & $\ZZ{2}{2}+\ZZ{13}{2}$ &   \\
$E_{120,119}$ & 2   & 1   & $\Z_2$    & $\Z_2$      & $\R\P^3$     \\
\ $(30,1)$    & 4   & 2   & $\mathrm{D}^*_{60}$ & $\ZZ{2}{2}$ & order 240 \\
              & 6   & 3   &           & $\ZZ{2}{3}$ &              \\
$E_{120,59}$  & 2   & 1   & $\Z_4$    & $\Z_4$      & lens space   \\
\ $(45,1)$    &     &     &           &             &              \\
$E_{120,5}$   & 2   & 1   & $\Z_{20}$ & $\Z_{20}$   & lens space   \\
\ $(58,1)$    & 4   & 2   &           & $\Z_3\!+\!\ZZ{4}{2}\!+\!\ZZ{5}{2}$ & \\
\textit{(i)}  & 2   & 1   & $\Z_2$    & $\Z_2$      & $\R\P^3$     \\
\ $(1,4)$     &     &     &           &             &              \\
\textit{(ii)} & 1   & 1   & $\Z_4$    & $\Z_4$      & lens space   \\
\ $(1,2)$     &     &     &           &             &              \\
\textit{(iii)}& 2   & 1   & 1         & 0           & $S^3$        \\
\ $(0,2)$     & 4   & 2   & $\Z_2$    & $\Z_2$      & $\R\P^3$     \\
              & $2n$& $n$ & $\Z_{2n}$ & $\Z_{2n}$   &              \\

\end{longtable}
}
}

All the elementary \tat\ graphs of the form $E_{2n,2n}$ with walk length $1$,
or equivalently $E_{4n,4n-1}$ with walk length $2$ can easily be shown to
produce $\R\P^3$ as its open book, as suggested by the list.
From their powers we therefore get branched covers of $\R\P^3$.

%==================================================
\chapter{Fibred knots in $\R^3$}\label{chap:elastic}
%==================================================
\section{Introduction}
\checkhere
This chapter digresses from the study of \tat\ twists and treats the subject
of fibred links in the 3-sphere, which already appeared in the previous one.
Its main result is a criterion for fibredness.
Whereas fibredness is a property of links and surfaces in $S^3$, it is
very natural to use the criterion in $\R^3$.

Recall that a knot or link $K$ in $S^3$ is called \emph{fibred} if the link complement $S^3
\smallsetminus K$ admits the structure of a fibration over $S^1$, and moreover, the closures
of the fibres are compact surfaces that intersect exactly in $K$.
The closures of fibres, called \emph{fibre surfaces}, are Seifert surfaces for the link.
When $K$ is oriented, the fibers are required to induce the correct orientation on it.

Many examples come from plane curve singularities: If a polynomial function
$f\colon \C^2 \to \C$ has an isolated singularity at $0$, we can define $K$ to
be the intersection of $f$'s zero set with a sufficiently small transverse
sphere, \[K=\left\{p \in \C^2 \mid f(p)=0 \textrm{ and } \abs{x(p)}^2 +
\abs{y(p)}^2 = \varepsilon \right\},\] and the projection $F\colon S^3\smallsetminus K
\to S^1$ to be the quotient $f/\abs{f}$; see Milnor's book \cite{Milnor1968}. The simplest examples of such
\emph{algebraic links} are the trivial knot, described by the function
$f(p)=x(p)$; the Hopf link, described by $f(p)=x(p)^2+y(p)^2$; and the trefoil,
described by $f(p)=x(p)^2+y(p)^3$. However, there are many knots which are
fibred but not algebraic, for instance the figure-eight knot. And not every knot
is fibred: A fibred knot has a monic Alexander polynomial,
therefore knots like $5_2$ or $6_1$ (in Rolfsen's notation) are not fibred.

The converse of this criterion is false, but others exist. In 1962, John
Stallings showed that a link is fibred if and only if the commutator subgroup of
its fundamental group is finitely generated (\cite{Stallings1962}). In 1986,
David Gabai presented his theory of sutured manifolds (\cite{Gabai1986}), which
in many cases allows to decide whether or not a link is fibred. More recently,
Yi Ni has shown that Knot Floer homology detects whether a knot is fibred
(\cite{Ni2007,Ni2009}).

If we look at the fibre $\Sigma$ over 1, that is
$\Sigma = \overline{F^{-1}(1)} = F^{-1}(1) \cup K$,
and then follow the points of $\Sigma$
as it is moved through the fibration, we get a diffeomorphism of $\Sigma$ called
the \emph{monodromy}. This can be done by choosing a vector field on $S^3$
that is zero on $K$ and otherwise projects
to a field of unit tangent vectors on $S^1$. The monodromy is unique up to isotopy. In particular,
we can look at the image of a properly embedded arc under the monodromy.
The interior of such an arc is moved by the flow of the vector field
through the complement of $\Sigma$ and, in general, ends up in a different
position.

\begin{figure}
\centering
\def\svgwidth{0.5\textwidth}
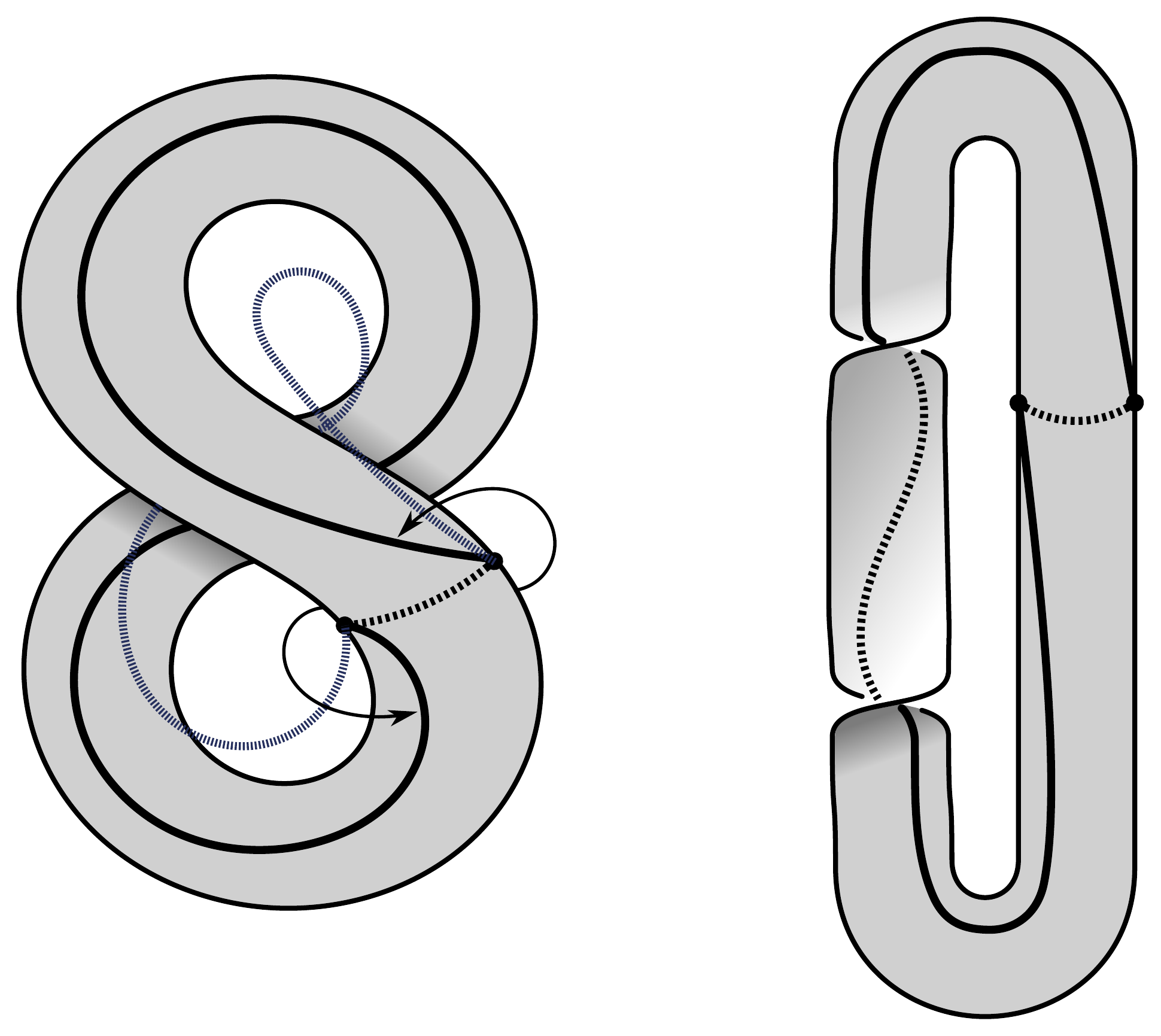
  \caption{Two views of the same Hopf band with the action of the monodromy on an elastic cord}
  \label{fig:hopfband}
\end{figure}

\checkhere
In this chapter we will see how and why studying such arcs is sufficient to determine
whether a knot is fibred as well as to describe the monodromy. It turns out that
all one needs to study fibre surfaces are a bunch of those elastic luggage
cords with hooks. The main statement is: \emph{If every elastic cord, attached to
the boundary of the surface, can be dragged to the other side of the surface,
the knot is fibred; and the monodromy is determined by where the elastic cords
end up.}

The results cited \checkhere in this chapter are formulated in the PL category,
hence all statements about surfaces and elastic cords are to be understood in the piecewise
linear sense as well, even if not explicitly stated, or drawn.
Of course, every continuous movement of an elastic cord can be piecewise linearly
approximated.

%==================================================
\section{Elastic cords}
From now on, let $\Sigma \subset S^3$ be an embedded connected compact
oriented surface with boundary. We will often have to thicken $\Sigma$ in a
specific way, illustrated by Figure \ref{fig:lens_thickening}, which we call
a ``lens thickening'' and is natural in the context of fibred links. This
thickening can be imagined to be very thin and is mainly used to distinguish
the two sides of the surface $\Sigma$.

\begin{definition}
Let $\NN(\Sigma)$ be a closed tubular neighbourhood ``with boundary'' of $\Sigma$,
parameterized by $\tau \colon \Sigma \times [-1,1] \to \NN(\Sigma)$. Let $h
\colon\Sigma \to [0,1]$ be a smooth function which is zero on the boundary of
$\Sigma$ and positive on its interior. The image $\mathcal{L}(\Sigma)$ of the map $(p,t)
\mapsto \tau(p,h(p)\cdot t)$, together with its structure as a fibration 
$\LL(\Sigma)\smallsetminus\partial\Sigma \to [-1,1]$ given by the parameter $t$, is
called a \emph{lens thickening} of $\Sigma$.
\end{definition}

Some further terminology: The image of $\Sigma \times ]0,1]$ lies \emph{above},
the image of $\Sigma \times [-1,0[$ \emph{below $\Sigma$}. The part of the
boundary of $\LL(\Sigma)$ which lies above $\Sigma$ will be denoted by
$\partial^+\LL(\Sigma)$, the part below by $\partial^-\LL(\Sigma)$. Finally, let
$\EE(\Sigma) = \overline{S^3 \smallsetminus \LL(\Sigma)}$, the exterior of $\Sigma$. We
will also tacitly remember the projection to $\Sigma$ induced by the tubular
neighbourhood structure, but this projection could be reconstructed up to
isotopy from the fibration structure.

\begin{figure}
\centering
\def\svgwidth{0.5\textwidth}
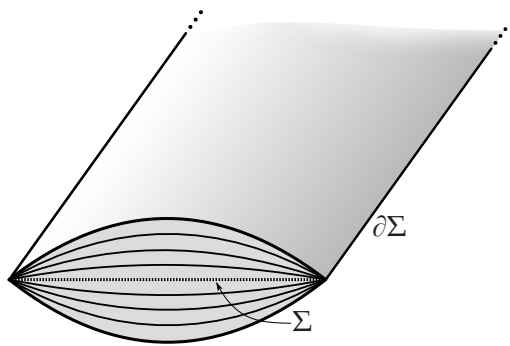
  \caption{Slice through a lens thickening of a band}
  \label{fig:lens_thickening}
\end{figure}

\begin{definition}
Choose a fixed lens thickening $\LL(\Sigma)$. An \emph{elastic cord} attached to
$\Sigma$ is an embedded arc in $\EE(\Sigma)$ whose
endpoints lie on $\partial\Sigma$.

We say that an elastic cord is spanned above $\Sigma$ when its interior is
contained in $\partial^+\LL(\Sigma)$, spanned below $\Sigma$ when its interior is
contained in $\partial^-\LL(\Sigma)$. A cord spanned below $\Sigma$ can be
\emph{dragged to the other side} of $\Sigma$ if there is an isotopy of elastic cords moving
it to a cord above, while keeping its endpoints fixed.
\end{definition}

\begin{theorem}[existence of a fibration]
\label{thm:existence}
If every elastic cord on $\Sigma$ can be dragged to the other side, $\Sigma$ is a
fibre surface.
\end{theorem}

In fact, it suffices to study a collection of disjoint cords whose projections generate
$H_1(\Sigma,\partial\Sigma)$, or equivalently, cut the surface into one disk.
Moreover, this existence statement can even be slightly strengthened to allow
for the dragged cords to cross over themselves, as in Theorem \ref{thm:stronger}
below. We also have:

\begin{theorem}[uniqueness of the monodromy]
\label{thm:uniqueness}
Monodromies are unique up to isotopy.
More precisely:
If $\Sigma$ is a fibre surface, there is only one way to drag a cord spanned below it to a cord
spanned above it, up to isotopy.
The position of the dragged cords determines the monodromy.
\end{theorem}

These two theorems have been obtained in collaboration with
Sebastian Baader.
\checkhere
In their proofs, we will mainly be concerned with resolving
singularities of images of disks, a standard problem in 3-manifold topology. We
will use, to some extent, four important classical theorems, namely Dehn's
lemma, the sphere theorem, Alexander's theorem, and later the loop theorem.
The first three are stated right below.
References for those statements are the book of Bing (\cite{Bing1983}) and the book
of Hempel (\cite{Hempel1976}), both with proofs. PL manifolds and maps are used,
which will be implicit in the statements below. The unpublished book fragment of Hatcher
(\cite{Hatcher2007}) states the theorems for continuous maps.

\begin{DEHNSLEMMA}[Papakyriakopoulos, 1957, \cite{Papakyriakopoulos1957}; see {\cite[p.~39]{Hempel1976}}]
Let $M$ be a 3-manifold and $f\colon D^2 \to M$ a map
such that for some open neighbourhood $A$ of $\partial D^2$ the restriction
$f|_A$ is an embedding and $f^{-1}(f(A)) = A$. Then there is an embedding $g\colon D^2 \to M$ such that $\partial(f(D^2)) = \partial(g(D^2))$.\qed
\end{DEHNSLEMMA}
\begin{SPHERE}[Papakyriakopoulos, 1957, \cite{Papakyriakopoulos1957}; see {\cite[p.~40]{Hempel1976}}]
Let $M$ be an orientable 3-manifold with nontrivial $\pi_2(M)$.
Then there exists an embedding of the 2-sphere which is nontrivial in $\pi_2(M)$.\qed
\end{SPHERE}
\begin{ALEXANDER}[Alexander, 1924, \cite{Alexander1924}]
A 2-sphere that is embedded in $S^3$ bounds a 3-ball on both sides.\qed
\end{ALEXANDER}

\begin{proof}[Proof of existence (Theorem \ref{thm:existence})]
We prove the statement in three steps:

\begin{enumerate-packed}
  \item Each cord can be moved along an embedded disk;
  \item for cords which do not intersect, those disks can be chosen to be
  disjoint;
  \item the fibration structure on the union of $\Sigma$ and the right amount of
  such disks can be extended (uniquely) to the complement, which is a ball.
\end{enumerate-packed}

The first claim follows from Dehn's lemma. Choose one cord $\alpha$ and drag it
to the other side. Since the movement of its interior happens away from the
surface, we can choose a small neighbourhood $\NN_\varepsilon$ of $\LL(\Sigma)$ and
modify the isotopy $H\!\colon [0,1]\times\alpha \to \EE(\Sigma)$ to
make it injective on $H^{-1}(\NN_\varepsilon)$. Dehn's lemma now says that there
is an embedded disk whose boundary is the original one, namely the union of the
two cords $H(\{0\}\times\alpha) \cup H(\{1\}\times\alpha)$.

Now take two disjoint cords $\alpha$ and $\beta$ and find two embedded disks
$D_\alpha$ and $D_\beta$ along which they can be dragged to cords $\alpha'$ and
$\beta'$. The disks, as well as $\alpha'$ and $\beta'$ themselves, might
intersect each other. To make them disjoint, start by perturbing one of the
disks slightly to make it transverse to the other one. Now, they intersect in
some disjoint embedded circles and arcs. Since the disks do not intersect the
surface and the four cords lie on a lens thickening of it, we can ask
furthermore that the boundary of each disk do not intersect the interior of the
other disk. The arc components of the intersection now have their endpoints on
the boundary, more precisely on $\alpha' \cap \beta'$.

The goal is to successively remove \emph{innermost} circles and arcs. A circle
is called innermost for a disk if it contains no other circles or arcs.
An arc divides the disk into two parts, only one of which touches $\alpha$ or
$\beta$. We will call the other part its ``inner disk'' and say that the arc is
innermost if its inner disk contains no other arcs or circles.

Here is what we do with circles: Choose one which is innermost for $D_\beta$. In
$D_\beta$, this circle bounds a disk $D_\beta'$. The part of $D_\alpha$ which
lies inside the circle is also a disk, call it $D_\alpha'$, but one that
possibly intersects $D_\beta$ many times. We can now do surgery to modify
$D_\alpha$ and reduce the number of intersections by at least one: Remove
$D_\alpha'$, together with a small annulus around it, from $D_\alpha$ and
replace it by a disk parallel to $D_\beta'$. Because the circle was innermost
for $D_\beta$, we do not introduce any self-intersections.

Arcs are treated in the same way. When an arc is innermost for $D_\beta$, the
inner disk in $D_\alpha$ is removed, together with a small band around it, and
replaced with a disk parallel to the inner disk bounded in $D_\beta$.

Repeating these steps produces two disjoint disks in the end. Because we
have only ever modified $D_\alpha$, we can continue this process to make $D_\alpha$
disjoint from as many other disks as we like, and by induction we can choose
all disks to be disjoint.

As a remark:
$\pi_2(S^3\smallsetminus\Sigma)$ is trivial, for if it were not, there would exist an
essential sphere (by the sphere theorem) which would bound two balls in $S^3$
(by Alexander's theorem). Since $\Sigma$ is connected, only one of the balls can
contain $\Sigma$, so the sphere was not essential after all. This means that our
disk $D_\alpha$ could indeed have been moved to the surgered disk by a homotopy.

To build the fibration, we choose a maximal nonseparating collection of
disjoint properly embedded arcs $\gamma_i$ in $\Sigma$, push them to
$\partial^-\LL(\Sigma)$, and construct a disk $D_i$ for each of them as before,
which we thicken slightly to a two-handle $\tilde{D_i}$.
There will be $b_1=2g+r-1$ of them, where $g$ is the genus and $r$ the number of
boundary components of $\Sigma$, and $b_1$ is the rank of
$H_1(\Sigma,\partial\Sigma)$ and $H_1(\Sigma)$. The fibration is already defined
on $\LL(\Sigma)\smallsetminus\partial\Sigma$; it can be extended to the (thickened)
disks $\tilde{D_i}$ where it reflects the movement of the cords through them,
and we would like to extend it to the rest of $S^3\smallsetminus\partial\Sigma$.

\begin{figure}
\centering
\def\svgwidth{0.4\textwidth}
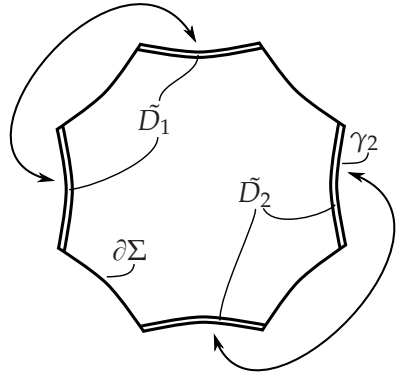
  \caption{An octagon obtained by cutting a three-holed sphere, for which $b_1=2$}
  \label{fig:polygon}
\end{figure}

The boundary of $\LL(\Sigma) \cup \bigcup_{i=1}^{b_1}\tilde{D_i}$ is a sphere. We
can prove this by simply calculating its Euler characteristic: Cutting $\Sigma$
along the $b_1$ arcs produces a $4b_1$-gon whose edges alternatingly belong to
the boundary of $\Sigma$ and to the cutting arcs. By pushing the arcs down, we
likewise cut $\partial^-(\Sigma)$ to a polygon like in Figure \ref{fig:polygon}.
$\partial^+(\Sigma)$, cut along the dragged arcs, will also look like Figure
\ref{fig:polygon} since on $\Sigma$ any two choices of a nonseparating
collection of disjoint embedded arcs with the same endpoints as the $\gamma_i$
are related by a diffeomorphism of $\Sigma$.
The sphere will consist of the cut $\partial^-(\Sigma)$ and
$\partial^+(\Sigma)$, glued together along the boundary parts, and of $2b_1$
disks attached at the cuts; we end up with $2 + 2b_1$ disks, $2b_1 + 4b_1$
edges, and $4b_1$ vertices.

By Alexander's theorem, the sphere bounds a ball on
both sides. The ball on the outside looks like in Figure
\ref{fig:flying_saucer}, and it is easy to extend the fibration to it. Another
way to look at it is to reglue the parts of its boundary which border the
thickened disks (the windows of the flying saucer in the picture), respecting their
fibration induced by the disks, and get a handlebody to which we can give the
structure of a lens thickening.

This completes the fibration of the link complement and proves the theorem.
\end{proof}

\begin{figure}
\centering
\def\svgwidth{0.51\textwidth}
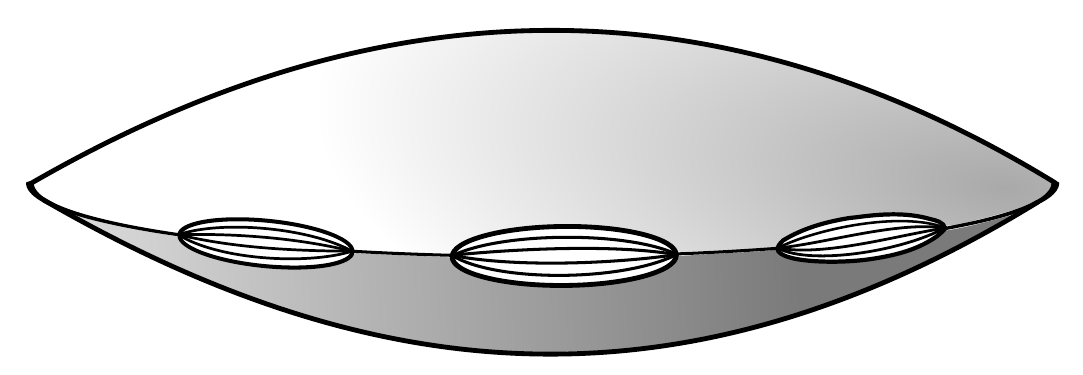
  \caption[The remaining ball with a prescribed fibration on its boundary]%
  {The remaining ball with a prescribed fibration on its boundary. Top and bottom are one fibre each; the rim, not including the windows, is $\partial\Sigma$.}
  \label{fig:flying_saucer}
\end{figure}

To prove the statement of Theorem \ref{thm:uniqueness}, uniqueness of the monodromy, we need the
following well-known proposition:
\begin{proposition}
A fibre surface is incompressible in the link complement.
\end{proposition}
Also some converse of this statement would be true: A Seifert surface of a fibred
link is a fibre surface if and only if it is incompressible in the knot
complement. We prove only the ``only-if'' part:

\begin{proof}
Assume the surface, which we call again $\Sigma$, is compressible. Let $D$
be a compressing disk whose boundary we assume to lie in the interior of the
surface.

There may or may not be boundary components of $\Sigma$ on either side of
$\partial D$ (possibly the same one on both sides if $\partial D$ is
nonseparating). If there are, choose a cord that lies on the same side of $\Sigma$
as $D$ and intersects $D$ in exactly one point. An isotopy that fixes the chord's endpoints
can only change the number of intersection points by an even number, so this
cord cannot be brought to the other side.

We are left with the case that one component of $\Sigma\smallsetminus\partial D$ is
capped off by $D$ to a closed two-sided surface whose genus is at least one.
Hence its inside, which we call $I$, contains an essential loop with basepoint
on $D$. Prolong this on the other side of $D$ to an essential loop $\gamma$ with
basepoint on $\partial\Sigma$. Choose now an elastic cord $\alpha$ attached to
$\partial\Sigma$ near $\gamma$'s basepoint and isotopic to $\gamma$ in
$S^3\smallsetminus\Sigma$. $\alpha$ can be laid down on one side of $\Sigma$,
but since it cannot leave $I$ it cannot be brought to the other side.
\end{proof}

\begin{proof}[Proof of uniqueness (Theorem \ref{thm:uniqueness})]
Let us assume that there exists a cord
$\alpha$ below which can be dragged to two nonisotopic cords $\alpha_1'$ and
$\alpha_2'$ above. As in the proof of Theorem \ref{thm:existence}, this movement can be
thought to happen along two embedded disks. Then, we can use the same surgery
techniques to find two embedded disks, one between $\alpha$ and $\alpha_1''$,
and one between $\alpha$ and $\alpha_2''$, that only intersect along $\alpha$.
For example, one could move $\alpha$ slightly away from itself, find two
disjoint disks and then undo the movement.

But then, provided $\alpha_1''$ and $\alpha_2''$ are still not isotopic, we can
combine the two disks to find a compressing disk whose boundary is
$\alpha_1''\cup\alpha_2''$, in contradiction to the assumption that $\Sigma$ was
a fibre surface. If it should happen (and it can) that $\alpha_1''$ and
$\alpha_2''$ are isotopic, then necessarily one of the disks we used for surgery
had a boundary which was an essential loop in the surface and can be used as a
compressing disk.

The second statement of the theorem is clear, since the disks are mapped to disks again, and
there is only one (orientation-preserving) way to do this, up to isotopy.
Therefore the monodromy is completely determined by the images of the arcs.
\end{proof}

With a little bit more work, we can allow elastic cords to be immersed instead of
embedded, or even to be just arbitrary continuous images of an interval, as well
when we put them down on the surface as during the movement. All that is needed
is a homotopy keeping the endpoints fixed and moving the interior of the cord
from the negative to the positive boundary part of a lens neighbourhood.

\begin{theorem}
\label{thm:stronger}
If every (embedded) elastic cord on $\Sigma$ can be moved to the
other side of the surface, not necessarily remaining embedded,
then $\Sigma$ is a fibre surface.
\end{theorem}

\begin{proof}
We repeat the first step of the proof of Theorem \ref{thm:existence} under
this weaker assumption.
The rest of the proof remains the same.
Since we never needed to consider elastic cords whose projection separates $\Sigma$,
we can assume that the cord which is to be moved is nonseparating, without loss of generality. 

First of all, the topological disk swept out by the elastic cord can be
approximated by a piecewise-linear disk $D$ that only has singularities of a certain
kind, namely double lines, triple points and branch points, see for example the
book of Bing (\cite[Chapter XVII.1.\ and p.~205]{Bing1983}). Bing calls this a
``normal singular disk''.

Now, we should find an embedded disk whose boundary is still contained in
$\partial \LL(\Sigma)$ and whose intersection with $\partial^-\LL(\Sigma)$ is the
original embedded elastic cord. We use the loop theorem, first proved by
Papakyriakopoulos (also in \cite{Papakyriakopoulos1957}), in a version which corresponds to Theorem XVII.1.E in Bing's
book. It says that, using local modifications of the singular disk near the
singularity set called ``cut, paste and discard'', there exists an embedded disk
whose boundary is a part of the original boundary with smoothed crossings. For
this new boundary, one can furthermore choose a forbidden normal subgroup $N$ of
the fundamental group of the surface, of course provided that $N$ does not
contain the original boundary. We use $N$ to ensure that the original elastic
cord is not discarded. The manifold $M$ in the theorem will be $S^3\smallsetminus
\LL(\Sigma)$.

\begin{LOOP}
Suppose $D$ is a normal singular disk in a PL 3-manifold-with-boundary $M$ and
$B$ is a boundary component of $\partial M$. Let $N$ be a normal subgroup of
$\pi_1(B)$ that does not contain the representatives of the conjugacy class of
$\partial D$. Then $D$ can be changed by cut, paste and discard to a nonsingular
disk $E$ such that $\partial E \subset B$ and the representatives of the conjugacy
class of $\partial E$ do not belong to $N$.\qed
\end{LOOP}

For $N \lhd \pi_1(\partial \LL(\Sigma))$, choose the normal subgroup generated by
the subgroup $\pi_1(\partial^+\LL(\Sigma))$. Since the elastic cord spanned below
$\Sigma$ does not separate $\partial^-\LL(\Sigma)$ by assumption, we can choose an
oriented simple closed curve $c$ in $\partial^-\LL(\Sigma)$ whose intersection
number with $\partial D$ is $1$. But the intersection number of any element of
$N$ with $c$ is $0$, so the representatives of $\partial D$ are not contained
in $N$.

The loop theorem provides us with an embedded disk $E$ such that $\partial E$ is
contained in $\partial D$ away from the intersection points of $\partial D$. And
$\partial E \cap \partial^-\LL(\Sigma)$ is the original elastic cord, because if
$\partial E$ did not pass at all through $\partial^-\LL(\Sigma)$, it would be
contained in $N$.

This disk allows to move the elastic cord to the other side of $\Sigma$ through
embedded elastic cords and thus the theorem is proved.
\end{proof}

%==================================================
\subsection{Decision problems}
Visualizing the moving elastic cords can be difficult. On a bad day, the
following modification can be easier to handle:
\begin{corollary}
A cord $\rho$ in upper position can be dragged to a cord $\rho'$ in lower
position (with the same endpoints) if and only if $\rho \cup \rho'$ is unknotted
and unlinked with $\Sigma$.
\end{corollary}
\begin{proof}
As we have seen in the proof of Theorem \ref{thm:existence}, the movement
from $\rho$ to $\rho'$ can be performed along a disk whose interior is embedded
and disjoint from $\Sigma$. $\rho \cup \rho'$ bounds this disk, so it is
unknotted and not linked with $\Sigma$.

On the other hand, a trivial knot that is not linked to $\Sigma$ always bounds
such a disk, so $\rho$ can be pushed along it.
\end{proof}
Since there exist algorithms to decide whether a surface is a fibre surface,
one may ask:
\begin{question}
How can this criterion be made into an algorithm?
\end{question}
Here, we should clarify the relation between fibredness of a link and fibredness
of its Seifert surfaces: When we are given only the link, we can use Haken's normal
surface theory to find a Seifert surface of minimal genus for it (and giving it its
proper orientation). If the link is fibred, it is well known that this surface must
be the fibre surface. Therefore Haken's algorithm together with a criterion to decide
whether a surface is fibred are sufficient to decide whether a link is fibred.

%==================================================
\section{Examples and applications}
A standard example in the theory of fibre surfaces is the twisted unknotted
annulus.
\begin{example}
An unknotted annulus in $S^3$ with $n$ full twists, $n \neq 0$, is (when we disregard orientation) a Seifert
surface for the $(2,2n)$-torus link, which is fibred. Trivially, an annulus is a
Seifert surface of minimal genus. For fibred \emph{knots} this would imply that
it would also be a fibre surface, but unless $\abs{n} = 1$, this is not the case
here. If it were, it would be possible to complete the cord $\rho$ in Figure
\ref{fig:torusknotband}, which lies on the back of the surface, with another
cord $\rho'$ on top of the surface, to a trivial knot that is unlinked with the
annulus. But if we choose $\rho'$ to have the same projection as $\rho$, the
linking number of $\rho \cup \rho'$ with the annulus is $1$, and if $\rho'$ goes
$k$ times around the annulus, the linking number is changed by $k \cdot n$.
\end{example}

\begin{figure}
\centering
\def\svgwidth{0.5\textwidth}
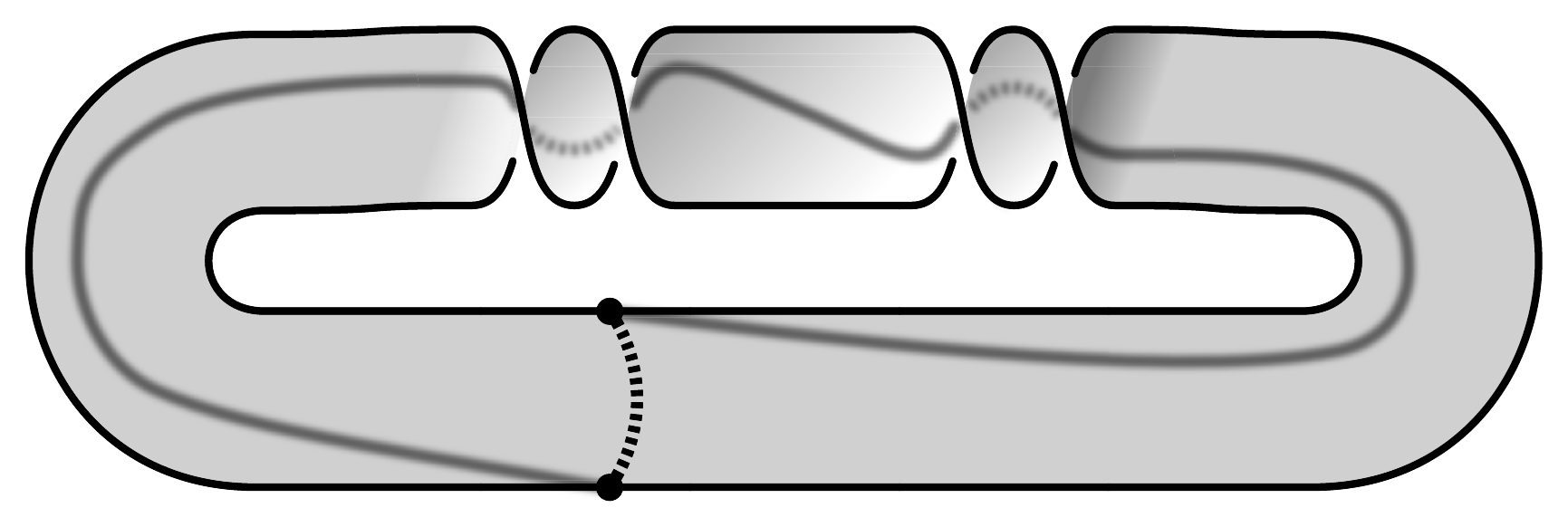
  \caption{A Seifert surface for the $(4,2)$-torus link which is not a fibre surface}
  \label{fig:torusknotband}
\end{figure}

\begin{example}
A complete bipartite graph with $p+q$ vertices, embedded as in Figure
\ref{fig:bipartite}, and with blackboard framing, is a fibre surface for the
$(p,q)$-torus link. Its monodromy is given by the \tat\ map of walk length 2.

\checkhere This is the example mentioned \vpageref{sec:introtoruslinks} in the
introduction. Norbert A'Campo has mentioned \checkhere to me that the monodromy
of torus links has been described (in the language of singularity theory) by
Frédéric Pham in 1965 (\cite{Pham1965}); he proved that the fibre surface retracts
to the join of $p + q$ roots of unity which are cyclically permuted.

To see why the statement is true, note first that the same surface can be drawn
in a more symmetric way as seen in Figure \ref{fig:bipartite-stereographic},
which shows the stereographic projection of a thickened complete bipartite graph whose
vertices lie on two Hopf circles in the 3-sphere.
Knowing this, the proof of the statement can be seen in Figure
\ref{fig:bipartite}:
A cord spanned below the surface is dragged to one spanned above it. Its
projection is given precisely by the application of a \tat\ twist to the
projection of the original cord.
\end{example}

\begin{figure}
\centering
\def\svgwidth{0.8\textwidth}
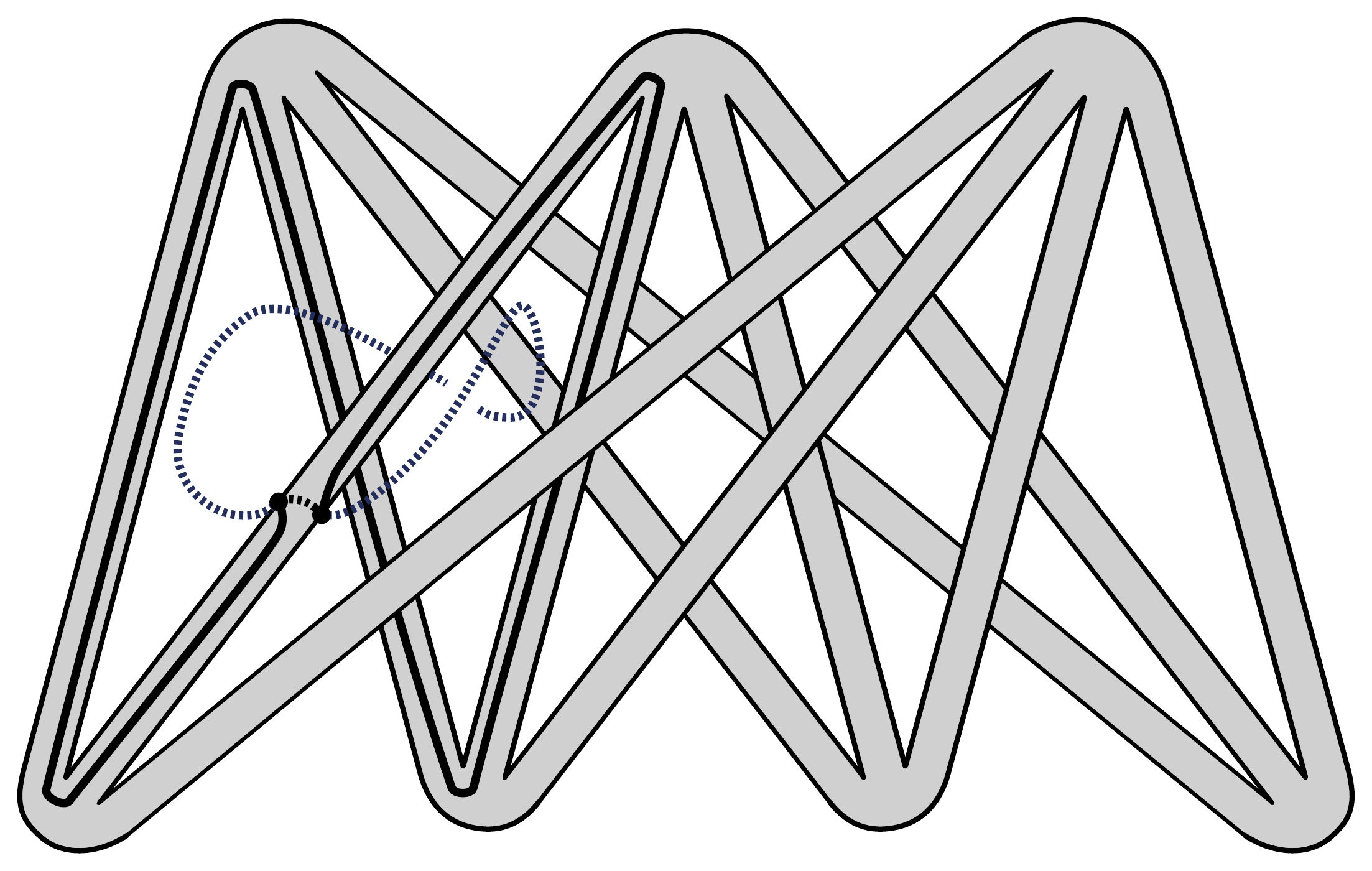
  \caption{A complete bipartite graph with $3+4$ vertices and blackboard framing, giving a fibre surface for the $(3,4)$-torus knot}
  \label{fig:bipartite}
\end{figure}

\begin{figure}
\centering
\includegraphics[width=1\textwidth]{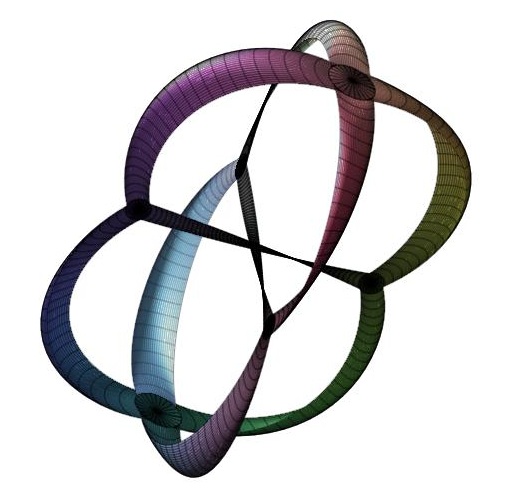}
  \caption{The same complete bipartite graph, drawn with vertices on two great circles of $S^3$ instead of two skew lines}
  \label{fig:bipartite-stereographic}
\end{figure}

%==================================================
\subsection{Murasugi sums}
From the main theorem, we can also deduce the following known result whose ``if"
part was proved by Stallings in 1978 (\cite{Stallings1978}), and the whole theorem later by Gabai (\cite{Gabai1983}).

\begin{corollary}
Let $\Sigma$ be a Murasugi sum of $\Sigma_1$ and $\Sigma_2$. Then $\Sigma$ is a
fibre surface if and only if both $\Sigma_1$ and $\Sigma_2$ are fibre surfaces.
\end{corollary}

\begin{figure}
\centering
\def\svgwidth{0.8\textwidth}
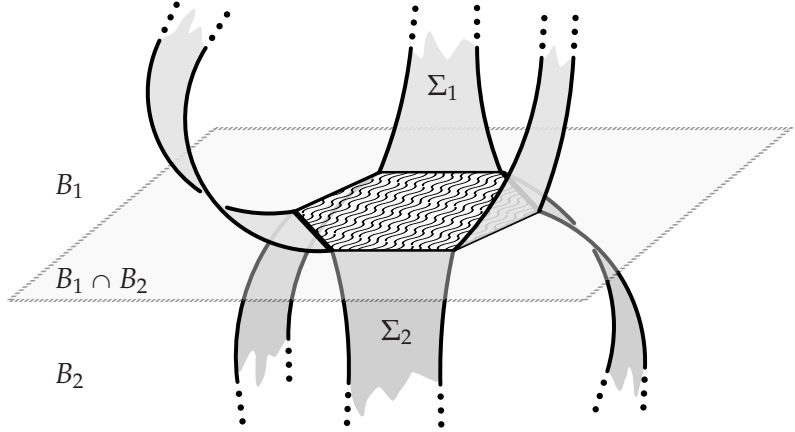
  \caption{A Murasugi sum along a hexagon}
  \label{fig:murasugi}
\end{figure}

$\Sigma$ is called a \emph{Murasugi sum} of two subsurfaces $\Sigma_1$ and
$\Sigma_2$ if $\Sigma_1 \cup \Sigma_2 = \Sigma$, their intersection is a polygon
$D$ whose edges alternatingly belong to $\Sigma_1$ and $\Sigma_2$, and there are
two balls $B_1$ and $B_2$ containing $\Sigma_1$ and $\Sigma_2$, respectively.
One usually requires that $B_1 \cup B_2 = S^3$ and $\partial B_1 \cap \Sigma_1 = \partial B_2 \cap
\Sigma_2 = D$.

This is a powerful theorem, as it permits in many cases to decompose the surface
into a Murasugi sum and then check fibredness for the simpler surfaces.

For the proof, we choose small lens thickenings $\LL(\Sigma_1)$ and $\LL(\Sigma_2)$
for the two surfaces in such a way that their upper boundaries lie inside the
respective ball, i.\,e.\@ $\partial^+\LL(\Sigma_1) \subset B_1$ and
$\partial^+\LL(\Sigma_2) \subset B_2$. We choose them in a compatible way such
that $\LL(\Sigma_1)\cup\overline{\LL(\Sigma_2)}$ fits together to a lens thickening
$\LL(\Sigma)$ of $\Sigma$, where $\overline{\LL(\Sigma_2)}$ is $\LL(\Sigma_2)$ with
the roles of ``up'' and ``down'' reversed.
\begin{proof}

It is easy to see that there exists a collection $\boldsymbol\alpha$ of disjoint
properly embedded arcs in $\Sigma_1$ that are disjoint from $D$, such that
$\Sigma_1\smallsetminus\bigcup\boldsymbol\alpha$ is a disk; likewise there is such a
collection $\boldsymbol\beta$ for $\Sigma_2$.

Assume now that $\Sigma_1$ and $\Sigma_2$ are both fibre surfaces, with
monodromies $\phi_1$ and $\phi_2$. Push the curves of $\boldsymbol\alpha$ down
to get elastic cords spanned below $\Sigma_1 \subset \Sigma$. Each of them can be
dragged to the other side of $\Sigma_1$ in $S^3\smallsetminus\Sigma_1$. Since
$\Sigma_2$ is contained in a ball which is disjoint from the elastic cord and can
be contracted to a part of $\partial^-(\Sigma)$ near $D$, the cord can clearly
still be dragged to $\partial^+(\Sigma)$ inside $S^3\smallsetminus\Sigma$.

Now take the arcs in $\boldsymbol\beta$ and push them up to get elastic cords
spanned above $\Sigma_2$, that is, below $\Sigma$. Now, $\Sigma_1$ might get in
the way of dragging these to the other side since the monodromy can well map
them to cords whose projections intersect $D$. But this can be avoided by
considering the collection $\phi_2^{-1}(\boldsymbol\beta)$ instead. The
(projections of the) corresponding cords below $\Sigma$ might intersect $D$, but
$B_1$ lies on the other side of the cords, so this is no problem and they can be
dragged. $\Sigma_2\smallsetminus\bigcup\phi_2^{-1}(\boldsymbol\beta)$ is still a
disk, $\Sigma_1\smallsetminus\left(\bigcup\boldsymbol\alpha\cup D\right)$ is a
collection of disks attached to it, therefore
$\Sigma\smallsetminus\left(\bigcup\boldsymbol{\alpha}
\cup\bigcup\phi_2^{-1}(\boldsymbol\beta)\right)$ is a disk and we have enough
elastic cords to prove that $\Sigma$ is a fibre surface.

To prove the converse, assume that $\Sigma$ is a fibre surface. Of course, this
also means, equivalently, that every elastic cord spanned above it can be dragged
to one below. Take the arcs of $\boldsymbol\alpha$, push them up to elastic cords
spanned above $\Sigma_1$ and drag them to the other side of $\Sigma$. They might
pass over $\Sigma_2$ as well, so use a retraction of $B_2$ to $D$ to get a
(possibly nonembedded) cord spanned below $\Sigma_1$.

The arcs in $\boldsymbol\beta$ are pushed down with respect to $\Sigma$, or up
with respect to $\Sigma_2$, and treated analogously.
\end{proof}

%==================================================
\chapter{Rendezvous with the mapping class group}\label{chap:mcg}
Max Dehn showed in the 1930s that the mapping class group of a surface is
generated by the set of all Dehn twists along simple closed curves (\cite{Dehn1938}). This paragraph will
relate this fact to \tat\ twists in two ways: Firstly by some considerations how a \tat\
twist can be written as the product of Dehn twists, and secondly by showing
that, vice-versa, the mapping class group can itself be generated by certain \tat\
twists.

%==================================================
\section{Bifoil and trefoil twists}\label{sec:bifoilandtrefoiltwists}
\begin{figure}
\centering
\def\svgwidth{0.75\textwidth}
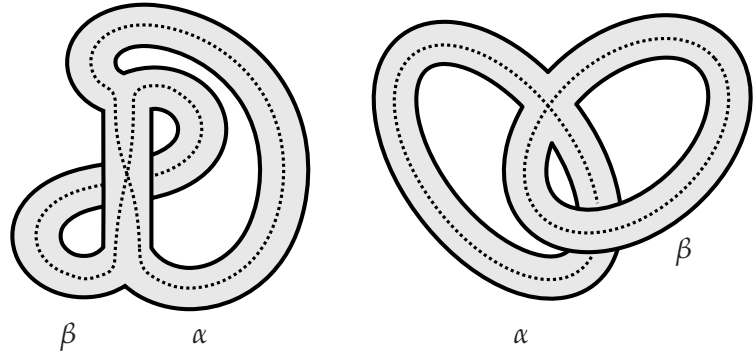
  \caption{How the tre- and bifoil twist can be factorized into Dehn twists}
  \label{fig:bifoiltrefoilfactorized}
\end{figure}
The trefoil and bifoil twists depicted in Figure \vref{fig:bifoiltrefoil} are the only
\tat\ twists on the torus. As elementary twists, their notation is $\T{E_{3,3},1}$ and
$\T{E_{2,2},1}$, respectively.
Both of these twists have a simple presentation as a product of Dehn twists
along two curves that generate the homology of the torus.
Those curves are drawn in Figure \ref{fig:bifoiltrefoilfactorized}; on the left
for the trefoil, on the right for the bifoil twist; the graphs are omitted.
One can verify the following statements by studying the images of two crossing
arcs, or any two nonseparating properly embedded arcs.
\begin{proposition}
Let $\Tr = \T{E_{3,3},1}$ and $\Bi = \T{E_{2,2},1}$ be the trefoil and bifoil twist.
Then
\begin{align}
	\Tr &= t_\alpha t_\beta = t_\gamma t_\alpha = t_\beta t_\gamma\\
		\shortintertext{and}
	\Bi &= t_\alpha t_\beta t_\alpha = t_\beta t_\alpha t_\beta,
\end{align}
for curves $\alpha$ and $\beta$ as indicated in the left and right part of Figure \ref{fig:bifoiltrefoilfactorized}, respectively, and $\gamma = t_\alpha(\beta)$.
\end{proposition}
Note that the roles of $\alpha$ and $\beta$ are symmetric for the bifoil twist, but
not so for the trefoil twist. In that case, when traversing the middle, vertical,
band of the graph, $\alpha$ meets $\beta$ coming from the right.
Also, the order of $t_\alpha$ and $t_\beta$ does matter.
When in doubt, its correctness can be checked easily:
For a crossing arc $a$ that intersects $\alpha$
but not $\beta$, $\Tr(a) = t_\alpha t_\beta(a) = t_\alpha (a)$.

%==================================================
\section{Positivity and veer of \tat\ twists}
On an oriented surface, Dehn twists come in two flavours: left and right,
depending on whether an arc which is transverse to the twisting curve is mapped to the left or to the right.
\begin{figure}
\centering
\def\svgwidth{0.6\textwidth}
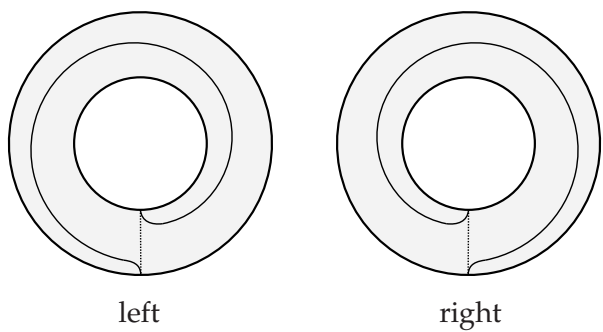
  \caption[The effects of a left (or negative) and a right (or positive) Dehn twist]%
  {The effects of a left (or negative) and a right (or positive) Dehn twist.
  The dotted interval is mapped to the one going once around.}
  \label{fig:leftandright}
\end{figure}

Often one of the two possibilities is called \emph{positive} and one
\emph{negative}, this choice is arbitrary and usage varies, but in this text
``positive'' is chosen to mean ``right''. It is interesting to restrict oneself
to only using positive (or only negative) Dehn twists and study
$Dehn^+(\Sigma)$, the monoid generated by all positive Dehn twists. For example,
Loi and Piergallini (\cite{LoiPiergallini2001}), and later Akbulut and Özbağcı
(\cite{AkbulutOzbagci2001}) proved the following theorem:

\begin{theorem}
If the monodromy of an open book can be factorized into a product of positive
Dehn twists, the open book is Stein-fillable.
\end{theorem}

This condition is fulfilled by bi- and trefoil twists, as we have seen.
There are possible strategies to factorize other twists into positive
Dehn twists, but up to now, this remains without a definite answer:
\begin{question}
Can every \tat\ twist be factorized as a product of positive Dehn twists?
\end{question}
On closed surfaces, every mapping class has this property.
This is because there are so-called \emph{positive relations}, meaning products
of positive Dehn twists about nonseparating curves which are trivial in the
mapping class group.
Because they exist, and because all Dehn twists about nonseparating curves
are conjugate, a negative twist can be written as a product of positive
twists.
As an example, choose standard homological generators $\alpha$ and $\beta$ on
the torus and check that
\[
	(t_\alpha t_\beta)^6 = \id,
\]
and therefore
\[
	t_\alpha^{-1} = t_\beta (t_\alpha t_\beta)^5.
\]
Indeed the twists $t_\alpha$ and $t_\beta$ can be represented by the action of the matrices
$\begin{psmallmatrix*}[r] 1 & 1 \\
                          0 & 1 \end{psmallmatrix*}$
and
$\begin{psmallmatrix*}[r] 1 & 0 \\
                        \!-1 & 1 \end{psmallmatrix*}$
on $\R^2/\Z^2$, from which the above relation follows.

Following Wajnryb (\cite{Wajnryb2006}), a \emph{positive relation} on a surface
with boundary is any way to write a product of Dehn twists along boundary
components as a product of positive Dehn twists.
Any \tat\ twist which can be factorized explicitly into positive Dehn twists
provides such a positive relation.
There is a proposed list of all positive relations due to Ivan Smith, but whether
it is complete is still unknown. 

A weaker property than positivity -- right veer -- that has been defined and
studied by Honda, Kazez, and Matić (\cite{HondaKazezMatic2007}), can be proved easily:
\begin{proposition}
\Tat\ twists with nonnegative walk length are right-veering.

More generally,
multi-speed \tat\ twists with nonnegative walk lengths are right-veering.
\end{proposition}
Begin \emph{right-veering} means that every properly embedded arc is moved to
the right.
More precisely:
A properly embedded arc can be lifted to the universal cover of the surface.
Looking from one of its endpoints, it will divide the universal cover into two
regions, the ``left'' and the ``right'' one.
Represent the lifted arc as well as its image geodesically, letting them share the
chosen endpoint on the boundary.
If the map is right-veering, the image must be contained in the closure
of the region to the right.

More generally, being right-veering with respect to one boundary component
means that every properly embedded arc with an endpoint on that
boundary component is mapped to the right when viewed from that endpoint.

Positive Dehn twists are right-veering.
Furthermore, one can show that compositions of positive Dehn twists are right-veering
as well, but also that there are right-veering diffeomorphisms that are not in any way
a product of positive Dehn twists.

In the same article, Honda, Kazez, and Matić consider the \emph{fractional
Dehn twist coefficient}. For a diffeomorphism 
$f \in \mathrm{Diff}(\Sigma,\partial\Sigma)$
that is freely isotopic to a diffeomorphism
$g \in \mathrm{Diff}(\Sigma)$
of finite order, say $g^k = \id_\Sigma$,
the fractional Dehn twist coefficient $c_i$ is the
amount of rotation of the $i$th boundary component in an isotopy that connects
$f$ to $g$.
This is well-defined once the mapping class of $f$ induces a unique symmetry
on the graph, which is true whenever $\Sigma$ is not a disk or an annulus.
In the language of \tat\ twists, $c_i$ is equal to $l_i/b_i$, where $b_i$ is
the length of the $i$th boundary component and $l_i$ the respective walk length.

The proposition follows from Proposition 3.2 in \cite{HondaKazezMatic2007}, which
states that a diffeomorphism is right-veering with respect to the $i$th boundary
component if and only if either $c_i>0$ or else $c_i=0$ and $c_j \geq 0$ for all other boundary
components.

%==================================================
\section{Generating the mapping class group}
Dehn twists generate the mapping class group, and a Dehn twist is a simple
example of a \tat\ twist. Like the annulus on which a Dehn twist is defined, the
ribbon graph of an arbitrary \tat\ twist can be embedded into a closed surface
and defines an element of its mapping class group. In this section, we will look
at the two next simplest examples of \tat\ twists, the bifoil twist
$\Tr$ and the trefoil twist $\Bi$ (see Section \ref{sec:bifoilandtrefoiltwists}), and show that they generate
the mapping class group as well, if the genus is high enough.

In the case of Dehn twists, it is easy to see that there are relations in the
mapping class group that allow us to write twists along separating curves as
compositions of twists along nonseparating curves. Since for every pair of
nonseparating curves there is a diffeomorphism sending one to the other, all
such twists are conjugate, and thus the mapping class group is generated by a
single conjugacy class of Dehn twists. The same is true for generation by
bifoil and trefoil twists, but for an even better reason: There is, up to
diffeomorphism, only one way to embed them into a surface.

\begin{theorem}\label{thm:generate}
Let $\Sigma$ be a closed surface of genus at least 3, let $G$ be one of
$E_{2,2}$ or $E_{3,3}$. Then $\mcg(\Sigma)$ is generated by \tat\ twists along
all embeddings of $G$. Moreover, all those twists are conjugate.
\end{theorem}

Since, unlike Dehn twists, $\Tr$ and $\Bi$ act on the homology
of $\Sigma$ by finite order, we immediately get:
\begin{corollary}
The symplectic group $\mathrm{Sp}(2n,\Z)$, $n \geq 3$, is generated by a set of
conjugate torsion elements of order $4$, and also by a set of conjugate
torsion elements of order $6$.\qed
\end{corollary}
By calculating the action on homology, one finds that those elements are conjugate to
\[\renewcommand{\arraystretch}{0.9}%
\begin{pmatrix*}[r]
0  &  1 &   &        & \\
-1 &  0 &   &        & \\
   &    & 1 &        & \\
   &    &   & \ddots & \\
   &    &   &        & 1 
\end{pmatrix*}
\text{\quad and \quad}
\begin{pmatrix*}[r]
\ 0  & -1 &   &        & \\
  1  &  1 &   &        & \\
     &    & 1 &        & \\
     &    &   & \ddots & \\
     &    &   &        & 1 
\end{pmatrix*},
\]
respectively.

To prove the theorem, it is not sufficient to note that a power of these \tat\
twists is a Dehn twist, because this twist happens along a separating curve.
This could only be used to show that they generate the \emph{Johnson kernel},
the subgroup of $\mcg(\Sigma)$ generated by separating twists, which is itself
contained in the \emph{Torelli group}, the subgroup acting trivially on
homology.

We need the factorization of bifoil and trefoil twists into Dehn twists from
Section \ref{sec:bifoilandtrefoiltwists}. In both cases, we have two curves
$\alpha$ and $\beta$ transversely intersecting in exactly one point,
\begin{align*}
	\Tr &= t_\alpha t_\beta
	\shortintertext{and}
	\Bi &= t_\alpha t_\beta t_\alpha.
\end{align*}
And both \tat\ graphs $E_{2,2}$ and $E_{3,3}$ live on a one-holed torus which
deformation retracts to the union $\alpha\cup\beta$. Moreover, every two choices
of such a pair of curves are related by a diffeomorphism of the whole surface;
this follows directly from the classification of surfaces when one cuts, one at
a time, along the two curves.
Uniqueness up to diffeomorphism proves the statement that all bi- or trefoil
twists in a surface are conjugate.

In what follows, ``$\alpha\caponept\beta$'' means ``$\alpha$
intersects/intersecting $\beta$ transversely in one point''.
\begin{proposition}
Given simple closed curves $\alpha$ and $\beta$ in a closed surface $\Sigma$,
$\alpha\caponept\beta$, there is a an embedded graph $G\subset\Sigma$
representing $E_{3,3}$ such that $\T{G,1} = t_\alpha t_\beta$. Likewise, there is
$H\subset\Sigma$ representing $E_{2,2}$ such that $\T{H,1} = t_\alpha t_\beta
t_\alpha$.
\end{proposition}
\begin{proof}
This is actually already proved by the fact that there is just one choice of
$\alpha$ and $\beta$ up to diffeomorphism, so if we embed $E_{3,3}$ or $E_{2,2}$
in any way we will only have erred by a diffeomorphism.

More explicitly, $E_{2,2}$ will of course just be embedded as $\alpha\cup\beta$
with its natural structure as a graph with one vertex $\alpha\cap\beta$, and
give us $t_\alpha t_\beta t_\alpha$, which is the same as $ t_\beta t_\alpha
t_\beta$. There is no choice involved here.

\begin{figure}
\centering
\def\svgwidth{0.6\textwidth}
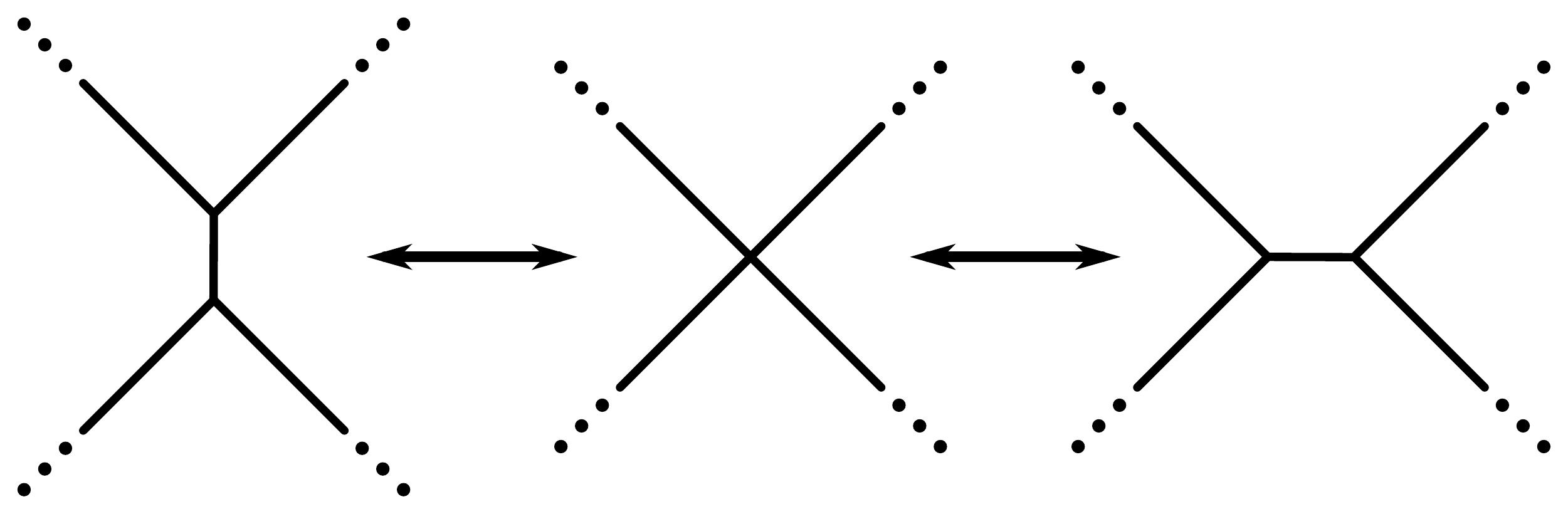
  \caption{Finding embeddings of $E_{3,3}$}
  \label{fig:expandvertex}
\end{figure}

On the other hand, when we expand the single vertex of $\alpha\cup\beta$ to two
trivalent vertices, we find $E_{3,3}$. There are two ways to do this,
corresponding to two different embeddings of $E_{3,3}$ that give back either
$t_\alpha t_\beta$ or $t_\beta t_\alpha$.
\end{proof}
\begin{proposition}\label{prop:generatedoubletwist}
Let $\alpha$ and $\gamma$ be two disjoint nonseparating simple closed curves.
Then $t_\alpha t_\gamma^{-1}$ can be written as a product of bifoil twist as
well as as a product of trefoil twists (and their inverses).
\end{proposition}
\begin{proof}
The union $\alpha\cup\gamma$ either disconnects the surface into two pieces, or
it is still nonseparating. In either case, it is easy to see that there is a
curve $\beta$ such that $\alpha\caponept\beta$ and $\beta\caponept\gamma$.

For trefoil twists, we use
\[
	\left(t_\alpha t_\beta\right)\left(t_\gamma t_\beta\right)^{-1}
	= t_\alpha t_\beta t_\beta^{-1} t_\gamma^{-1}
	= t_\alpha t_\gamma^{-1}
\]
and notice that both factors on the left-hand side can be realized as trefoil twists.

For bifoil twists, we use a twist along the $E_{2,2}$ \tat\ graph
$t_\beta^{-1}(\alpha) \cup \beta$, an inverse twist along $t_\beta^{-1}(\gamma)
\cup \beta$, and calculate their product:
\begin{multline*}
	\left( t_\beta^{} t_{t_\beta^{-1}(\alpha)} t_\beta^{} \right)
	\left( t_\beta^{} t_{t_\beta^{-1}(\gamma)} t_\beta^{} \right)^{-1}\\
	=
	t_\beta^{}   t_{t_\beta^{-1}(\alpha)}      t_\beta^{}\enspace
	t_\beta^{-1} t_{t_\beta^{-1}(\gamma)}^{-1} t_\beta^{-1}\\
	=
	\Bigl( t_\beta^{} t_{t_\beta^{-1}(\alpha)}      t_\beta^{-1} \Bigr)
	\Bigl( t_\beta^{} t_{t_\beta^{-1}(\gamma)}^{-1} t_\beta^{-1} \Bigr)\\
	=
	t_\alpha^{} t_\gamma^{-1}.
\end{multline*}
\end{proof}
The last equation uses an identity in mapping class groups which is obvious when
stated in a more general form: For any curve $\delta$ and any diffeomorphism
$\varphi$,
\[
	\varphi t^{}_\delta \varphi^{-1} = t_{\varphi(\delta)}.
\]

\begin{figure}
\centering
\def\svgwidth{0.9\textwidth}
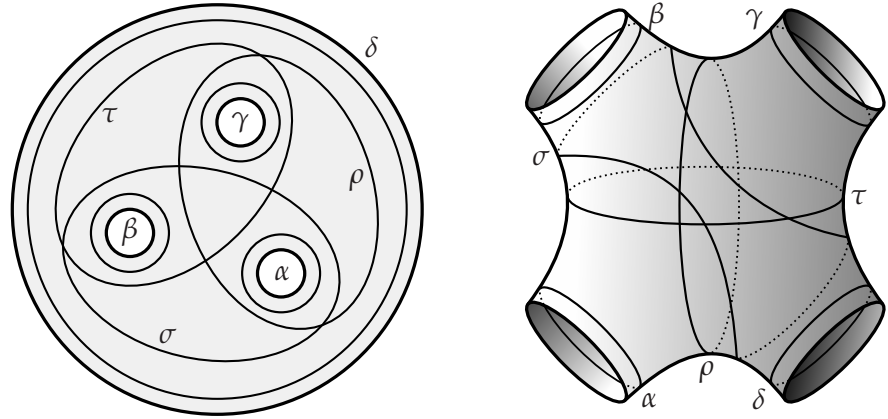
  \caption{Two views of the lantern relation}
  \label{fig:lantern}
\end{figure}
We are now ready to finish the proof of Theorem \ref{thm:generate}. We will show
how to write a Dehn twist along a nonseparating curve as a product of bi- or
trefoil twists. This part uses the \emph{lantern relation}, an important
equation between a product of three Dehn twists one side and four on the
other that has been found by Dehn and later by Johnson (\cite{Johnson1979}). The
``lantern'' is the four-holed sphere in Figure \ref{fig:lantern}, and we have that
\[
	t_\rho t_\sigma t_\tau = t_\alpha t_\beta t_\gamma t_\delta.
\]
\begin{figure}
\centering
\def\svgwidth{0.6\textwidth}
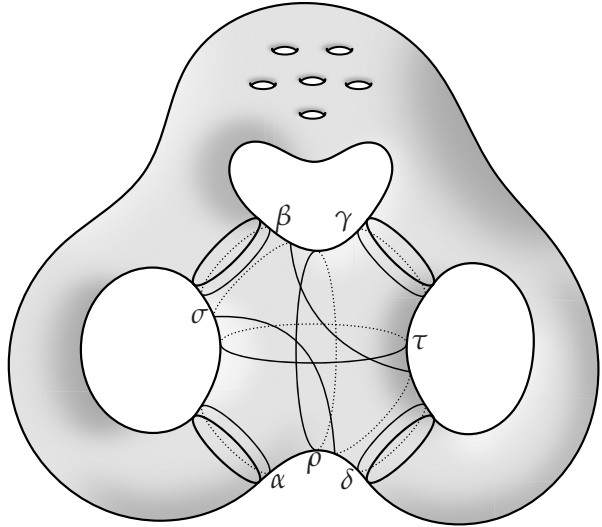
  \caption{A lantern in a surface of genus at least $3$}
  \label{fig:greatlantern}
\end{figure}
When the genus of our surface is at least $3$, we can embed the lantern in such a way that
all curves involved are nonseparating. This is shown in Figure \ref{fig:greatlantern}.
The trick is to rewrite the lantern relation in the following way:
\[
	t_\delta = \bigl( t_\rho^{}   t_\alpha^{-1} \bigr)
               \bigl( t_\sigma^{} t_\beta^{-1}  \bigr)
               \bigl( t_\tau^{}   t_\gamma^{-1} \bigr)
\]
Together with the previous proposition, this demonstrates the theorem.\qed

%==================================================
\subsection{Small genus}
Neither bi- nor trefoil twists generate the mapping class group of the torus;
and trefoil twists do not generate the mapping class group of a surface of genus
$2$.

Would they generate, they would map to a single generator of the abelianization
$\mcg^{ab}$ because they are all conjugate. But the same is true for Dehn twists
along separating curves since they also generate the mapping class group; call
their image $1 \in \mcg^{ab}$. $\T{E_{2,3}}$ now gets mapped to $2$,
$\T{E_{2,2}}$ to $3$.

The abelianization of $\mathrm{SL}(2,\Z)$, the mapping class of the torus, is
$\Z/{12\Z}$, which is not generated by $2$ nor by $3$. For the surface of genus
$2$ it is $\Z/{10\Z}$, which is not generated by $2$.

\subsubsection{Trefoil twists in genus $2$}
There is a presentation for the mapping class group of genus $2$, suggested by
Bergau and Mennicke (\cite{BergauMennicke1960}), proved to be correct by Birman
and Hilden (\cite{BirmanHilden1971}).
From it, can derive the abelianization $\Z/{10\Z}$ mentioned before and conclude:
When we have a diffeomorphism of the surface of genus $2$, the abelianization
allows us to count the number of Dehn twists along nonseparating curves needed
to write it, modulo $10$. In particular, it allows for a nontrivial
homomorphism to $\Z/2\Z$, i.\,e.\@ to count modulo two and distinguish
between ``even'' and ``odd'' diffeomorphisms.

It can be shown that trefoil twists in genus $2$ generate the normal subgroup of
even diffeomorphism. It is obvious that every product of trefoil
twists is even. And the converse is also true, as the following proposition shows.

\begin{proposition}
For any pair of nonseparating curves $\alpha$ and $\gamma$, $t_\alpha^{}
t_\gamma^{}$, $t_\alpha^{-1} t_\gamma^{-1}$, $t_\alpha^{} t_\gamma^{-1}$, and
$t_\alpha^{-1} t_\gamma^{}$ can be written as a product of trefoil twists.
\end{proposition}
We need a small lemma which is proved by techniques similar to the ones in Lickorish's proof
that the mapping class group is generated by Dehn twists.
\begin{lemma}
For any pair of nonseparating curves $\alpha$ and $\gamma$, intersecting
transversely, there is a chain of curves
$\alpha=\beta_0\caponept\beta_1\caponept\ldots\caponept\beta_{k-1}\caponept\beta_k=\gamma$.
\end{lemma}
\begin{proof}
\begin{figure}
\centering
\def\svgwidth{0.9\textwidth}
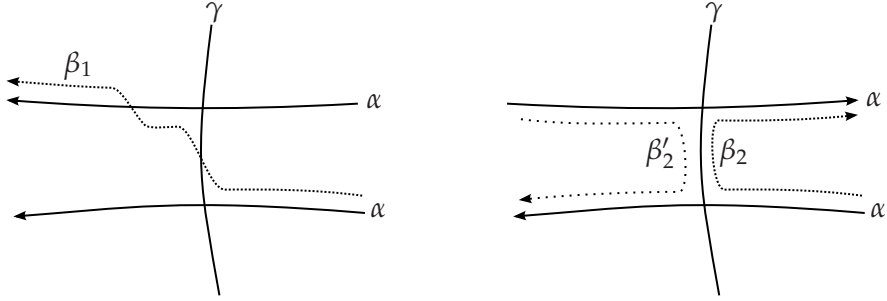
  \caption{Modifications between two crossings}
  \label{fig:curveslikelickorishs}
\end{figure}
If $\alpha\caponept\gamma$, we are done. If $\alpha\cap\gamma=\varnothing$, we
choose a curve $\beta$ with $\alpha\caponept\beta\caponept\gamma$. Else, choose
one arc of $\gamma$ between two intersection points with $\alpha$. We have one
of the two configurations pictured in Figure \ref{fig:curveslikelickorishs},
depending on whether $\alpha$ passes twice from the same side or not. Modify
$\alpha$ as in the pictures.

In the first situation, we find $\beta_1$ such that
$\alpha\caponept\beta_1$ and $1\leq\#\beta_1\cap\gamma<\#\alpha\cap\gamma$. Since
$\alpha\caponept\beta_1$, $\beta_1$ is still nonseparating. Proceed by induction.

In the second situation, we find $\beta_2$ with $\alpha\cap\beta_2=\varnothing$
and $\#\beta_2\cap\gamma\leq\#\alpha\cap\gamma-2$. There is a caveat: $\beta_2$
may well be separating. But as shown in the picture, there is a second possible choice
$\beta_2'$, and $[\beta_2]+[\beta_2']=[\alpha]\neq 0\in H_1(\Sigma)$, so one of
$\beta_2$ and $\beta_2'$ is nonseparating. Assume it is $\beta_2$. Choose
$\beta_1$ such that $\alpha\caponept\beta_1\caponept\beta_2$. Proceed by
induction. When at some point $\beta_k\caponept\gamma$, we are done; when
$\beta_{k-1}\cap\gamma = \varnothing$, we choose $\beta_{k}$ with
$\beta_{k-1}\caponept\beta_k\caponept\gamma$.
\end{proof}
\begin{proof}[Proof of the proposition] Choose such a chain for the two twist
curves $\alpha$ and $\gamma$. $t_\alpha^{} t_\gamma^{-1}$, and analogously
$t_\alpha^{-1} t_\gamma^{}$, are then a product of bifoil twists by induction
and Proposition \ref{prop:generatedoubletwist} from above.

Choose some curve $\beta\caponept\gamma$. Then
\begin{multline*}
	\left(t_\gamma t_\beta\right)^2\left(t_\beta t_\gamma\right)^{-1}\\
	= t_\gamma\left(t_\beta t_\gamma t_\beta\right)\left(t_\beta t_\gamma\right)^{-1}\\
	= t_\gamma\left(t_\gamma t_\beta t_\gamma\right)\left(t_\beta t_\gamma\right)^{-1}\\
	= t_\gamma^2.
\end{multline*}
Therefore, we are also able to get $t_\alpha t_\gamma = t_\alpha t_\gamma^{-1} t_\gamma^2$ and its
inverse, $\left(t_\alpha t_\gamma\right)^{-1}$.
\end{proof}

\subsubsection{Bifoil twists in genus 2}
\begin{figure}
\centering
\def\svgwidth{0.5\textwidth}
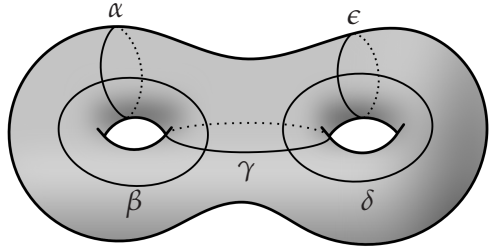
  \caption{Generators for the surface of genus $2$}
  \label{fig:genus2mcg}
\end{figure}
$3$ generates $\Z/10\Z$, so could bifoil twists also generate the mapping class
group of a genus $2$ surface?

Indeed they do; we can write a Dehn twist along a nonseparating curve as a
product of 21 bifoil twists or their inverses. We need a relation in the mapping
class group of genus $2$, to be found for example in the survey article by
Ivanov (\cite[p.~568]{Ivanov2001}). Set $A = t_\alpha$, $B = t_\beta$, $C =
t_\gamma$, $D = t_\delta$, and $E = t_\epsilon$ for the curves in picture
\ref{fig:genus2mcg}. Then
\[
   (ABC)^4 = E^2.
\]
As we have seen, products of the form $CA^{-1}$, for example, are a product
of a bifoil twist and an inverse bifoil twist. Thus we can form
\[
	\bigl((CBC)^{-1}(CA^{-1})\bigr)^4 = (ABC)^{-4} = E^{-2},
\]
using 4 bifoil twists and 8 inverses, and
\[
	(EA^{-1})\bigl((AD^{-1})(DED)(D^{-1}A)\bigr)(A^{-1}E) = (EA^{-1})(AEA)(A^{-1}E) = E^3,
\]
using 5 bifoil twists and 4 inverses. Hence we can form $E$, a twist along a
nonseparating curve, and thus produce the entire mapping class group.

%==================================================
\section{Roots of \tat\ and Dehn twists}
A diffeomorphism $\rho$ is called a \emph{root} of a diffeomorphism
$\phi$ if $\phi$ is a nontrivial power of $\rho$.

%==================================================
\subsection{Which \tat\ twists have roots?}
Some \tat\ twists are obviously powers of others, namely when the walk
length is not minimal: $\T{G,kl}^{} = \T{G,l}^k$.
In general, when a \tat\ twist $\T{G,l}$ is a power of another diffeomorphism
$\rho$, $\rho$ is again (freely) of finite order, hence we know from Chapter
\ref{chap:periodic} that $\rho$ is again a \tat\ twist, say $\rho = \T{G',l'}$
and $\T{G,l} = \T{G',kl'} = \rho^k$.
$G$ and $G'$ are spines for the same surface.

The graphs $G$ and $G'$ need not be equal, unfortunately.
For instance, $G'$ can have more symmetries than $G$, and could be obtained from
$G$ by the collapse of an edge orbit that consist of contractible components
(see Section \ref{sec:equivalence}).
It is also possible that $G$ has more symmetries, e.\,g.\@ when $\T{G,l}$
is a composition of Dehn twists around boundary components, in which case
it can be a bouquet of circles.
Even if the answer to the question \vpageref{q:collapsible} is ``yes'',
one would therefore have to introduce as well as remove contractible
edge orbits in order to see whether a \tat\ twist (or any periodic
map) has a root.

We see interesting examples when we look at powers of elementary twists.
They are best studied as chord diagrams, remembering that chords correspond
to edges and internal boundaries to vertices.
It is often possible to collapse an edge orbit and find roots of powers.
Recall also that collapsing an edge means removing a chord; see Section \ref{sec:buildingribbongraphs}.
As an example, the periodic map $E_{12,7}^3$ (short for $\T{E_{12,7},2}^3$)
on the surface of genus $3$ with one boundary component is the same as
$E_{8,5}^2$, and hence has a square root.
Or consider the fourth power $E_{12,7}^4$ --
it is equal to $E_{9,5}^3$, and hence has a cube root.

%==================================================
\subsection{Monodromies have no roots}
A monodromy of a fibred knot or link can never be a nontrivial power.
``Monodromy'' must be understood in the sense of Section \ref{sec:openbooks},
as a class of diffeomorphisms that fix the boundary, modulo isotopies that
fix the boundary; see also Example \vref{exmpl:boileau}.
The statement follows from the answer to a question asked by Paul A. Smith in  
the 40's of the last century when he could prove that the fixed point
set of a nontrivial finite-order orientation-preserving homeomorphism of the
3-sphere can either be empty or be a circle.
Smith then asked whether such a circle could be knotted; the answer turns
out to be ``no'' for diffeomorphisms:
\begin{SMITH}
	Let\/ $\phi$ be a nontrivial finite-order orientation-preserving diffeomorphism
	of\/ $S^3$ that has fixed points.
	Then the fixed point set of\/ $\phi$ is an unknot.
\end{SMITH}
The proof of this result was built upon the work of many mathematicians and was
finally assembled by Cameron Gordon.
Those efforts are described in a book (\cite{MorganBass1984}).

We can conclude from the Smith conjecture:
\begin{corollary}
	A nontrivial power of a nontrivial mapping class is never the monodromy of a
	fibred link.
\end{corollary}
\begin{proof}
	Assume that $\phi = \rho^n$ is the monodromy of a fibred link with fibre
	surface $\Sigma$.
	The open book $M_\phi = M_{\rho^k} \cong S^3$ can be constructed from $n$ copies
	of $\Sigma \times [0,1]$, labelled as $\Sigma \times [0,1] \times \Z/n\Z$,
	by identifying $(p,1,k)$ with $(\rho(p),0,k+1)$ for $0 \leq k \leq n-1$ and
	collapsing the boundary.
	Define a diffeomorphism $f \colon S^3 \to S^3$ by
	$f \colon (p,t,k) \mapsto (p,t,k+1)$.
	$f$ is well-defined since it respects the gluings.
	If $k>1$, its fixed point set is precisely $\partial\Sigma$,
	so by Smith's result $\partial\Sigma$ is connected.
	And $f^n = \id_{S^3}$, so the Smith conjecture implies that $\partial\Sigma$
	is an unknot.
\end{proof}

\begin{remark}
	The Smith conjecture in fact implies, as shown by Smith himself, that such a
	diffeomorphism is conjugate via diffeomorphism to an action of
	$\mathrm{SO}(4)$, where $S^3$ is considered as the unit sphere in $\R^4$.
	This looks like a direct analogue to Kerékjártó's lemma
	\vpageref{lem:kerekyarto}, which is also true for the sphere $S^2$.
	But the Smith conjecture is not true for homeomorphisms:
	Montgomery and Zippin, based on constructions by Bing, gave examples of
	homeomorphisms of the 3-sphere whose fixed-point set is a wild knot
	(see \cite{MontgomeryZippin1954}).
\end{remark}

The following example, suggested by Michel Boileau, shows that monodromies can
be nontrivial powers when we consider diffeomorphisms that can move the boundary,
modulo isotopies that can move the boundary.
To keep the naming conventions consistent, this mapping class is called the
\emph{free monodromy} in the example.
Two remarks should be made first:
\begin{remark}
Unlike for links (see Section \ref{sec:trivialmonodromy}), the free monodromy of a nontrivial \emph{knot} is always nontrivial.
If it were trivial, the monodromy would be freely periodic
(with order one), hence the knot would need to be a torus knot (see Theorem \ref{thm:finmon}).
But the free monodromy of (nontrivial) torus knots is also nontrivial, as we can see
from the description by complete bipartite graphs.
\end{remark}
\begin{remark}
A knot $K$ in $S^3$ is fibred if and only if its complement is fibred, that is, if there is a
fibration $\pi \colon S^3 \smallsetminus K \to S^1$.
This is, for example, true because of Stalling's fibration criterion.
Moreover, two fibrations of a knot complement are isotopic, therefore $\pi$ indeed
extends to a fibration for $K$, i.\,e.\@ the fibres of $\pi$ are Seifert surfaces for $K$.
\end{remark}
\begin{example}\label{exmpl:boileau}
The pretzel knot $K = P(-2,3,7)$ is a famous example of a knot with \emph{lens space surgery}:
Both $18$- and $19$-surgery on $K$ produce lens spaces, namely $L(18, 5)$ and $L(19, 8)$,
respectively, as shown by Fintushel and Stern in \cite{FintushelStern1980}.

When we do surgery on $K$, say $18$-surgery, we get an induced knot $K'$ in the new manifold, namely
the soul of the surgered solid torus.
A meridian of $K$ will induce a generator of the fundamental group of the lens space,
which is a cyclic group of order $18$.
An $18$-fold (unbranched) cover of this lens space is the 3-sphere, and it contains a new
knot $K''$ which is the preimage of $K'$ under the cover.
The exterior of $K''$ is the mapping torus of the $18$th power of $K$'s (free) monodromy.
Hence it is also fibred, and $K''$ is a fibred knot whose free monodromy is an $18$th power.
\end{example}

%==================================================
\subsection{Dehn twists have roots}
This is the title of a 2-page paper by Margalit and Schleimer
(\cite{MargalitSchleimer2008}), where they prove that every Dehn twist
along a nonseparating simple closed curve on a closed, connected, orientable
surface of genus at least two has a nontrivial root.
As they note, it easy to find a square root of a Dehn twist along a
\emph{separating} curve:
The curve cuts the surface into two halves; one can put the surface into a standard
position where one half is on the left and one on the right and then twist the left
half by 180 degrees.
We have seen that \tat\ twist with one boundary component give rise to many more roots
of separating Dehn twists.

The case of nonseparating curves is less obvious.
Note first, as mentioned by Margalit and Schleimer, that Dehn twists on the
torus have no roots; this comes from the fact that the mapping class group of the torus is $\mathrm{SL}(2,\Z)$ and
the matrix
$\begin{psmallmatrix*}[r] 1 & 1 \\
                          0 & 1 \end{psmallmatrix*}$,
which represents a Dehn twist in an appropriate basis, has no roots.
Indeed it has trace 2 and therefore, in the classification of matrices of $\mathrm{SL}(2,\R)$,
is parabolic.
That is to say, its action on the upper half plane has a unique fixed point
on the boundary, which in this case is the point $\infty$, and the same must be
true for every root of the matrix.
Since the matrices in $\mathrm{SL}(2,\R)$ which fix $\infty$ are all of the form
$\begin{psmallmatrix*}[r] 1 & x \\
                          0 & 1 \end{psmallmatrix*}$
for $x \in \R$, the claim follows.

For higher genus, however, we can always find roots, and we can use \tat\ twists to do so.
A Dehn twist $t_\alpha$ is obviously reducible with reduction system $\alpha$, and the following
lemma shows that the same is true for any root of $t_\alpha$: 
\begin{lemma}
Let $\alpha$ be an isotopy class of a simple closed curve on a closed, connected,
orientable surface.
Let $\rho$ be a mapping class such that $\rho^k = t_\alpha$
for some $k$.
Then $\rho(\alpha) = \alpha$.
\end{lemma}
\begin{proof}
$t_\alpha = \rho^k = \rho \circ \rho^k \circ \rho^{-1}
 = \rho \circ t_\alpha \circ \rho^{-1}
 = t_{\rho(\alpha)}$,
thus $\rho(\alpha) = \alpha$.
\end{proof}

We can therefore cut the surface along $\alpha$ and examine the root
as a diffeomorphism defined on a surface with two boundary components.

\begin{question}
	How can one show, directly, that a root will not exchange the two sides
	of $\alpha$, and that therefore the induced diffeomorphism of the cut surface preserves
	the two boundary components setwise?
\end{question}
In that case, by definition, this induced diffeomorphism is freely of finite order; hence it is described
by a multi-speed \tat\ twist with two boundary components, as seen in Chapter
\ref{chap:periodic}.
Note that, by Corollary \ref{cor:tatcusps}, we can choose the \tat\ graph to have an embedded
circle around one of the boundary components, which simplifies the set of graphs to consider.
This makes it possible to use a chord diagram of the graph along the second boundary component, together with
instructions on how to glue it to the circle around the first one.

Some power $\T{}^k$ of the \tat\ twist $\T{}$ will consist of a composition of Dehn twists around the
two boundary components.
We cannot choose one walk length to be zero because this would imply that we must choose
the walk length for the other boundary component so as to induce a power of a Dehn twist;
thus we would not get a nontrivial root.
Moreover, we must choose different signs for the two walk lengths, such that $\T{}^k$
consists of positive as well as negative Dehn twists that add up to a single
one on the original surface.

Modulo the question, \tat\ twists should allow for a complete description of
all possible roots of Dehn twists along nonseparating curves.
For example, those roots are subject to the conditions on the bounds for periodic
maps given in Chapter \ref{chap:periodic}.
Note that different embeddings of the same \tat\ graph into the same (cut) surface,
of which there are usually many, lead to different, but conjugate roots.
They are related by some map $\psi$ on the closed surface such that
$\psi(\alpha) = \alpha$, and when $\rho^k = t_\alpha$, then
$(\psi\rho\psi^{-1})^k = t_\alpha$.
\begin{example}
\begin{figure}
\centering
\def\svgwidth{0.375\textwidth}
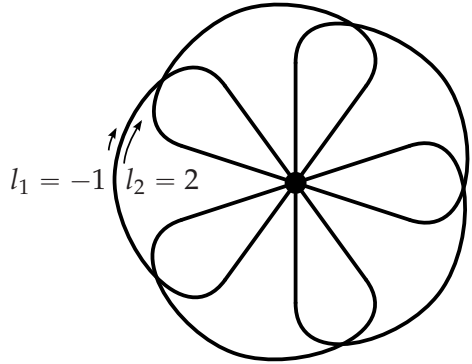
  \caption{A fifth root on a surface of genus 3}
  \label{fig:fifthroot}
\end{figure}
Figure \ref{fig:fifthroot} describes a fifth root of a Dehn twist along a nonseparating curve on a closed surface of genus 3.
The walk lengths are $l_1=-1$ and $l_2=2$.
The graph $G$ describes a surface of genus 2 with two boundary components
$\alpha_1$ and $\alpha_2$, both of length 5, which we identify to get the
nonseparating curve $\alpha$ for the Dehn twist.
With the chosen walk lengths,
$\T{G,-1,2}^5 = t_{\alpha_1}^{-1} t_{\alpha_2}^2 = t_\alpha$.

This example has been found by a chord diagram with two circles;
however, the symmetry is less obvious from the drawing than in the
single-boundary case.
\end{example}

\begin{example}
\begin{figure}
\centering
\def\svgwidth{0.8\textwidth}
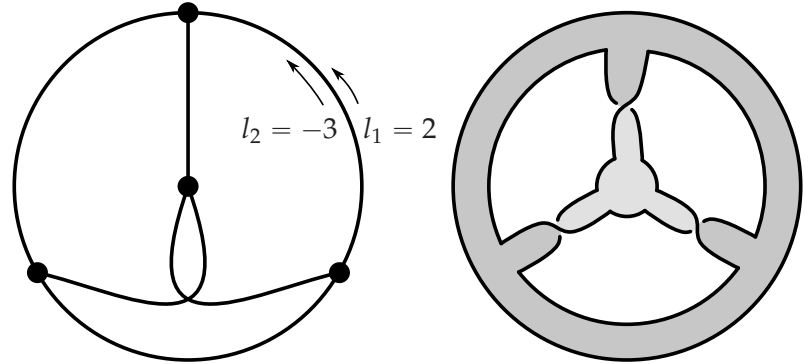
  \caption[A third root on a surface of genus 2]%
  {A third root on a surface of genus 2, to the left with blackboard
  framing, to the right, more symmetrically, as a ribbon graph}
  \label{fig:thirdroot}
\end{figure}
Figure \ref{fig:thirdroot} describes a third root of a Dehn twist along a nonseparating
curve on a closed surface of genus 2.
The walk lengths are $l_1=2$ and $l_2=-3$.
The graph $G$ describes a surface of genus 1 with two boundary components
$\alpha_1$ and $\alpha_2$, one of length 3, the other of length 9, which we identify like before to get the
nonseparating curve $\alpha$ for the Dehn twist.
With the chosen walk lengths,
$\T{G,2,-3}^3 = t_{\alpha_1}^{2} t_{\alpha_2}^{-1} = t_\alpha$.

This example corresponds to one constructed by Margalit and Schleimer:
Take the square of the monodromy of the trefoil.
This map leaves the two trivalent vertices invariant, so we can blow them up to
a circle.
When we cap off the third (original) boundary component, we get an embedded
graph on a two-holed torus.
Using the collapse from Chapter \ref{chap:periodic}
(see Figure \ref{fig:leveebreach}), we can get rid of the enclosed disk and get
the \tat\ twist from the picture.
\end{example}

%==================================================
\chapter{The computer program ``t.a.t.''}\label{chap:software}
The computer program ``t.a.t'' allows for some experiments and calculations with \tat\
twists. It has been written in the programming language Java and is best run
from within an integrated development environment (IDE) like the free application
``Eclipse'', where the code for the experiments can be easily modified and run.

%==================================================
\section{Features}
The main purpose of the program is to take a chord diagram and a compatible
walk length, display the diagram graphically and write down a presentation of
the fundamental group of the corresponding open book as a string that can
be given to the computer algebra system GAP.
Since it relies on chord diagrams, only \tat\ graphs with one boundary component
can be studied.

The class \code{TaTTest} contains the \code{main(\ldots)} method which is
executed.
It also provides some documented sample use cases, mostly for specific twists,
which can be called from \code{main(\ldots)}.
\code{TaTTest} provides some additional methods, like counting all
elementary \tat\ twists of a given order and a given genus, or finding the
unique elementary twist of order $3g$ or $3g+3$, respectively (verifying Lemma
\ref{lem:elementarygraphgenus}).

\code{ChordDiagram} is the class which represents a chord diagram.
It provides methods to calculate the genus of the \tat\ graph, a list of vertices,
a possible isomorphism with another chord diagram, and its rotational
symmetry and thus the minimal walk length for the graph.

Chord diagrams are constructed by the \code{ChordDiagramFactory}.
This factory class provides additional methods to create random chord diagrams,
random diagrams with a given symmetry, elementary chord diagrams, and chord
diagrams for torus knot monodromies.

Some simple mathematical methods are contained in the \code{MoreMath}
class, among others one to generate a list of pairs of coprime integers,
if one is interested in torus knot monodromies.

Finally, \tat\ twists are described by a chord diagram together with a walk length,
which is encapsulated by the \code{ChordTaTTwist} class, and
objects of this class can be asked for a presentation of the fundamental group
of their corresponding open book.

\appendix

%==================================================
\chapter{Java code extracts}
\lstset{language=Java, breaklines=true, breakatwhitespace=true,}
\lstset{basicstyle=\small, tabsize=3,}
%\lstset{prebreak=\raisebox{0ex}[0ex][0ex]
%        {\ensuremath{\hookleftarrow}}}
\lstset{postbreak=\raisebox{0ex}[0ex][0ex]
        {\ensuremath{\hookrightarrow\space}}}
%\lstset{numbers=left, numberstyle=\scriptsize}
This appendix contains the Java code which writes down a presentation
of the fundamental group of the open book that corresponds
to a \tat\ twist, as described in Section \ref{sec:fundamentalgroup}.
The following code in the class \code{ChordTaTTwist} calculates the generators.
The method \code{getOpposite(\ldots)} that is used here is called with
a number that denotes an endpoint of a chord in a chord diagram; it returns
the opposite endpoint of the chord. \code{getSize()} returns the number of
chords.

\begin{lstlisting}
/**
 * Calculates the generators of the fundamental group.
 * 
 * @param inGenList
 *			the list of generators
 */
private void fillGenerators(List<String> inGenList) {
	// add a generator for each chord
	for (int i = 0; i < 2 * diagram.getSize(); i++) {
		if (diagram.getOpposite(i) > i) {
			// add a generator of the form "g0", "g27"
			// or the like, labelled by the first
			// endpoint of the chord
			inGenList.add(GEN_PREFIX + i);
		}
	}

	// add a generator for "going once around the circle"
	inGenList.add(GEN_OMEGA);
}
\end{lstlisting}

A method in the same class writes down the relations that come from the interior
boundaries of the chord diagram, or the vertices of the graph:

\begin{lstlisting}
/**
 * Calculates the relations coming from the
 * interior boundaries of the chord diagram.
 * 
 * @param inRelList
 *			the list of relations
 */
private void fillBoundaryRelations(List<String> inRelList) {
	// iterate over the list of boundaries
	for (List<Integer> aBoundary : diagram.getBoundaries()) {
		StringBuffer aRelation = new StringBuffer();
		boolean isEmpty = true;
		// iterate over chords in the boundary component
		for (Iterator<Integer> anIterator = aBoundary.iterator(); anIterator.hasNext();) {
			int aChordEnd = anIterator.next().intValue();
			int aChordStart = diagram.getOpposite(aChordEnd);
			// append a multiplication sign if necessary
			if (!isEmpty) {
				aRelation.append("*");
			}
			// append the name of the chord or its inverse
			if (aChordStart < aChordEnd) {
				aRelation.append(GEN_PREFIX);
				aRelation.append(aChordStart);
			} else {
				aRelation.append(GEN_PREFIX);
				aRelation.append(aChordEnd);
				aRelation.append("^-1");
			}
			// if we're at the last chord,
			// append the generator omega
			if (aChordEnd == diagram.getSize() * 2 - 1) {
				aRelation.append("*");
				aRelation.append(GEN_OMEGA);
			}
			isEmpty = false;
		}
		// add the new relation to our list
		inRelList.add(aRelation.toString());
	}
}
\end{lstlisting}

Finally, the relations coming from the twist itself, or the gluing map of the
open book, are listed. The function \code{pmod(\ldots)}, called with two
integers $x$ and $m$, $m$ positive, returns the number $\mathrm{pmod(x,m)}
\equiv x \pmod m$, represented by an integer between $0$ and $m-1$.

\begin{lstlisting}
/**
 * Calculates the relations coming from the mapping of
 * chords to one another.
 * 
 * @param aRelList
 *			the list of relations
 */
private void fillMappingRelations(List<String> aRelList) {
	// the modulus is the number of endpoints on the
	// chord diagram, which is twice the number of chords
	int aMod = 2 * diagram.getSize();
	for (int i = 0; i < 2 * diagram.getSize(); i++) {
		// Let n be the number of chords, l the walk length.
		// The chord c_i from i to j (i<j) is mapped to
		// c_(i+l), which goes from i+l to j+l.
		// The indices are modulo aMod, but we have to keep
		// track of how many times we made a complete turn.
		// Also, chords are labelled by their lower
		// endpoint, so we may have to change the label
		// accordingly and use inverses.
		int j = diagram.getOpposite(i);
		// treat each chord only once
		if (i > j) {
			continue;
		}
		// where i and j are mapped to
		int aNewI = pmod((i + walkLength), aMod);
		int aNewJ = pmod((j + walkLength), aMod);
		// how many times i and j are turned past the
		// basepoint. Note: this is integral division.
		int aRotI = (i + walkLength) / aMod;
		int aRotJ = (j + walkLength) / aMod;
		// where the chord is mapped to
		String aNewChord = GEN_PREFIX + (aNewI < aNewJ ? aNewI : aNewJ + "^-1");
		// the new relation
		StringBuffer aRelation = new StringBuffer();
		// write omega, or omega to some power, or nothing
		if (aRotI >= 2) {
			aRelation.append(GEN_OMEGA);
			aRelation.append("^");
			aRelation.append(aRotI);
			aRelation.append("*");
		} else if (aRotI == 1) {
			aRelation.append(GEN_OMEGA);
			aRelation.append("*");
		}
		// write the chord
		aRelation.append(aNewChord);
		// write the inverse of omega to some power,
		// or nothing
		if (aRotJ >= 1) {
			aRelation.append("*");
			aRelation.append(GEN_OMEGA);
			aRelation.append("^-");
			aRelation.append(aRotJ);
		}
		// all of this is the image of the original chord
		// c_i, so we add the inverse of c_i to complete the
		// relation
		aRelation.append("*");
		aRelation.append(GEN_PREFIX);
		aRelation.append(i);
		aRelation.append("^-1");
		// store the new relation
		aRelList.add(aRelation.toString());
	}
}
\end{lstlisting}

\backmatter

%\listoffigures

%==================================================
\chapter{Glossary of symbols}
% some hack to make those symbols bold as well
\newcommand{\bl}{\pmb{(}}
\newcommand{\br}{\pmb{)}}
\newcommand{\ba}{\pmb{'}}
\newcommand{\bsimeq}{\pmb{\simeq}}
\newcommand{\bcong}{\pmb{\cong}}
\newcommand{\bcaponept}{\pmb{\caponept}}
\newcommand{\bpartial}{\pmb{\partial}}
% some vertical space after each symbol
\newcommand{\ooo}{\\[0.9ex]}

% vertical space to start with
\vspace{0.1ex}
\begin{tabularx}{\textwidth}{>{\boldmath}l X}
%\begin{tabularx}{\textwidth}{l >{\it}X}
$\bsimeq$ &
   isotopic
   \ooo
   
$\bcong$ &
   isomorphic, homeomorphic, or diffeomorphic
   \ooo

$\bcaponept$ &
   intersecting transversely in one point
   \ooo

$\bpartial G$ &
   (when $G$ is a \tat\ graph) boundary of the surface that
   deformation retracts to $G$
   \ooo

$\Bi$ &
   bifoil twist: $\T{E_{2,2},1}$
   \ooo

$E_{n,a}$ &
   elementary \tat\ graph belonging to a chord diagram with $n$ chords
   of length $a$ (see Section \ref{sec:elementary})
   \ooo

$M_{\bl\Sigma,\phi\br}$ or $M_\phi$ &
   open book built from the mapping class
   or diffeomorphism $\phi$ of the surface $\Sigma$ (see Section
   \ref{sec:openbooks})
   \ooo

$\mcg\bl\Sigma\br$ &
   mapping class group of the (compact oriented) surface
   $\Sigma$, fixing its boundary pointwise
   \ooo

$\mcg\ba\bl\Sigma\br$ &
   mapping class group of the (compact oriented) surface
   $\Sigma$, fixing its boundary components setwise; often encountered as the
   \emph{pure mapping class group} of a surface with punctures
   \ooo

$\ord\bl\phi\br$ &
   the order of the mapping class $\phi$
   \ooo

$\Tr$ &
   trefoil twist: $\T{E_{3,3},1}$
   \ooo

$t_\alpha$ &
   Dehn twist along the curve $\alpha$
   \ooo

$\T{G}$ &
   \tat\ twist along the graph $G$
   \ooo
   
$\T{G,l}$ &
   \tat\ twist along the graph $G$ with walk length $l$
   \ooo
   
$\T{G,l_1,\ldots,l_b}$ &
   multi-speed \tat\ twist along the graph $G$ with walk lengths $l_1$ to $l_b$
   \ooo
\end{tabularx}

%For BibLatex:
\printbibliography

\end{document}